\documentclass{article}
\usepackage[%
journal=FoCM,    
lang=american,   
]{ems-journal}

\newif\ifpreprint
\preprinttrue   
\ifpreprint
  \makeatletter
  \renewcommand*\ps@titlepage{%
    \let\@oddhead\@empty \let\@evenhead\@empty
    \let\@oddfoot\@empty \let\@evenfoot\@empty}
  \makeatother
\fi
\preprinttrue

\usepackage{mymacros}
\usepackage{algorithm}
\usepackage{algpseudocode}
\usepackage{multirow}
\usepackage{microtype}
\usepackage{acro}
\usepackage{cleveref}
\usepackage{longtable}
\usepackage{mathtools}
\acsetup{make-links}
\usepackage{pgfplots,tikz}
\usepackage{mytikzconfig}

\numberwithin{equation}{section}
\numberwithin{algorithm}{section}

\renewcommand{\tilde}{\widetilde}
\renewcommand{\hat}{\widehat}
\renewcommand{\Re}{\operatorname{Re}}
\renewcommand{\Im}{\operatorname{Im}}
\newcommand{\cod}{\mathrm{\bullet}}

\RenewAcroTemplate[list]{description}
  {
    \acroheading
    \acropreamble
    \begin{description}
      \acronymsmapF
        {
          \item [ \textcolor{emsblue}{\acrowrite{short}} \acroifT{alt}{/} \acrowrite{alt} ]
            \acrowrite{list}
            \acroifanyT{foreign,extra}{~(}
            \acrowrite{foreign}
            \acroifallT{foreign,extra}{,~}
            \acrowrite{extra}
            \acroifanyT{foreign,extra}{)}
            \acropagefill
            \acropages
              { \acrotranslate{page} \nobreakspace }
              { \acrotranslate{pages} \nobreakspace }
        }
        { \item \AcroRerun }
    \end{description}
  }

\newcommand{\itemeq}[1]{%
  \refstepcounter{equation}%
  $\displaystyle #1$\hfill\textup{(\theequation)}}

\DeclareAcronym{mor}{short=MOR, long=model order reduction}
\DeclareAcronym{rom}{short=ROM, long=reduced--order model}
\DeclareAcronym{fom}{short=FOM, long=full--order model}
\DeclareAcronym{rkhs}{short=RKHS, long=reproducing kernel Hilbert space}
\DeclareAcronym{lti}{short=LTI, long=linear time--invariant}
\DeclareAcronym{irka}{short=IRKA, long=iterative rational Krylov algorithm}
\DeclareAcronym{gmt}{short=GMT, long=generalized Malmquist--Takenaka}
\DeclareAcronym{mt}{short=MT, long=Malmquist--Takenaka}
\DeclareAcronym{bgmt}{short=BGMT, long=biorthogonal--GMT}
\DeclareAcronym{bt}{short=BT, long=balanced truncation}
\DeclareAcronym{cpl}{short=CPL, long=continuous piecewise linear}
\DeclareAcronym{dwr}{short=DWR, long=dual weighted residuals}
\DeclareAcronym{eim}{short=EIM, long=empirical interpolation method}
\DeclareAcronym{fem}{short=FEM, long=finite element method}
\DeclareAcronym{fft}{short=FFT, long=fast Fourier transform}
\DeclareAcronym{rbf}{short=RBF, long=radial basis function}
\DeclareAcronym{siso}{short=SISO, long=single--input single--output}
\DeclareAcronym{svm}{short=SVM, long=support vector machine}
\DeclareAcronym{pde}{short=PDE, long=partial differential equation}
\DeclareAcronym{mse}{short=MSE, long=mean squared error}

\theoremstyle{plain}
\newtheorem{theorem}{Theorem}[section]
\newtheorem{lemma}[theorem]{Lemma}

\newtheorem{probstat}[theorem]{Problem statement}

\newtheorem{corollary}[theorem]{Corollary}
\newtheorem{assumption}[theorem]{Assumption}
\theoremstyle{definition}
\newtheorem{definition}[theorem]{Definition}
\newtheorem{remark}[theorem]{Remark}
\newtheorem{example}[theorem]{Example}

\begin{document}

\title{Norm--optimal model order reduction in reproducing kernel Hilbert spaces}
\titlemark{Norm--optimal MOR in RKHSs}


%

\emsauthor{1}{
	\givenname{Tamás}
	\surname{Dózsa}
	\orcid{0000-0003-0919-4385}}{T.~Dózsa}
\emsauthor{2}{
	\givenname{Andrea}
	\surname{Angino}
	\orcid{0009-0000-8525-375X}}{A.~Angino}
\emsauthor{3}{
	\givenname{Rolf}
	\surname{Krause}
	\orcid{0000-0001-5408-5271}}{R.~Krause}
\emsauthor{4}{
	\givenname{Matthias}
	\surname{Voigt}
	\orcid{0000-0001-8491-1861}}{M.~Voigt}


\Emsaffil{1}{
    \pretext{On leave from:}
	\department{Systems and Control Laboratory}
	\organisation{HUN--REN Institute for Computer Science and Control}
	\address{Lágymányosi Utca 11}
	\zip{1111}
	\city{Budapest}
	\country{Hungary}
    \affemail{dozsa.tamas@hsztaki.hu}}
\Emsaffil{2}{
	\department{Faculty of Mathematics and Computer Science}
	\organisation{FernUni Schweiz/UniDistance Suisse}
	\address{Schinerstrasse~18}
	\zip{3900}
	\city{Brig}
	\country{Switzerland}
    \affemail{andrea.angino@unidistance.ch}}
    
\Emsaffil{3}{
	\department{Computer, Electrical and Mathematical Sciences and Engineering Division}
	\organisation{King Abdullah University of Science and Technology}
	\zip{4700}
	\city{Thuwal}
	\country{Kingdom of Saudi Arabia}
	\affemail{rolf.krause@kaust.edu.sa}
    }

\Emsaffil{4}{
	\department{Faculty of Mathematics and Computer Science}
	\organisation{FernUni Schweiz/UniDistance Suisse}
	\address{Schinerstrasse~18}
	\zip{3900}
	\city{Brig}
	\country{Switzerland}
    \affemail{matthias.voigt@fernuni.ch}
    }

\classification[41A05, 42B35, 46E22, 65D15]{41A65}

\keywords{reproducing kernel Hilbert spaces, reproducing kernels, Malmquist--Takenaka systems, model order reduction, support vector machines, neural networks}

\begin{abstract}
We consider the problem of model order reduction (MOR) in reproducing kernel Hilbert spaces (RKHSs). We prove generalizations of well--known necessary interpolatory conditions for norm--optimal reduced--order models, originally stated for rational functions in $H_2$ Hardy--Hilbert spaces. We construct an alternative to the iterative rational Krylov (IRKA) algorithm for RKHSs with differentiable kernels to obtain norm--optimal reduced--order models. Then, we introduce generalized Malmquist--Takenaka (GMT) functions that can be used to define orthonormal bases in an arbitrary RKHS. Using GMT functions, we propose unitary transformations between separable RKHSs. This allows us to extend our optimality conditions and subsequent algorithms to MOR problems posed in RKHSs with non--differentiable kernels. The potential and versatility of the proposed MOR framework is demonstrated in several numerical experiments. We perform norm--optimal MOR on various machine learning models such as support vector machines and deep ReLU networks. In addition, we show how the proposed MOR framework can be used to solve seemingly unrelated problems. In particular, we provide an example for fitting high degree polynomials to data and also obtain a FEM--like method for solving partial differential equations (PDEs) with an optimal kernel basis.
\end{abstract}

\maketitle

\setcounter{tocdepth}{2}
\tableofcontents


\section{Introduction}

The study of \ac{mor} deals with the simplification of mathematical models referred to as \acp{fom}. In \ac{mor}, \acp{fom} and their simplified approximants usually have prescribed structure. The obtained approximant is called the \ac{rom}. Traditionally, \ac{mor} techniques have been developed for reducing the complexity of dynamical systems~\cite{gugercin_survey_2004, gugercin_mathcalh_2_2008, borghi_mathcal_2024, baur_model_2014, meier_approximation_1967}. Since many control engineering applications must operate in real time under limited hardware capacity, control schemes based on \acp{rom} have proven particularly useful~\cite{obinata_model_2012}. Well--known techniques exist for the reduction of models that describe \ac{lti} systems~\cite{gugercin_survey_2004, gugercin_mathcalh_2_2008, schroder_balanced_2023, soumelidis_hyperbolic_2017, meier_approximation_1967}.

In order to address the \ac{mor} of \ac{lti} systems, several well--known methods have been developed that preserve the structure of the full order model, but provide theoretical guarantees on the performance of the \ac{rom}. One such well--known method is \ac{bt}~\cite{gugercin_survey_2004}, which operates on the state--space representation of \ac{lti} systems. It relies on a balanced realization of the system, which can be obtained by solving Lyapunov equations. Then, the least controllable and observable states in the balanced system are disregarded to obtain the \ac{rom}. One important benefit of the \ac{bt} method is the theoretical guarantee it provides on the $H_{\infty}$ error of the \ac{rom}.

Another popular approach to the reduction of \ac{lti} models is the iterative rational Krylov algorithm (\ac{irka}) introduced in~\cite{gugercin_mathcalh_2_2008}. \ac{irka} is an iterative method that computes the \ac{rom} and guarantees that the result satisfies necessary interpolatory (or, equivalently, orthogonality--based) conditions for optimal approximation in the $H_2$ norm. In recent works~\cite{borghi_mathcal_2024, borghi_data-driven_2025}, the orthogonality--based optimality conditions proposed in~\cite{gugercin_mathcalh_2_2008} for $H_2(\C^{+})$ have been adapted to $H_2$ spaces corresponding to more general domains. For example, in~\cite{borghi_mathcal_2024, borghi_data-driven_2025}, interpolatory necessary conditions and an \ac{irka} analog have been developed for $H_2$ spaces defined over simply connected sets. We also mention~\cite{benner_interpolation-based_2012}, where $H_2$--norm--optimal \acp{rom} are obtained for bilinear systems. Again, the construction uses a variation of \ac{irka} and satisfies an appropriate analog of the interpolatory conditions necessary for optimal approximation in the $H_2$ norm. 

In this study, our objective is to further develop \ac{mor} methods along several directions of generalization. First, we exploit the fact that the $H_2(\C^{+})$ Hardy--Hilbert space is an \ac{rkhs} containing functions $\C^{+} \to \C$, where $\C^{+}$ denotes the open right complex half--plane. We show that the necessary interpolatory optimality conditions proposed in~\cite{gugercin_mathcalh_2_2008, meier_approximation_1967} for $H_2$--norm--optimal \ac{mor} can be generalized to an \ac{rkhs} setting. The proposed optimality conditions are then used to construct iterative schemes to obtain norm--optimal \acp{rom}. We note that our optimality conditions and subsequent algorithms can only be applied directly in \acp{rkhs} whose kernel is Wirtinger differentiable with respect to both variables. We refer to \acp{rkhs} adhering to this condition as admissible. Several relevant large models, such as deep ReLU neural networks~\cite{shilton_reproducing_2025, plonka_spline_2023, daubechies_nonlinear_2019}, cannot be represented in admissible \acp{rkhs}.

To overcome this issue, we introduce the \ac{rkhs} variants of \ac{mt} functions~\cite{mi_fast_2012, soumelidis_hyperbolic_2017} and refer to them as \ac{gmt} functions. The classical \ac{mt} functions form an orthonormal system in the Hardy--Hilbert space $H_2(\D)$, where $\D\coloneqq\{ z \in \C \ : \ |z| < 1 \}$ denotes the interior of the complex unit disk. We prove that the proposed \ac{rkhs} formulation of the \ac{mt} system retains some important properties of its rational counterpart. We mention that drawing an analogy between special complex analytic functions and general \acp{rkhs} is not without precedent. In~\cite{felder_analogues_2022}, for example, an \ac{rkhs} formulation of finite Blaschke products (special inner functions of $H_2(\mathbb{D})$) is proposed. Recently, in~\cite{qian_kernel_2026}, the rational \ac{mt} system has been defined for a family of \acp{rkhs} known as boundary weighted Hardy spaces. We note that our definition of \ac{gmt} functions contains these and other previous generalizations such as~\cite{muller_newton_2009} as special cases. Using the introduced \ac{gmt} bases and the Hilbert space isomorphism theorem, we define unitary transformations between any two separable \acp{rkhs}. These bijections allow us to apply the proposed optimality conditions and \ac{mor} algorithms by transforming the problem to an admissible \ac{rkhs} such as $H_2(\D)$. 

Finally, we present a collection of novel applications of norm--optimal MOR. The purpose of these examples is to demonstrate the versatility of the proposed \ac{rkhs}--\ac{mor} framework. To this end, we include examples that perform \ac{rkhs}--norm--optimal \ac{mor} on modern machine learning models. In particular, we consider the \ac{mor} of support vector machines and deep ReLU neural networks. In addition, we consider two somewhat surprising applications. These illustrate that, under some circumstances, our framework can be used to obtain norm--optimal approximations of \acp{fom} that are not fully known. In particular, we first provide an academic example of fitting high--degree polynomials to data in an \ac{rkhs}--norm--optimal manner. Then, we demonstrate how our \ac{rkhs} approximation framework can be used to obtain kernel--basis weak solutions to certain \acp{pde}. These four applications range from an illustrative academic example to fully developed case studies, with the \ac{pde} application in particular presented as a proof of concept for a genuinely new use of the framework, rather than a completed method.

Below we list the key novelties of our study.
\begin{enumerate}
    \item \textbf{Optimality conditions in general \acp{rkhs}.} We prove necessary conditions for \ac{rkhs} norm optimality of the reduced--order model in admissible \acp{rkhs}. These results, see \Cref{sec:orthCond}, unify and generalize conditions known for $H_2(\C^+)$ approximation~\cite{meier_approximation_1967, gugercin_mathcalh_2_2008}.

    \item \textbf{Iterative \ac{rkhs} \ac{mor} algorithms.} We develop iterative schemes that, in admissible \acp{rkhs}, produce \acp{rom} satisfying the introduced optimality conditions.

    \item \textbf{Generalized Malmquist--Takenaka (\ac{gmt}) functions.} We introduce \ac{gmt} functions for \acp{rkhs}, showing that existing variants, including the classical $H_2(\mathbb{D})$ function system and previous generalizations~\cite{qian_kernel_2026, muller_newton_2009}, arise as special cases. We show interpolation properties, zero locations, factorizations, and several other important properties of \ac{gmt} functions.

   \item \textbf{The \ac{gmt}--transformation.} We introduce the \ac{gmt} transformation: a unitary map between two arbitrary separable \acp{rkhs}. We use the \ac{gmt} transformation to solve \ac{mor} problems posed in non--admissible \acp{rkhs}.

    \item \textbf{Diverse applications.} We provide a diverse collection of example applications. We perform \ac{rkhs}--norm--optimal \ac{mor} on \acp{svm} and neural networks. We develop an adaptive \ac{pde} discretization scheme based on the proposed approximation framework.
\end{enumerate}

The paper is organized as follows. In \Cref{sec:preRKHS}, we revisit the definition of \acp{rkhs} and provide a formal problem statement. In \Cref{sec:pre}, we review related mathematical concepts and prior results, focusing on properties of \acp{rkhs} and $H_2$--optimal \ac{mor}. Our novel optimality conditions are introduced in \Cref{sec:orthCond}. \ac{gmt} functions, their properties and the \ac{gmt} transformation are introduced in \Cref{sec:gmt}. \Cref{sec:irka} contains our proposed \ac{mor} algorithms, while in \Cref{sec:exp}, we demonstrate their effectiveness through four numerical examples. Finally, in \Cref{sec:conc}, we summarize our results and discuss future research directions.

\section{Problem statement}
\label{sec:preRKHS}

In this work, we are interested in the \ac{mor} of \acp{fom} belonging to \acp{rkhs}. In order to provide a formal problem statement, we recall the definition of \acp{rkhs} and reproducing kernels.
\begin{definition}[Reproducing kernel Hilbert space]
    \label{def:RKHS}
    Let $\Omega \subseteq \C$ and $\origspace$ be a Hilbert space of functions $\Omega \to \C$. If the point evaluation functionals
    \begin{equation*}
        \cT_{\lambda} : \origspace \to \C, \quad \cT_{\lambda} (\bF) \coloneqq \bF(\lambda) \quad (\lambda \in \Omega)
    \end{equation*}
    are bounded (continuous) for all $\lambda \in \Omega$, then $\origspace$ is called a \emph{reproducing kernel Hilbert space}.
\end{definition}

By Riesz's representation theorem (see, e.g.,~\cite{conway_course_2007}), any bounded linear operator $\cK$ acting on a Hilbert space $\origspace$ can be expressed as
$$
    \cK\left(\bF\right) = \left\langle \bF, h_{\cK} \right\rangle_{\origspace} \quad \left(\bF \in \origspace\right)
$$
for some unique $h_{\cK} \in \origspace$. Hence, if $\origspace$ is an \ac{rkhs}, for all $\lambda \in \Omega$, there is a unique $\xi_{\lambda} \in \origspace$ such that
$$
    \cT_{\lambda} (\bF ) = \bF(\lambda) = \left\langle \bF, \xi_{\lambda} \right\rangle_{\origspace} \quad \big(\bF \in \origspace \big).
$$
This gives rise to the concept of a reproducing kernel.
\begin{definition}[Reproducing kernel, kernel slice]
    \label{def:reprKer}
    Let $\origspace$ be an \ac{rkhs} of functions $\C \supseteq \Omega \to \C$. The map
    $$
        \xi : \Omega \times \Omega \to \C,
    $$
    that satisfies
    $$
        \cT_{\lambda} ( \bF ) = \bF(\lambda) = \left\langle \bF, \xi(\cdot, \lambda) \right\rangle_{\origspace} = \left\langle \bF, \xi_{\lambda} \right\rangle_{\origspace} \quad (\bF \in \origspace)
    $$
    is called the \emph{reproducing kernel} associated with $\origspace$. For some fixed $\lambda \in \Omega$, the map
    $$
       \xi_{\lambda} \coloneqq   \xi(\cdot, \lambda) \in \origspace
    $$
    is called a \emph{kernel slice}.
\end{definition}
It can verified that a reproducing kernel is Hermitian (i.e., $\xi(x,\lambda) = \overline{\xi(\lambda,x)}$ for all $x,\lambda \in  \Omega$) and positive semi--definite (i.e., $\sum_{i,j=0}^{n-1} c_i\overline{c_j} \xi (x_i,x_j) \ge 0$ for all $x_0,\ldots,x_{n-1}$,  $c_0,\ldots,c_{n-1} \in \C$ and $n \in \N$).
By the Moore--Aronszajn theorem~\cite{aronszajn_theory_1950}, an \ac{rkhs} and its associated reproducing kernel are in one--to--one correspondence. \acp{rkhs} and kernels have many useful properties, which have been exploited in various data--driven modeling applications such as machine learning (see, e.g.,~\cite{hearst_support_1998, vapnik_complete_2020, dozsa_variable_2024}). In \Cref{subsec:rkhs}, we briefly review some of these properties necessary for the methods proposed in this study. For a thorough review of \acp{rkhs}, we recommend~\cite{paulsen_introduction_2016}. 

We are now ready to formally state the problem of \ac{mor} in an \ac{rkhs}. For this, we first define generalized model spaces.
\begin{definition}[Generalized model space, dimension]
\label{def:gmodspace}
    Let $\Omega \subseteq \C$ and $\origspace$ be an \ac{rkhs} of functions $\Omega \to \C$ with reproducing kernel $\xi$. The generalized \emph{generalized model space} in $\origspace$ of \emph{dimension} $n \in \N$ is defined by
    \begin{equation*}
         \fomspaceo \coloneqq \left\{ \bG \in \myspan\left\{ \xi_{\lambda_0},\ldots,\xi_{\lambda_{n-1}}\right\} \ : \ (\lambda_0,\ldots,\lambda_{n-1}) \in \Omeganeq{n} \right\},
    \end{equation*}
    where we define
    \begin{equation*}
     \Omeganeq{n} \coloneqq \left\{ (\lambda_0,\ldots,\lambda_{n-1}) \in \Omega^n \ : \ \lambda_j \neq \lambda_k \ \text{for} \ j \neq k \right\}.
    \end{equation*}
\end{definition}
We refer to the functions in generalized model spaces as \emph{models} throughout the paper. Our objective in this study is to approximate models with functions that adhere to a special structure, e.g., simpler models. To obtain a meaningful \ac{mor} problem, we need to define the \emph{order}, which serves as a measure of model complexity.
\begin{definition}[Model order]
    \label{def:modelord}
   Let $n \in \N$ be the order of the model $\bG$, if $\bG \in \fomspaceo \subseteq \origspace$, but $\bG \not \in \romspaceo \subseteq \origspace$ for any $r < n$.
\end{definition}

In Sections~\ref{sec:gmt} and~\ref{sec:exp}, we encounter problems where we need to transform the \ac{fom} into a companion admissible \ac{rkhs}. This will be done via a unitary transformation $\cF$ acting between the original and companion \acp{rkhs}. For such cases, the \ac{rom} belongs to a weaker notion of a model space (in the original \ac{rkhs}), specified by the following definition.
\begin{definition}[Weak model space, dimension, weak model, order]
    \label{def:weak}
    Let $\Omega \subseteq \C$ and $\origspace$ be an \ac{rkhs} of functions $\Omega \to \C$ with reproducing kernel $\xi$. Let furthermore $\hat{\Omega} \subseteq \C$ and $\ambspace$ be an \ac{rkhs} of functions $\hat{\Omega} \to \C$ with kernel $\hat{\xi}$. The  \emph{weak model space} in $\origspace$ (with respect to $\ambspace$) of \emph{dimension} $n \in \N$  is 
    \begin{equation*}
        \wfomspaceo \coloneqq \left\{\bG = \cF^{-1}\big[ \hat{\bG} \big] \ : \ \hat{\bG} \in \fomspaceg \subseteq \ambspace, \cF : \origspace \to \ambspace \ \textrm{is unitary} \right\}.
    \end{equation*}%
    If $\bG \in \wfomspaceo$, but $\bG \not \in \wromspaceo$ for any $r < n$, we say that $\bG$ is a $\origspace$--\emph{weak model} of \emph{order} $n$.

\end{definition}

\begin{remark}
    \begin{enumerate}[label=(\alph*)]
        \item
  We have $\bG \in \wfomspaceo$ if it can be written as a linear combination of kernel slice preimages
    under a given unitary transformation $\cF : \origspace \to \ambspace$.
    \item For an $\origspace$--weak model $\bG \in \wfomspaceo$, there is no guarantee that there exist $n$ points $(\lambda_0, \ldots, \lambda_{n-1}) \in \Omeganeq{n}$ and $\res_0,\ldots,\res_{n-1} \in \C$ such that
        \begin{equation*}
            \bG = \sum_{k=0}^{n-1} \res_k \xi_{\lambda_k},
        \end{equation*}
        where $\xi$ is the reproducing kernel of $\origspace$. On the other hand, by \Cref{def:weak}, $\bG \in \wfomspaceo$ belongs to an $n$--dimensional subspace of $\origspace$ denoted by $W_n$. Here, $W_n$ is an $n$--dimensional \ac{rkhs} in its own right (see \Cref{thm:findimRKHS} in \Cref{sec:pre}). Consequently, $\bG$ can be written as
        \begin{equation*}
            \bG = \sum_{k=0}^{n-1}  \wres_k \xi^{W_n}_{\lambda_k} 
        \end{equation*}
        for some $(\lambda_0, \ldots, \lambda_{n-1}) \in \Omeganeq{n}$ and $\wres_0, \ldots, \wres_{n-1}\in \C$, where $\xi^{W_n}$ denotes the kernel of the $n$--dimensional \ac{rkhs} $W_n \subseteq \origspace$.
    \end{enumerate}
\end{remark}

We define norm--optimal \ac{mor} in an \ac{rkhs} $\origspace$ as follows.
\begin{probstat}[Norm--optimal \ac{mor} in \acp{rkhs}]
    \label{def:Objective}
    Let $\Omega \subseteq \C$ and $\origspace$ be an \ac{rkhs} of functions $\Omega \to \C$ with reproducing kernel $\xi$. Consider a \ac{fom} $\bG \in \fomspaceo$ for some $n \in \N$. Given $r < n$, the objective of norm--optimal \ac{mor} is to obtain $\bG_r \in \romspaceo$ that minimizes the functional $\rkhsError{\xi}{r} : \romspaceo \to [0, \infty)$ defined as
     \begin{equation}
        \label{eq:Objective}
        \rkhsError{\xi}{r} \left( \cdot \right) = \left\| \bG - \cdot \right\|_{\origspace}.
    \end{equation}
    In the case of weak norm--optimal \ac{rkhs}--\ac{mor} problems, we extend the search space (the domain of $\rkhsError{\xi}{r}$) to \acp{rom} satisfying $\bG_r \in \wromspaceo$.
\end{probstat}
\begin{definition}[Local minimum points of $\rkhsError{\xi}{r}$]
    \label{def:locmin}
    Let $\rkhsError{\xi}{r}$ be an objective functional as defined in \Cref{def:Objective}. The \ac{rom} $\bG_r \in \romspaceo$ is a local minimum point of $\rkhsError{\xi}{r}$ if there exists $\delta > 0$ such that, for any $\tilde{\bG}_r \in \romspaceo$ that satisfies $\big\| \bG_r - \tilde{\bG}_r \big\|_{\origspace} < \delta$, we have
    \begin{equation*}
        \rkhsError{\xi}{r}\left(\bG_r\right) \leq \rkhsError{\xi}{r}\big(\tilde{\bG}_r\big).
    \end{equation*}
\end{definition}

In this work, we seek \acp{rom} which satisfy necessary optimality conditions introduced in \Cref{sec:orthCond}. Since they are not sufficient in general, we cannot ensure to find a local minimum in any case. But we conjecture that local maxima are repellent, so using a numerical method based on the necessary optimality conditions is a feasible route to follow. This is similar to what classical \ac{mor} methods guarantee for the reduction of linear dynamical systems~\cite{gugercin_mathcalh_2_2008, mlinaric_irka_2025, beattie_ch7_2017}.

\begin{remark}
    \begin{enumerate}[label=(\alph*)]
        \item The local minimizer definition given in \Cref{def:locmin} is similar to the definitions given for specific \acp{rkhs} in~\cite{gugercin_mathcalh_2_2008}, where \ac{mor} problems for \ac{lti} systems are considered.
        \item \Cref{thm:ExistLocalOpt} establishes the well--posedness of the minimization problem in~\Cref{def:Objective}.
        
        \item The logic of the \ac{mor} methods presented in our study is illustrated in the diagram in \Cref{fig:pipeline}. Given a \ac{fom} $\bG \in \fomspaceo$, we introduce \ac{mor} strategies, where a \ac{rom} is constructed directly in the \ac{rkhs} $\origspace$ associated with the kernel $\xi$, provided that this space admits a suitable structure. This approach is illustrated on the left side of the diagram.

Alternatively, the right branch defines another reduction procedure: using a family of unitary transformations, termed the \ac{gmt} transform (\Cref{def:GMTtransform}), the problem is mapped to a companion \ac{rkhs} $\ambspace$ with kernel $\hat{\xi}$. The companion \ac{rkhs} can be chosen to have favorable properties from the standpoint of the \ac{mor} problem. A \ac{rom} is then computed in $\ambspace$ and mapped back to $\origspace$. 
        \item This general \ac{mor} strategy, following \Cref{fig:pipeline}, admits a wide range of potential applications (see \Cref{sec:exp}).
    \end{enumerate}
\end{remark}

\begin{figure}[h!]
    \centering
        \begin{tikzpicture}[node distance = 1.5cm, auto]
    
        \node (H) {\textbf{Original \ac{rkhs}} $\origspace$ };
    
        \node[right=5cm of H, align=center] (G) {  
        \textbf{Companion \ac{rkhs}} $\ambspace$};
        
        \node [block, below of = H, node distance=1cm] (FOM) { \ac{fom} $\bG \in \fomspaceo \subseteq \origspace$};
        
        \node [block, right of = FOM, node distance=8cm] (FOMG) {\ac{fom} $\hat{\bG} \in \frak{G}_{\hat{\xi}}^n \subseteq \ambspace$};

        \node [block, below of = FOMG, node distance=5cm] (ROMG) {\ac{rom} $\widehat{\bG}_r \in \frak{G}_{\hat{\xi}}^n \subseteq \ambspace$};
        
        \node [block, below of = FOM, node distance=5cm] (ROM) {\ac{rom} $\bG_r \in \romspaceo$ \\ (or $\bG_r \in \wromspaceo$)};
    
        \draw [->, line width = 3pt] (FOM) -- 
        node[align=center] {\ac{gmt} transformation \\ (Definition~\ref{def:GMTtransform})} 
        (FOMG);

        \draw [->, line width = 3pt] (FOMG) -- 
        node[sloped, align=center] {\ac{mor} in companion space \\ (Algorithm~\ref{alg:cdcRKHS})} 
        (ROMG);
    
        \draw [->, line width = 3pt] (ROMG) -- 
        node[align=center] {Inverse \ac{gmt} \\ transformation } 
        (ROM);
    
        \draw [->, sloped, line width = 3pt] (FOM) -- 
        node[align=center] {Direct methods \\ (Algorithm~\ref{alg:irkaRKHS})} 
        (ROM);
    \end{tikzpicture}
    \caption{The proposed \ac{mor} pipeline for \acp{rkhs}.}
    \label{fig:pipeline}
\end{figure}
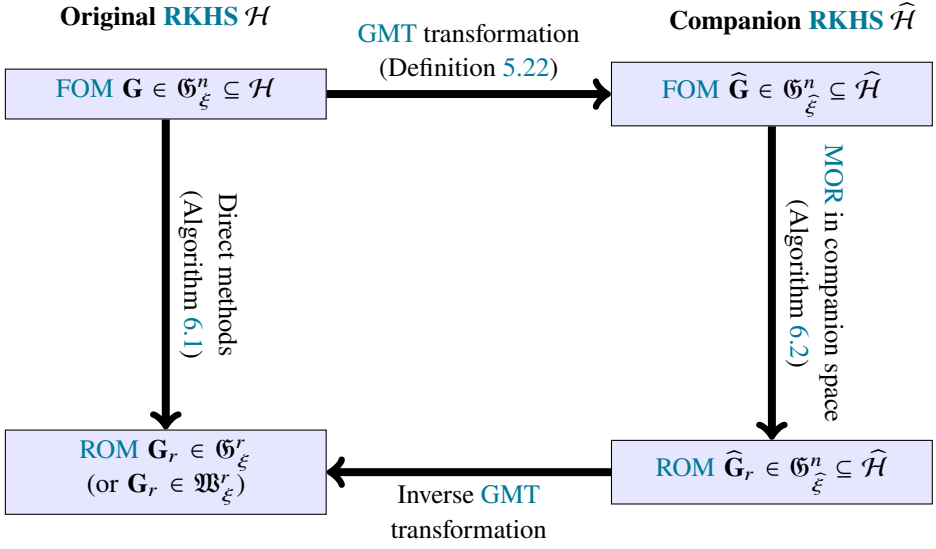

\section{Preliminaries}
\label{sec:pre}

In this section, we first review some important \ac{rkhs}--related concepts that will be used later. In addition, we motivate the research by briefly revisiting existing \ac{mor} approaches for \ac{lti} systems.

\subsection{Properties of RKHSs}
\label{subsec:rkhs} 
  \Cref{def:reprKer} gives the following property.
\begin{lemma}[Reproducing inner products]
    \label{thm:reprInnerProd}
    Let $\origspace$ be an \ac{rkhs} with kernel $\xi$. Then,
    \begin{equation*}
        \xi(x, \lambda) = {\langle \xi_{x}, \xi_{\lambda} \rangle}_{\origspace} \quad (x, \lambda \in \Omega).
    \end{equation*}
\end{lemma}
We will use the following theorem, which draws a connection between the derivatives of kernel slices and $\origspace$ inner products. Several notions of derivatives will be used throughout the paper. To ease the discussion, we first introduce some notation. First, consider a kernel slice $\xi_{\lambda} \in \origspace$. Whenever existence conditions are satisfied, the notations $\funder{\xi_{\lambda}}{\lambda}$ and $\funder{\xi_{\lambda}}{\overline{\lambda}}$ refer to the Wirtinger derivative functions defined by
\begin{multline}
    \label{eq:FrechDer}
    \lim_{\Delta \lambda \to 0} \left\| \frac{\xi_{\lambda + \Delta \lambda}(\cdot) - \xi_{\lambda}(\cdot)}{\Delta \lambda} - (\funder{\xi_{\lambda}}{\lambda}) (\cdot)\right\|_{\origspace} = 0, \quad \\ \lim_{\Delta \lambda \to 0} \left\| \frac{\xi_{\lambda + \Delta \lambda}(\cdot) - \xi_{\lambda}(\cdot)}{\overline{\Delta \lambda}} - (\funder{\xi_{\lambda}}{\overline{\lambda}})(\cdot)\right\|_{\origspace} = 0.
\end{multline}
Since $\origspace$ is an \ac{rkhs}, convergence in norm implies pointwise convergence~\cite{paulsen_introduction_2016} and~\eqref{eq:FrechDer} matches the Wirtinger derivatives pointwise, i.e., we have $\frac{ \partial \xi_{\lambda}}{\partial \lambda}(z) = \funder{\xi_{\lambda}}{\lambda}(z)$ and $\frac{ \partial \xi_{\lambda}}{\partial \overline{\lambda}}(z) = \funder{\xi_{\lambda}}{\overline{\lambda}}(z)$ for all $z \in \Omega$. Let $\blambda = (\lambda_0,\ldots,\lambda_{n-1}) \in \Omeganeq{n}$ and consider the function $\Omeganeq{n} \to \origspace$  
\begin{equation*}
 \cU(\blambda) \coloneqq \sum_{k=0}^{n-1} \res_k \xi_{\lambda_k}(\cdot).
\end{equation*}
For a fixed $\blambda \in \Omeganeq{n}$, let $\bG \coloneqq \cU(\blambda) \in \fomspaceo$. Then we can define the derivative of  $\bG$ with respect to the parameter $\lambda_k$ of a kernel slice as
\begin{equation*}
    \funder{\bG}{\lambda_k} \coloneqq \left[\nabla \cU (\blambda) \right]_k \coloneqq \res_k (\funder{\xi_{\lambda}}{\lambda}){\big|}_{\lambda = \lambda_k} \quad \textrm{and} \quad \funder{\bG}{\olambda_k} \coloneqq \big[\overline{\nabla} \cU (\blambda) \big]_k \coloneqq  \res_k (\funder{\xi_{\lambda}}{\olambda}){\big|}_{\lambda = \lambda_k}.
\end{equation*}
Moreover, for $\bG \in \fomspace$, we define
\begin{equation*}
    \funder{\bG}{\lambda} \coloneqq \sum_{k=0}^{n-1} \funder{\bG}{\lambda_k}  \quad \textrm{and} \quad  \funder{\bG}{\olambda} \coloneqq \sum_{k=0}^{n-1} \funder{\bG}{\olambda_k}  .
\end{equation*}

For higher--order derivatives, we use similar notation. In particular, to denote the $m$th order Wirtinger derivative (function) of some $\bG \in \fomspaceo$ for $m \ge 2$, we use the notations $\funder[m]{\bG}{\lambda_k} \coloneqq \funder{\big(\funder[m-1]{\bG}{\lambda_k}\big)}{\lambda_k}$ and  $\funder[m]{\bG}{\lambda} \coloneqq \funder{\left(\funder[m-1]{\bG}{\lambda}\right)}{\lambda}$.

In several places, we also need to denote pointwise Wirtinger derivatives of $\bG \in \origspace$. Since we will only need to express $\frac{\partial \bG}{\partial z}(z)$, we use the notation
\begin{equation*}
    \frac{\partial \bG}{\partial z}(z) = \bG'(z).
\end{equation*}
Finally, note that for higher--order derivatives, for $m \ge 2$ we use the notation
\begin{equation*}
    \frac{\partial^m \bG}{\partial z^m}(z) = \bG^{(m)}(z).
\end{equation*}
Recall the following connection between derivatives and reproducing kernels. Since this theorem is frequently relied upon throughout the study, we provide a short proof for clarity.
\begin{theorem}[Derivatives of kernels~\cite{zhou_derivative_2008}]
    \label{thm:derivKer}
    Let $\origspace$ be an \ac{rkhs} of functions $\C \supseteq \Omega \to \C$ with kernel $\xi$. Let $(\lambda \mapsto \xi_{\lambda}) \in C^{1}(\Omega, \origspace)$ such that 
    \begin{equation*}
    \funder{\xi_\lambda}{\olambda} \in \origspace \quad (\lambda \in \Omega).     
    \end{equation*}
    Then,
    \begin{equation*}
        \left\langle \bG, \funder{\xi_{\lambda}}{\olambda} \right\rangle_{\origspace} = \bG'(\lambda).
    \end{equation*}
\end{theorem}
\begin{proof}
We have
    \begin{multline*}
        \left\langle \bG, \funder{\xi_{\lambda}}{\olambda} \right\rangle_{\origspace} = \left\langle \bG, \lim_{\Delta \lambda \to 0} \frac{\xi_{\lambda + \Delta \lambda} - \xi_{\lambda} }{\overline{\Delta \lambda}}  \right\rangle_{\origspace} = \lim_{\Delta \lambda \to 0} \frac{1}{\Delta \lambda} \left\langle \bG, \xi_{\lambda + \Delta \lambda} - \xi_{\lambda} \right\rangle_{\origspace}  \\
        = \lim_{\Delta \lambda \to 0} \frac{\left\langle \bG, \xi_{\lambda + \Delta \lambda} \right\rangle_{\origspace} - \left\langle \bG, \xi_{\lambda} \right\rangle_{\origspace}}{\Delta \lambda} = \lim_{\Delta \lambda \to 0} \frac{\bG(\lambda + \Delta \lambda) - \bG(\lambda)}{\Delta \lambda}  = \bG'(\lambda).
    \end{multline*}
\end{proof}
\Cref{thm:derivKer} extends to derivatives of arbitrary order by induction.
\begin{corollary}[Higher--order derivative reproduction]
    \label{thm:derivKerOrder}
    Let $\origspace$ be an \ac{rkhs} of functions $\C \supseteq \Omega \to \C$ with kernel $\xi$. Let $m \in \N$ and suppose $(\lambda \mapsto \xi_{\lambda}) \in C^{m}(\Omega, \origspace)$ such that $\funder[i]{\xi_\lambda}{\olambda} \in \origspace$ for $i=0,\ldots,m$ and all $\lambda \in \Omega$. Then, for any $\bG \in \origspace$ and $i = 0,\ldots,m$,
    \begin{equation*}
        \left\langle \bG, \funder[i]{\xi_{\lambda}}{\olambda} \right\rangle_{\origspace} = \bG^{(i)}(\lambda).
    \end{equation*}
\end{corollary}
\begin{proof}
    The case $m=0$ is the reproducing property and $m=1$ is \Cref{thm:derivKer}. Let $m \ge 2$ and suppose the statement holds for some $m > i \ge 1$. Then,
    \begin{align*}
        \left\langle \bG, \funder[i+1]{\xi_{\lambda}}{\olambda} \right\rangle_{\origspace}  &= \left\langle \bG, \lim_{\Delta \lambda \to 0} \frac{\funder[i]{\xi_{\lambda + \Delta \lambda}}{\olambda} - \funder[i]{\xi_{\lambda}}{\olambda}}{\overline{\Delta \lambda}}  \right\rangle_{\origspace} \\&= \lim_{\Delta \lambda \to 0} \frac{1}{\Delta \lambda} \left\langle \bG, \funder[i]{\xi_{\lambda + \Delta \lambda}}{\olambda} - \funder[i]{\xi_{\lambda}}{\olambda} \right\rangle_{\origspace}  \\
        &= \lim_{\Delta \lambda \to 0} \frac{\left\langle \bG, \funder[i]{\xi_{\lambda + \Delta \lambda}}{\olambda} \right\rangle_{\origspace} - \left\langle \bG, \funder[i]{\xi_{\lambda}}{\olambda} \right\rangle_{\origspace}}{\Delta \lambda} \\ &= \lim_{\Delta \lambda \to 0} \frac{\bG^{(i)}(\lambda + \Delta \lambda) - \bG^{(i)}(\lambda)}{\Delta \lambda}  = \bG^{(i+1)}(\lambda).
    \end{align*}
\end{proof}
One notable property of \acp{rkhs} is that kernel slices are dense in $\origspace$.
\begin{theorem}[Density of kernel slices (e.g.,~\cite{dozsa_adaptive_2026})]
    \label{thm:kerdens}
    Let $\origspace$ be an \ac{rkhs} with associated reproducing kernel $\xi$. Then,
    $$
        \origspace = \mathrm{cl} \big( \myspan \{\xi_{\lambda} \ : \ \lambda \in \Omega \} \big),
    $$
    where $\mathrm{cl}(\cdot)$ denotes the closure of a set with respect to the norm induced by $\langle \cdot, \cdot \rangle_{\origspace}$. Furthermore, if $\origspace$ is separable, then a countable subset of $\Omega$ exists such that the kernel slices defined by it are dense in $\origspace$. 
\end{theorem}
The following characterization of kernels is useful for our discussion.
\begin{definition}[Feature maps, feature spaces~\cite{jorgensen_new_2025}]
    \label{def:features}
    Let $\origspace$ be an \ac{rkhs} of functions $\Omega \to \C$ with associated kernel $\xi$. A Hilbert space $\mathfrak{S}$ is called a feature space, if there exists a map $\Phi : \Omega \to \mathfrak{S}$ that satisfies
    $$
        \xi(x, \lambda) = \langle \Phi(x), \Phi(\lambda) \rangle_{\mathfrak{S}}.
    $$
    In this case, $\Phi$ is referred to as a feature map. Finally, denote the set of all feature map--feature space pairs by
    $$
        \mathfrak{S}_{\xi} \coloneqq \left\{ (\Phi, \mathfrak{S}) \ : \ \xi(x,\lambda) = \langle \Phi(x), \Phi(\lambda) \rangle_{\mathfrak{S}} \ \text{for all} \ x,\lambda \in \Omega \right\}.
    $$
\end{definition}
For any \ac{rkhs} $\origspace$, we have $\mathfrak{S}_{\xi} \neq \emptyset$~\cite{jorgensen_new_2025}. The feature map representation can be useful for characterizing kernels. For example, let $\ell$ be the vector space of one--sided complex sequences and
\begin{equation*}
    \ell^2 \coloneqq \left\{ \bx \in \ell \ : \ \sum_{k=0}^{\infty} |x_k|^2 < \infty \right\}
\end{equation*}
be the Hilbert space endowed with the usual inner product
\begin{equation*}
    {\langle \bx, \by \rangle}_{\ell^2} \coloneqq \sum_{k=0}^{\infty} x_k \overline{y_k} \quad \big(\bx, \by \in \ell^2\big).
\end{equation*}
If $\origspace$ is a separable \ac{rkhs} and $\{ \varphi_k \}_{k=0}^{\infty} \subseteq \origspace$ is a complete orthonormal basis, then choosing the feature space as $\ell^2$ and $\Phi(x) = ( \varphi_k(x) )_{k=0}^{\infty} \in \ell^2$ yields $\xi(x, t) = {\langle \Phi(x), \Phi(t) \rangle}_{\ell^2}$. This statement also follows from Mercer's theorem~\cite{mercer_xvi_1909}.

We note that every finite--dimensional Hilbert space is an \ac{rkhs} and its kernel can be conveniently expressed using the feature map approach.
\begin{theorem}[Finite--dimensional Hilbert spaces~\cite{dozsa_adaptive_2026}]
    \label{thm:findimRKHS}
    Let $\origspace$ be a Hilbert space of functions $\Omega \to \C$ and $\{ \varphi_k \}_{k=0}^{n-1} \subseteq \origspace$ be a basis for some $n \in \N$. Then, the $n$--dimensional subspace $\cH_n \coloneqq \myspan\{\varphi_0,\ldots,\varphi_{n-1}\}$ is an \ac{rkhs} defined by the positive semi--definite reproducing kernel
    \begin{equation*}
        \xi(x,\lambda) = \sum_{k=0}^{n-1} \varphi_k(x) \overline{\varphi_k(\lambda)} \quad (x, \lambda \in \Omega).
    \end{equation*}
Furthermore, for any two functions $\bG = \sum_{k=0}^{n-1} \mathfrak{g}_k \varphi_k, \bF = \sum_{k=0}^{n-1} \mathfrak{f}_k \varphi_k \in \cH_n$, the inner products on $\cH_n$ and $\cH$ are related by
\begin{equation*}
    \left\langle \bG, \bF \right\rangle_{\cH_n} = \sum_{k=0}^{n-1}\mathfrak{g}_k \overline{\mathfrak{f}_k}, \quad 
    \left\langle \bG, \bF \right\rangle_{\cH} = \sum_{k=0}^{n-1} \sum_{j=0}^{n-1} \langle \varphi_k, \varphi_j \rangle_{\cH} \mathfrak{g}_k \overline{\mathfrak{f}_j}.
\end{equation*}
\end{theorem}

Importantly, for strictly positive definite kernels, kernel slices are independent for different choices of kernel centers.
\begin{theorem}[Linear independence of kernel slices (Theorem 3.6 in~\cite{paulsen_introduction_2016})]
    \label{thm:posdef}
    Let $\origspace$ be an \ac{rkhs} of functions $\Omega \to \C$ with strictly positive definite kernel $\xi$. Then, for any $(\lambda_0,\dots, \lambda_{n}) \in \Omeganeq{n+1}$, $\xi_{\lambda_n} \not \in \myspan\{ \xi_{\lambda_0},\dots,\xi_{\lambda_{n-1}} \}$.
\end{theorem}
Finally, we consider the following statement about finite--dimensional \acp{rkhs} embedded into Hilbert spaces.
\begin{lemma}[Embedded finite--dimensional \acp{rkhs}]
    \label{thm:embeddedRKHS}
    Let $\origspace$ be a Hilbert space (not necessarily an \ac{rkhs}) of functions $\Omega \to \C$. Let $r \in \N$ and let $ V_r \subseteq \origspace$ be an $r$--dimensional subspace of $\origspace$, and let $\left\{ \varphi_j \right\}_{j=0}^{r-1} \subseteq \origspace$ be an orthonormal system with respect to ${\langle \cdot, \cdot \rangle}_{\origspace}$. 
    Fix an arbitrary $\big(\tlambda_0,\ldots,\tlambda_{r-1} \big) \in \Omeganeq{r}$. Then
    \begin{equation}\label{eq:spaneq}
        V_r \coloneqq \myspan \left\{ \varphi_0, \ldots, \varphi_{r-1} \right\} = \myspan \left\{ \xi_{\tlambda_0}, \ldots, \xi_{\tlambda_{r-1}} \right\},
    \end{equation}
    where the $\Omega \times \Omega \to \C$ map
    \begin{equation*}
        \xi(x, \lambda) \coloneqq \sum_{j=0}^{r-1} \varphi_j(x) \overline{\varphi_j(\lambda)}
    \end{equation*}
    is the reproducing kernel in $V_r$ corresponding to the inner product ${\langle \cdot, \cdot \rangle}_{\origspace}$. Furthermore, for any $\bG \in \origspace$, we have
    \begin{equation}
        \label{eq:projRep}
        \left\langle \bG, \xi_{\tlambda_j} \right\rangle_{\origspace} = \cP_{V_r}[\bG]\big(\tlambda_j\big) \quad (j=0,\ldots,r-1),
    \end{equation}
    where $\cP_{V_r} : \origspace \to V_r$ is the orthogonal projector to $V_r$.
\end{lemma}
\begin{proof}
 By \Cref{thm:findimRKHS}, $V_r$ is an \ac{rkhs} associated with the reproducing kernel $\xi$ and with respect to the inner product ${\langle \cdot, \cdot \rangle}_{\origspace}$. Moreover, for any $(\alpha_0,\ldots,\alpha_{n-1}) \in \C^n$, we have
 \begin{equation*}
  \sum_{k=0}^{r-1} \alpha_k \xi_{\tlambda_k}(\cdot) = \sum_{k=0}^{r-1} \sum_{j=0}^{r-1}  \alpha_k \varphi_j(\cdot) \overline{\varphi_j(\tlambda_k)} = \sum_{j=0}^{r-1} \left(\sum_{k=0}^{r-1} \alpha_k \overline{\varphi_j(\tlambda_k)}\right) \varphi_j(\cdot),
 \end{equation*}
 which implies $\myspan\left\{\xi_{\tlambda_0},\ldots,\xi_{\tlambda_{r-1}} \right\} \subseteq V_r$. By \Cref{thm:posdef}, we obtain \eqref{eq:spaneq}. 
 
   If $\bG \in \origspace$, then~\eqref{eq:projRep} is a consequence of the fact that $\bG = \cP_{V_r}^{\perp}[\bG] + \cP_{V_r}[\bG]$, where $\cP_{V_r}^\perp$ is the projection operator to $\origspace \ominus V_r \coloneqq \origspace \cap V_r^\perp$, where $\ominus$ denotes the orthogonal difference between $\origspace$ and $V_r$. Then, using the reproducing property of the kernel, we get
    \begin{align*}
        \left\langle \bG, \xi_{\tlambda_j} \right\rangle_{\origspace} &= \left\langle \cP_{V_r}^{\perp}[\bG] + \cP_{V_r}[\bG], \xi_{\tlambda_j} \right\rangle_{\origspace} \\ &= \left\langle \cP_{V_r}^{\perp}[\bG], \xi_{\tlambda_j} \right\rangle_{\origspace} + \cP_{V_r}[\bG]\big(\tlambda_j\big) = \cP_{V_r}[\bG]\big(\tlambda_j\big).
    \end{align*}
\end{proof}

Finally, we recall the Hilbert space isomorphism theorem. The method, introduced in \Cref{sec:gmt}, gives a recipe to write such an isomorphism between any two separable \acp{rkhs} of the same dimension.

\begin{theorem}[Hilbert space isomorphism theorem (see, e.g., Theorem 4.14 in~\cite{conway_course_2007})]
    \label{thm:Hism}
    Two Hilbert spaces are isometrically isomorphic if and only if they have equal dimensions. 
\end{theorem}

\subsection{MOR of LTI systems}
\label{subsec:mor}
The methods proposed in this paper can be interpreted as generalizations of well--known \ac{mor} schemes developed for linear time--invariant (\ac{lti}) dynamical systems. For this reason, we briefly review the corresponding \ac{mor} methods. Later, we show that these well--known methods rely mostly on the fact that they operate in an \ac{rkhs}, namely the Hardy--Hilbert space $H_2(\C^+)$. This observation allows us to pose the problem in a more general \ac{rkhs} setting, where we obtain similar results.

First, consider a \ac{siso}--\ac{lti}--\ac{fom} in continuous--time, given by the state--space representation
\begin{equation}
    \label{eq:ss}
    \begin{split}
        \dot{\bx}(t) &= \bA \bx(t) + \bb \bu(t), \quad \bx(0) = 0, \\
        \by(t) &= \bc^\top \bx(t),
    \end{split}
\end{equation}
where $\bA \in \C^{n \times n}$, $\bb,\,\bc \in \C^{n \times 1}$ and $\bx(t) \in \C^{n}$, $\bu(t),\,\by(t) \in \C$ for all $t \in [0, \infty)$. The quantities $\bx(t)$, $\bu(t)$, and $\by(t)$ are referred to as the state, input, and output at time $t$. We refer to $n$ as the order of the system. The goal of \ac{lti}--\ac{mor} is to identify a \ac{rom} of order $r \ll n$ given by
\begin{equation}\label{eq:statROM}
    \begin{split}
        \dot{\bx}_r(t) &= \bA_r \bx_r(t) + \bb_r \bu(t), \quad \bx_r(0) = 0, \\
        \by_r(t) &= \bc_r^\top \bx_r(t),
    \end{split}
\end{equation}
where $\bA_r \in \C^{r \times r}$, $\bb_r,\bc_r \in \C^{r \times 1}$, $\bx_r(t) \in \C^{r}$, and $\by_r(t) \in \C$ for $t \in [0, \infty)$ such that $\by_r(t) \approx \by(t)$ for all admissible inputs $\bu(\cdot)$. 

The systems given in~\eqref{eq:ss} and~\eqref{eq:statROM} are fully defined by their respective transfer functions
\begin{equation}
    \label{eq:transfun}
    \bG(s) \coloneqq \bc^{\top} \left( s \bI - \bA\right)^{-1} \bb, \quad \bG_r(s) \coloneqq \bc_r^{\top} \left( s \bI - \bA_r\right)^{-1} \bb_r,
\end{equation}
where $s \in \C^{+} \coloneqq \{ w \in \C \ : \ \Re{w} > 0 \}$ belongs to the right complex half--plane. The transfer functions $\bG$ and $\bG_r$ are rational and their poles are eigenvalues of $\bA$ and $\bA_r$, respectively. If $\bA$ and $\bA_r$ from~\eqref{eq:ss} and~\eqref{eq:statROM} are Hurwitz, we call the corresponding systems asymptotically stable. In this case, the transfer functions belong to the Hardy--Hilbert space
\begin{equation*}
    H_{2}(\C^{+}) \coloneqq \left\{ \bG \in \cA(\C^{+}) \ : \ \frac{1}{2 \pi} \int_{-\infty}^{\infty} \left| \bG(\iu \omega) \right|^2\, \du \omega < \infty \right\},
\end{equation*}
where $\cA(\C^{+})$ denotes the set of analytic functions on $\C^{+}$. The inner product on this space is
\begin{equation*}
    {\langle \bG, \bH \rangle}_{H_2(\C^{+})} \coloneqq \frac{1}{2 \pi} \int_{-\infty}^{\infty} \bG(\iu \omega) \overline{\bH(\iu \omega)}\, \du \omega \quad (\bG,\bH \in H_2(\C^{+})).
\end{equation*}
Furthermore, it is an \ac{rkhs} with the Cauchy kernel
\begin{equation}
    \label{eq:kerh2c}
    \zeta(s, w) = \frac{1}{s + \overline{w}} \quad (s, w \in \C^{+}).
\end{equation}
Let $\mu_k$ and $\tilde{\mu}_j$ denote the eigenvalues of $\bA$ and $\bA_r$ from~\eqref{eq:ss} and~\eqref{eq:statROM}, respectively ($k=0,\ldots,n-1$, $j=0,\ldots,r-1$).  Then, we can write the transfer functions in~\eqref{eq:transfun} (see, e.g.,~\cite{gugercin_mathcalh_2_2008}) as
\begin{equation*}
    \bG = \sum_{k=0}^{n-1} \res_k \zeta_{\lambda_k}, \quad \bG_r = \sum_{j=0}^{r-1} \tres_j \zeta_{\tlambda_j},
\end{equation*}
where $\lambda_k = -\mu_k$, $\tlambda_j = - \tilde{\mu}_j$ and $\res_k, \res_j \in \C \ (k=0,\ldots,n-1, \ j=0,\ldots,r-1)$. Noticing that $H_2(\C^{+})$ is an \ac{rkhs} allows us to formulate the classical $H_2(\C^+)$--optimal \ac{mor} problem for \ac{siso}--\ac{lti} systems using \Cref{def:Objective}. In particular, fixing $\Omega \coloneqq \C^{+}$, $\xi \coloneqq \zeta$, and $\origspace \coloneqq  H_2(\C^{+})$, we notice that $\bG$ from~\eqref{eq:transfun} belongs to $\fomspacec$, where $\fomspacec \subseteq H_2(\C^{+})$ is a generalized model space defined according to \Cref{def:gmodspace}. Then, $H_2(\C^{+})$--optimal \ac{mor} seeks the \ac{rom} $\bG_r$ (see~\eqref{eq:transfun}) that is a local minimizer (see \Cref{def:locmin}) of the objective function $\rkhsError{\zeta}{r}$ given in \Cref{def:Objective}.
Such a local minimizer satisfies the following necessary conditions.
\begin{theorem}[Meier--Luenberger conditions~\cite{meier_approximation_1967, gugercin_mathcalh_2_2008}]
\label{thm:MLc}
      Let $\bG \in \fomspacec$ be a \ac{fom}. Suppose that $\bG_r \in  \romspacec$ of order $r$ is a local minimizer of $\rkhsError{\zeta}{r}$. Then,
        \begin{equation}
            \label{eq:MLC}
            \bG_r\big(\tlambda_j\big) = \bG\big(\tlambda_j\big), \quad \bG_r'\big(\tlambda_j\big) = \bG'\big(\tlambda_j\big) \quad (j=0,\ldots,r-1),
        \end{equation}
        where $\tlambda_j$ denote the parameters identifying $\bG_r$ (see~\eqref{eq:transfun}).
\end{theorem}
\Cref{thm:MLc} is equivalent to the following orthogonality conditions.
\begin{theorem}[Gugercin--Antoulas--Beattie conditions~\cite{gugercin_mathcalh_2_2008}]
    \label{thm:orthoCondH2}
    Let $\bG \in \fomspacec$ be a \ac{fom}. Suppose that $\bG_r \in  \romspacec$ of order $r$ is a local minimizer of $\rkhsError{\zeta}{r}$. Then,
    \begin{equation}
        \label{eq:orthCondH2}
        \left\langle \bG - \bG_r, \bG_r \cdot \bH_1 + \bH_2 \right\rangle_{H_2(\C^{+})} = 0 \quad \bigg(\bH_1, \bH_2 \in \myspan \left\{ \zeta_{\tlambda_0}, \ldots, \zeta_{\tlambda_{r-1}} \right\}\bigg),
    \end{equation}
    where $\tlambda_j$ denote the parameters identifying $\bG_r$ (see~\eqref{eq:transfun}).
\end{theorem}
\begin{remark}
    The equivalence between Theorems~\ref{thm:MLc} and~\ref{thm:orthoCondH2} is proven in~\cite{gugercin_mathcalh_2_2008}. This equivalence can also be easily shown using an \ac{rkhs} argument, see Appendix~\ref{app:equivMLc}. Even though this proof is a direct consequence of the general optimality condition introduced in \Cref{sec:orthCond}, it demonstrates that the result in $H_2(\C^+)$ can be proven by exploiting the reproducing property in \acp{rkhs}.
  \end{remark}   

The current state--of--the--art approach for obtaining a \ac{rom} that satisfies~\Cref{thm:MLc} and~\Cref{thm:orthoCondH2} is \ac{irka} introduced in~\cite{gugercin_mathcalh_2_2008}. This method uses repeated Galerkin projections, based on the state--space representations~\eqref{eq:ss} and~\eqref{eq:statROM}, to obtain the \ac{rom}. Although convergence is formally proven only for state--space--symmetric systems~\cite{flagg_convergence_2012}, the method has been empirically verified to yield high--quality \acp{rom} for most \ac{lti} systems. A limitation of \ac{irka} is that it can provide \acp{rom} whose parameters $\tlambda_j \not \in \C^{+}$ for some $j=0,\ldots,r-1$ and consequently $\bG_r \not \in \origspace$. In system theory, such systems are referred to as unstable. Recently, alternatives to \ac{irka} address this limitation in~\cite{mlinaric_irka_2025} and~\cite{cdcAngino}.

Finally, we briefly consider discrete--time \ac{lti} systems. Although analogs of Theorems~\ref{thm:MLc} and~\ref{thm:orthoCondH2} hold for this case~\cite{gugercin_mathcalh_2_2008}, we instead review an equivalent theorem recently introduced in~\cite{cdcAngino}. This formulation is very useful for the \ac{rkhs}--\ac{mor} methods developed in \Cref{sec:irka}. In the discrete--time case, the time variable $t$ in~\eqref{eq:ss} and~\eqref{eq:statROM} belongs to $\N$ and $\dot{\bx}(t)$ is replaced by $\bx[t+1]$. Discrete--time \ac{lti} systems can be fully described by rational functions
\begin{equation}
    \label{eq:disctransfun}
    \bG(z) = \sum_{k=0}^{n-1} \frac{\res_k}{1 - \overline{\lambda_k} z}, \quad \bG_r(z) = \sum_{j=0}^{r-1} \frac{\tres_j}{1 - \overline{\tlambda_j} z} \quad (z \in \D),
\end{equation}
where $\D$ denotes the open complex unit disk, $\res_k, \tres_j \in \C$ and $\lambda_k, \tlambda_j$ are defined as $\lambda_k = \overline{\mu_k}, \ \tlambda_j = \overline{\tilde{\mu}_j}$, with $\mu_k$ and $\tilde{\mu}_j$ denoting the eigenvalues of $\bA$ and $\bA_r$, respectively, for $k=0,\ldots,n-1$ and $j=0,\ldots,r-1$. For a thorough overview of this statement, see, e.g.,~\cite{cdcAngino, dozsa_impulse_2026}. In discrete time, asymptotically stable systems correspond to $\bA$ and $\bA_r$ from~\eqref{eq:ss} and~\eqref{eq:statROM} having eigenvalues in $\D$. In this case, the functions in~\eqref{eq:disctransfun} belong to the Hardy--Hilbert space
\begin{equation}
    \label{eq:H2D}
    H_2(\D) \coloneqq \left\{ \bG \in \cA(\D) \ : \ \frac{1}{2 \pi} \int_{-\pi}^{\pi} \left|\bG( \eu^{\iu t} )\right|^2 \du t < \infty \right\},
\end{equation}
where $\cA(\D)$ is s short--hand notation for the complex--analytic functions on $\D$.
This is again an \ac{rkhs} with the reproducing Szeg\H{o} kernel
\begin{equation}
    \label{eq:kerh2d}
    \rho(z, \lambda) = \frac{1}{1 - \overline{\lambda} z} \quad (z, \lambda \in \D).
\end{equation}
Kernel slices in $H_2(\D)$ will be denoted by $\rho_\lambda$ similarly to \Cref{def:reprKer}.
In order to state the necessary optimality condition, following~\cite{cdcAngino}, we first define the equilibrium function associated with the \ac{fom} $\bG \in \fomspaceh$. 
\begin{definition}[Equilibrium function]
    \label{def:equibFunH2}
    Let $\bG = \sum_{k=0}^{n-1} \res_k \rho_{\lambda_k} \in \fomspaceh$.The \emph{equilibrium function} associated with $\bG$ is given by $F_{\bG} : \D^r \to \C^r$, where \begin{equation*}
         \left[ F_{\bG}(\bz) \right]_j \coloneqq \sum_{k=0}^{n-1} \res_k \rho_{\lambda_k}(z_j) \prod_{i=0}^{r-1} \frac{\overline{\lambda_k} - \overline{z_i}}{1 - z_i \overline{\lambda_k}} \quad (j=0,\ldots,r-1).
     \end{equation*}
     for $\bz = [z_0, \ldots, z_{r-1}]^\top$.
\end{definition}
The following theorem provides a necessary condition that a discrete--time \ac{lti}--\ac{rom} $\bG_r \in \romspaceh$ must satisfy if it is a local minimizer of the error functional $\rkhsError{\rho}{r}$.\footnote{In~\cite{cdcAngino}, in the following theorem the authors only require $\bG_r \in \romspaceh$, whereas they should additionally require $\bG_r$ to be of order $r$ for the statement to hold. This inaccuracy is now addressed in \Cref{thm:equib}.}
\begin{theorem}[Equilibrium condition~\cite{cdcAngino}]
    \label{thm:equib}
    Let $\bG \in \frak{G}_{\rho}^n(\D)$, $F_{\bG}$ be the associated equilibrium function and $\bG_r = \sum_{j=0}^{r-1} \tres_j \rho_{\tlambda_j} \in \romspaceh$ of order $r$ a local minimizer of $\rkhsError{\rho}{r}$ (see \Cref{def:Objective}). Then, $\tilde{\blambda} = \left[ \tlambda_0, \ldots, \tlambda_{r-1} \right]^\top$ satisfies
  \begin{equation*}
      F_{\bG}\big(\tilde{\blambda}\big) = 0.
  \end{equation*}
\end{theorem}
Using \Cref{thm:equib}, we can pose the discrete time \ac{siso} \ac{lti}--\ac{mor} problem as a zero finding task. In~\cite{cdcAngino} this is solved with a Newton method which exploits hyperbolic geometry. This intrinsically guarantees that the obtained parameters $\tilde{\blambda}$ remain in $\D^r$ in each iteration of the algorithm, always resulting in an asymptotically stable \ac{rom}.

At each iteration, the method first attempts a standard Euclidean Newton step for the equilibrium equation.
If the resulting candidate remains in $\D^r$, the update is accepted; otherwise, the method switches to a hyperbolically corrected step that keeps the parameters inside $\D^r$. 
Thus, the equilibrium--based approach gives stricter structural guarantees than \ac{irka}, since the recovered parameters $\tilde{\blambda}$ remain in $\D^r$ throughout each iteration.  
For further details about the hybrid hyperbolic Newton step, we refer to~\cite{cdcAngino}.

In this study, we propose Newton iterations similar to Algorithm~\ref{alg:hybrid_equilibrium} in any \acp{rkhs} with (Wirtinger) differentiable kernels.  In addition, we extend this strategy to more general \acp{rkhs} in \Cref{sec:irka} by transforming the \ac{fom} in the original \ac{rkhs} to a \ac{fom} in a companion \ac{rkhs} with more desirable properties such as $H_2(\D)$ (because of the availability of \Cref{alg:hybrid_equilibrium}). This can be done exploiting the unitary transformations proposed in \Cref{subsec:GMTtrans}. 
\begin{algorithm}[tbp]
\caption{Hybrid equilibrium--based \ac{mor} in $H_2(\D)$}
\label{alg:hybrid_equilibrium}
\begin{algorithmic}[1]
\Require \ac{fom} $\bG \in \frak{G}_{\rho}^n$, reduced order $r < n$, initial parameters $\bz^{(0)} \in \D^r$.
\Ensure \ac{rom} $\bG_r \in \romspaceh$.
\State Construct the equilibrium function $F_{\bG} : \D^r \to \C^r$.
\State $\nu \gets 0$.
\Repeat
    \State Compute a Newton correction $\Delta \bz^{(\nu)}$ for $F_{\bG}\big(\bz^{(\nu)}\big) = 0$.
    
    \If{$\bz^{(\nu)} + \Delta \bz^{(\nu)} \in \D^r$}
        \State $\bz^{(\nu+1)} \gets \bz^{(\nu)} + \Delta \bz^{(\nu)}$.
        \Comment{Euclidean Newton step.}
    \Else
        \State Compute hyperbolic correction $\Delta \bz_{\mathrm{H}}^{(\nu)}$. \Comment{See Section 4 in~\cite{cdcAngino}.}
        \State $\bz^{(\nu+1)} \gets \operatorname{Exp}_{\bz^{(\nu)}}\!\left(\Delta \bz_{\mathrm{H}}^{(\nu)}\right)$.
    \EndIf
    \State $\nu \gets \nu + 1$.
\Until{convergence of $\bz^{(\nu)}$}
\State $\tilde{\blambda} \gets \bz^{(\nu)}$.
\State Construct $\bG_r \in \romspaceh$ using the parameters $\tilde{\blambda}$.
\State \Return $\bG_r$.
\end{algorithmic}
\end{algorithm}

Later, we will observe a unique advantage of Algorithm~\ref{alg:hybrid_equilibrium} over Galerkin--projection--based methods like \ac{irka} when reducing models in \acp{rkhs} defined over real intervals. Consider the space 
\begin{equation}
    \label{eq:realH2}
    \realHtwo \coloneqq H_2(\D)\big|_{(-1,1)}.
\end{equation}
Note that $\realHtwo$ is an \ac{rkhs} that is induced by the Szeg\H{o} kernel restricted to the domain $\mathbb{I}^2$, see~\cite[page 351]{aronszajn_theory_1950}. In particular, all rational models in $\fomspaceh \subseteq H_2(\D)$ with real kernel centers and real linear coefficients belong to $\realHtwo$. For such models, \Cref{alg:hybrid_equilibrium} produces a \ac{rom} candidate in each iteration maintaining these properties.

\subsection{Summary of notation}
\label{subsec:notation}
\Cref{tab:notation} collects the core notation introduced above that is used throughout the remainder of the paper. Notation specific to a single section (e.g., the \ac{gmt} system of \Cref{sec:gmt}) is instead introduced, and recalled, where it is first needed.
{
\small
\begin{longtable}{@{}ll@{}}
    
    \caption{Summary of the core notation used throughout the paper.}
    \label{tab:notation} \\
    \toprule
    \multicolumn{2}{l}{\textit{Domains and function spaces}} \\
    \midrule
    $\N$ & Natural numbers including $0$. \\
    $\Omega \subseteq \C$, $\origspace$ & Domain of the functions and the original \ac{rkhs} of functions $\Omega \to \C$. \\
    $\hat\Omega$, $\ambspace$ & The companion domain and companion \ac{rkhs} \\
    $\D$, $\C^{+}$ & Open unit complex disk, open right complex half--plane \\
    $H_2(\D)$, $H_2(\C^{+})$ & Hardy--Hilbert spaces on $\D$, $\C^{+}$ \\
    \midrule
    \multicolumn{2}{l}{\textit{Kernels}} \\
    \midrule
    $\xi$, $\hat\xi$ & Reproducing kernels of $\origspace$, $\ambspace$ \\
    $\xi_{\lambda} = \xi(\cdot, \lambda)$ & Kernel slice at $\lambda \in \Omega$ \\
    $\zeta$, $\rho$ & Reproducing kernels of $H_2(\C^{+})$~\eqref{eq:kerh2c} and  $H_2(\D)$~\eqref{eq:kerh2d} \\
    \midrule
    \multicolumn{2}{l}{\textit{Model spaces}} \\
    \midrule
    $n$, $r$ & Order of the \ac{fom} and  \ac{rom} with $r < n$ \\
    $\blambda = (\lambda_0, \ldots, \lambda_{n-1})$, $\tlambda_j$ & \ac{fom} kernel centers, an individual \ac{rom} kernel center \\
    $\fomspaceo$, $\romspaceo$ & Generalized \ac{fom}/\ac{rom} model spaces of order $n$ and $r$ in $\origspace$ \\
    $\wfomspaceo$, $\wromspaceo$ & Weak \ac{fom}/\ac{rom} model spaces (\Cref{def:weak}) \\
    $\bG$, $\bG_r$ & Full-- and reduced--order models \\
    $\res_k$ & Linear (kernel--expansion) coefficient \\
    \midrule
    \multicolumn{2}{l}{\textit{Derivatives and operators}} \\
    \midrule
    $\funder{\xi_{\lambda}}{\lambda}$, $\funder{\xi_{\lambda}}{\overline{\lambda}}$ & Wirtinger derivatives of $\xi_{\lambda}$ with respect to $\lambda$, $\overline{\lambda}$ \\
    $\bG'(z)$, $\bG^{(m)}(z)$ & Pointwise Wirtinger derivative in $z$, and its $m$th order analog \\
    $\rkhsError{\xi}{r}$ & norm--optimal \ac{mor} objective functional (\Cref{def:Objective}) \\
    $\cP_{V}$ & Orthogonal projector onto a subspace $V$ \\
    \bottomrule
\end{longtable}
}

\section{Necessary conditions of norm optimality in RKHS model reduction}
\label{sec:orthCond}

In this section, we introduce orthogonality and interpolatory conditions analogous to Theorems~\ref{thm:MLc} and~\ref{thm:orthoCondH2} in an \ac{rkhs} setting. In addition, we generalize the notion of the equilibrium function from \Cref{def:equibFunH2} and obtain a generalization of \Cref{thm:equib}.

In order to state the results in a general form, we update our definition of generalized model spaces from \Cref{def:gmodspace}. The remainder of this study uses the following, extended definition of $\fomspaceo$.

\begin{definition}[Generalized model space, dimension, model order]
\label{def:gmodspacediff}
Let $\Omega \subseteq \C$ and $\origspace$ be an \ac{rkhs} of functions $\Omega \to \C$. Suppose $(\lambda \mapsto \xi_{\lambda}) \in C^m(\Omega,\origspace)$ with $m \ge 1$ and $\funder[i]{\xi_{\lambda}}{\olambda} \in \origspace$ for all $\lambda \in \Omega$ and $i=0,\ldots,m$. The \emph{generalized model space} of \emph{dimension} $n$ is defined as
\begin{multline*}
         \mathfrak{G}_\xi^{n,m} \coloneqq \left\{ \bG \in \vphantom{\sum_{j=0}^{k-1}}\myspan \left\{ \xi_{\lambda_0}, \funder{\xi_{\lambda_0}}{\overline{\lambda}} \ldots, \funder[m_0]{\xi_{\lambda_0}}{\overline{\lambda}}, \ldots, \xi_{\lambda_{k-1}},\funder{\xi_{\lambda_{k-1}}}{\overline{\lambda}} \ldots, \funder[m_{k-1}]{\xi_{\lambda_{k-1}}}{\overline{\lambda}}  \right\} : \right.  \\ (\lambda_0,\ldots,\lambda_{k-1}) \in \Omeganeq{k}, \ \sum_{j=0}^{k-1} m_j = n , \ 0 \leq m_j\leq m \ \text{for} \ j=0,\ldots,k-1\Bigg\}.
    \end{multline*}
 The model $\bG$ is of \emph{order} $n \in \N$, if $\bG \in \mathfrak{G}_\xi^{n,m} \subseteq \origspace$, but $\bG \not \in  \mathfrak{G}_\xi^{r,m} \subseteq \origspace$ for any $r < n$.
\end{definition}

\begin{remark}
    \begin{enumerate}[label = \alph*)]
        \item  Note that $\mathfrak{G}_\xi^{n,0} = \fomspaceo$ with $\fomspaceo$ as in \Cref{def:gmodspace}.
        \item For ease of notation, let us from now on assume that the kernel in \Cref{def:gmodspacediff} is at most $m$--times continuously differentiable, i.e., $m$ denotes the maximum regularity of the kernel. For this reason, we will drop the superscript $m$ in $\mathfrak{G}_\xi^{n,m}$ from now on. 
    \end{enumerate}
\end{remark}

\begin{theorem}[A necessary condition for norm--optimal \acp{rom} in \acp{rkhs} with $m$--times continuously differentiable kernels]
\label{thm:orthocond}
Let $\Omega \subseteq \C$ and $\origspace$ be an \ac{rkhs} of functions $\Omega \to \C$. Suppose $(\lambda \mapsto \xi_{\lambda}) \in C^m(\Omega,\origspace)$ with $m \ge 1$ and $\funder[i]{\xi_{\lambda}}{\olambda} \in \origspace$ for all $\lambda \in \Omega$ and $i=0,\ldots,m$.

Given a fixed $\bG \in \origspace$, suppose that
\begin{equation*}
     \bG_r = \sum_{j=0}^{k-1} \sum_{i=0}^{m_j-1} \tres_{j,i} \funder[i]{\xi_{\tlambda_j}}{\overline{\tlambda}},
\end{equation*}
where $\big(\tlambda_0,\ldots,\tlambda_{k-1}\big) \in \Omeganeq{k}$, $\sum_{j=0}^{k-1} m_j = r$, $1 \leq m_j < m$, and $\tres_{j, m_j-1} \neq 0$ for $j=0,\ldots,k-1$ is a local minimizer of $\rkhsError{\xi}{r}$. Consider the $\origspace$ subspaces
\begin{equation}
    \label{eq:Xr}
    \Vn{r}{\tilde{\blambda}} \coloneqq \myspan \left\{ \xi_{\tlambda_0}, \funder{\xi_{\tlambda_0}}{\overline{\tlambda}} \ldots, \funder[m_0 - 1]{\xi_{\tlambda_0}}{\overline{\tlambda}}, \ldots, \xi_{\tlambda_{k-1}},\funder{\xi_{\tlambda_{k-1}}}{\overline{\tlambda}}, \ldots, \funder[m_{k-1}-1]{\xi_{\tlambda_{k-1}}}{\overline{\tlambda}}  \right\}
\end{equation}
and
\begin{equation}
    \label{eq:dXr}
    D\Vn{r}{\tilde{\blambda}} \coloneqq \myspan \left\{  \funder{\xi_{\tlambda_0}}{\overline{\tlambda}}, \funder[2]{\xi_{\tlambda_0}}{\overline{\tlambda}} \ldots, \funder[m_0]{\xi_{\tlambda_0}}{\overline{\tlambda}}, \ldots, \funder{\xi_{\tlambda_{k-1}}}{\overline{\tlambda}},\funder[2]{\xi_{\tlambda_{k-1}}}{\overline{\tlambda}}, \ldots, \funder[m_{k-1}]{\xi_{\tlambda_{k-1}}}{\overline{\tlambda}}  \right\}.
\end{equation}
Let us construct $\tVn{r}{\tilde{\blambda}} \subseteq \Vn{r}{\tilde{\blambda}}$ and  $D\tVn{r}{\tilde{\blambda}} \subseteq D\Vn{r}{\tilde{\blambda}}$ by removing $\funder[i]{\xi_{\tlambda_j}}{\overline{\tlambda}}$ and $\funder[i+1]{\xi_{\tlambda_j}}{\overline{\tlambda}}$ from $\Vn{r}{\tilde{\blambda}}$ and $D\Vn{r}{\tilde{\blambda}}$, respectively, if $\tres_{j, i} = 0$ for  $j=0,\ldots,k-1$ and $i=0, \ldots, m_j-2$.
Then,
\begin{equation}
    \label{eq:orthocond}
    \left\langle \bG - \bG_r, \bH_1 + \bH_2 \right\rangle_{\origspace} = 0
\end{equation}
for all $\bH_1 \in \tVn{r}{\tilde{\blambda}} \subseteq \romspaceo$ and $\bH_2 \in D\tVn{r}{\tilde{\blambda}}$.
\end{theorem}
\begin{proof}
    Notice that $\bG_r \in \tVn{r}{\tilde{\blambda}}$ and $\tVn{r}{\tilde{\blambda}}$ is a finite dimensional subspace of $\origspace$. From the Hilbert projection theorem, we know that for any choice of
    $\left\{ \tlambda_,\ldots,\tlambda_{k-1} \right\} \subseteq \Omega$,
    $\bG_r \in \tVn{r}{\tilde{\blambda}}$
    is the solution to 
    \begin{equation*}
        \min_{\bH \in \tVn{r}{\tilde{\blambda}}}
        {\| \bG - \bH \|}_{\origspace}
    \end{equation*}
    if and only if
    \begin{equation}
        \label{eq:projArg}
        \left\langle \bG - \bG_r, \bH \right\rangle_{\origspace} = 0
    \end{equation}
    for any
    $\bH \in \tVn{r}{\tilde{\blambda}}$.
    It therefore suffices to show
    \begin{equation}
        \label{eq:projArgD}
         \left\langle
            \bG - \bG_r,
            \bH_2
         \right\rangle_{\origspace} = 0 \quad \left(\bH_2 \in D \tVn{r}{\tilde{\blambda}}\right).
    \end{equation}
    Consider the error functional $\rkhsError{\xi}{r}$ given in
    \Cref{def:Objective}. Let
    \begin{equation*}
        \cJ\left(
            \tlambda_{0,\mathrm{re}}, \tlambda_{0,\mathrm{im}},
            \ldots,
            \tlambda_{k-1,\mathrm{re}}, \tlambda_{k-1,\mathrm{im}}
        \right)
        \coloneqq
        \left( \rkhsError{\xi}{r} \left( \bG_r \right) \right)^2
        =
        \left\langle
            \bG - \bG_r,
            \bG - \bG_r
        \right\rangle_{\origspace},
    \end{equation*}
    where
    $\tlambda_j = \tlambda_{j,\mathrm{re}} + \iu \cdot \tlambda_{j,\mathrm{im}}$
    for $j=0,\ldots,k-1$ and $\tlambda_i \neq \tlambda_j$ for all  $i,j = 0,\ldots,k-1$ with $i \neq j$.

    To avoid duplicating identical expressions for the real and imaginary parts,
    we write $\cod \in \{\mathrm{re},\mathrm{im}\}$ and use the notation
    $\tlambda_{j,\cod}$ to denote either $\tlambda_{j,\mathrm{re}}$ or
    $\tlambda_{j,\mathrm{im}}$. Then, since $\bG_r$ is a local minimizer of $\cJ$, we have
    \begin{equation}
        \label{eq:derivErr}
        \frac{\partial \cJ}{\partial \tlambda_{j,\cod}}
        =
        -2 \Re \left\langle
            \bG - \bG_r,
            \frac{\partial \bG_r}{\partial \tlambda_{j,\cod}}
        \right\rangle_{\origspace} = 0.
    \end{equation}
    From this, $\tlambda_i \neq \tlambda_j$ for $i \neq j$ gives
    \begin{equation}
        \label{eq:bGrDer}
        \frac{\partial \bG_r}{\partial \tlambda_{j,\cod}}
        =
       \sum_{i=1}^{m_j} \tres_{j, i-1}
        \frac{\partial^i \xi_{\tlambda_j}}{\partial \tlambda_{j,\cod}^i}.
    \end{equation}
    From~\eqref{eq:derivErr} and~\eqref{eq:bGrDer}, we obtain
    \begin{equation}
        \label{eq:derSums}
        \frac{\partial \cJ}{\partial \tlambda_{j,\cod}}
        =
        -2 \Re \left\langle
            \bG - \bG_r,
            \sum_{i=1}^{m_j} \tres_{j, i-1}
        \frac{\partial^i \xi_{\tlambda_j}}{\partial \tlambda_{j,\cod}^i}
        \right\rangle_{\origspace} = 0.
    \end{equation}
    We note that for $i=1,\ldots,m_j-1$ and $j=0,\ldots,k-1$, 
    and since \eqref{eq:projArg} is true for all $\bH \in \tVn{r}{\tilde{\blambda}}$, 
    \begin{equation*}
        \left\langle \bG - \bG_r, \frac{\partial^i \xi_{\tlambda_j}}{\partial \tlambda^i_{j, \cod}}  \right\rangle_{\origspace} = 0. 
    \end{equation*}
    With this, due to $\tres_{j, m_j-1} \neq 0$, $j=0,\ldots,k-1$, \eqref{eq:derSums}, and \eqref{eq:derivErr}, we obtain
    \begin{equation}
        \label{eq:ReCond}
        \Re \left\langle
            \bG - \bG_r,
            \frac{\partial^{m_j} \xi_{\tlambda_j}}{\partial \tlambda_{j,\cod}^{m_j}}
        \right\rangle_{\origspace}
        =
        0
        \quad
        (\cod \in \{\mathrm{re},\mathrm{im}\}, j=0,\ldots,k-1).
    \end{equation}
    Considering $
        D_{\overline{\tilde{\lambda}}}^i
        \coloneqq
        \frac{1}{2}
        \left(
            \frac{\partial^i}{\partial \tlambda_{j,\mathrm{re}}^i}
            +
            \iu
            \frac{\partial^i}{\partial \tlambda_{j,\mathrm{im}}^i}
        \right),$ for $j=0,\ldots,k-1$ and $i=0,\ldots,m_j$ we obtain
    \begin{equation*}
        \left\langle
            \bG - \bG_r,
           \funder[m_j]{\xi_{\tlambda_j}}{\overline{\tlambda}}
           \right\rangle_{\origspace}
        =
        \frac{1}{2}
        \left(
            \left\langle
                \bG - \bG_r,
                \frac{\partial^{m_j} \xi_{\tlambda_j}}{\partial \tlambda_{j,\mathrm{re}}^{m_j}}
            \right\rangle_{\origspace}
            +
            \iu
            \left\langle
                \bG - \bG_r,
                \frac{\partial^{m_j} \xi_{\tlambda_j}}{\partial \tlambda_{j,\mathrm{im}}^{m_j}}
            \right\rangle_{\origspace}
        \right).
    \end{equation*}
    Finally, due to~\eqref{eq:ReCond} we have
    \begin{equation*}
        \Re \left\langle
            \bG - \bG_r,
            \tres_{j, m_j-1}
            \frac{\partial^{m_j} \xi_{\tlambda_j}}{\partial \tlambda_{j,\mathrm{re}}^{m_j}}
        \right\rangle_{\origspace}
        =
        0,
    \end{equation*}
    from which
    \begin{equation*}
        \Re \left\langle
            \bG - \bG_r,
            \iu \cdot \tres_{j, m_j-1}
            \frac{\partial^{m_j} \xi_{\tlambda_j}}{\partial \tlambda_{j,\mathrm{re}}^{m_j}}
        \right\rangle_{\origspace}
        =
        \Im \left\langle
            \bG - \bG_r,
            \tres_{j, m_j-1}
            \frac{\partial^{m_j} \xi_{\tlambda_j}}{\partial \tlambda_{j,\mathrm{re}}^{m_j}}
        \right\rangle_{\origspace}
        =
        0
    \end{equation*}
    follows. By a similar argument, 
    $    \Re \left\langle
            \bG - \bG_r,
            \tres_{j, m_j-1}
            \frac{\partial^{m_j} \xi_{\tlambda_j}}{\partial \tlambda_{j,\mathrm{im}}^{m_j}}
        \right\rangle_{\origspace}
        =
        0,$
    thus~\eqref{eq:ReCond} implies
    \begin{equation*}
        \left\langle
            \bG - \bG_r,
            \funder[m_j]{\xi_{\tlambda_j}}{\overline{\tlambda}}
        \right\rangle_{\origspace}
        =
        0
        \quad
        (j=0,\ldots,k-1).
    \end{equation*}
    Together with~\eqref{eq:projArg}, this equation implies~\eqref{eq:projArgD} and hence~\eqref{eq:orthocond}. 
\end{proof}
In the next theorem, we obtain the \ac{rkhs} analog of the Meier--Luenberger optimality conditions from \Cref{thm:MLc}.
\begin{theorem}[Interpolatory optimality conditions in \acp{rkhs} with $m$--times differentiable kernels]
\label{thm:MLcRKHS}
With the assumptions and the notation of \Cref{thm:orthocond}, if 
\begin{equation*}
     \bG_r = \sum_{j=0}^{k-1} \sum_{i=0}^{m_j-1} \tres_{j,i}  \funder[i]{\xi_{\tlambda_j}}{\overline{\tlambda}},
\end{equation*}
where $\big(\tlambda_0,\ldots,\tlambda_{k-1}\big) \in \Omeganeq{k}$, $\sum_{j=0}^{k-1} m_j = r$, $1 \leq m_j < m$, and $\tres_{j, m_j-1} \neq 0$ for $j=0,\ldots,k-1$ is a local minimizer of $\rkhsError{\xi}{r}$ from \Cref{def:locmin}, then
\begin{equation}
    \label{eq:MLcRKHS}
    \bG^{(i)}\big(\tlambda_j\big) = \bG_r^{(i)}\big(\tlambda_j\big), \quad \textrm{and} \quad \bG^{(m_j)}\big(\tlambda_j\big) = \bG_r^{(m_j)}\big(\tlambda_j\big)
\end{equation}
for $j=0,\ldots,k-1$ and all $i=0,\ldots,m_j-1$ for which $\tres_{j,i} \neq 0$.
\end{theorem}
\begin{proof}
    Let $j \in \left\{0,\ldots,k-1 \right\}$ and $i \in \left\{ 0,\ldots,m_j-1 \right\}$ be arbitrary. If $\tres_{j,i} \neq 0$, we have $\funder[i]{\xi_{\tlambda_j}}{\overline{\tlambda}} \in \tVn{r}{\tilde{\blambda}}$, and hence, by \Cref{thm:orthocond},
    \begin{equation*}
        \left\langle \bG - \bG_r, \funder[i]{\xi_{\tlambda_j}}{\overline{\tlambda}} \right\rangle_{\origspace} = 0.
    \end{equation*}
    Similarly, since $\tres_{j, m_j-1} \neq 0$, we have $\funder[m_j]{\xi_{\tlambda_j}}{\overline{\tlambda}} \in D\tVn{r}{\tilde{\blambda}}$, which by \Cref{thm:orthocond} gives
    \begin{equation*}
        \left\langle \bG - \bG_r, \funder[m_j]{\xi_{\tlambda_j}}{\overline{\tlambda}} \right\rangle_{\origspace} = 0.
    \end{equation*}
    The statements now follow from \Cref{thm:derivKerOrder}, that is,
    \begin{equation*}
        0 = \left\langle \bG - \bG_r, \funder[i]{\xi_{\tlambda_j}}{\overline{\tlambda}} \right\rangle_{\origspace} \stackrel{\textrm{Cor.~\ref{thm:derivKerOrder}}}{=} \bG^{(i)}\big(\tlambda_j\big) - \bG_r^{(i)}\big(\tlambda_j\big)
    \end{equation*}
    for $j \in \left\{0,\ldots,k-1 \right\}$ and $i \in \{0,\ldots,m_j-1\}$ with $\tres_{j,i} \neq 0$, and for $i = m_j$.
\end{proof}

\begin{remark}
        \begin{enumerate}[label=(\alph*)]
            \item If we assume $\tlambda_0, \ldots, \tlambda_{r-1}$ are pairwise distinct (corresponding to $m_j = 1$ for $j=0,\ldots,r-1$), then the optimality conditions of \Cref{thm:orthocond} and \Cref{thm:MLcRKHS} simplify to 
            \begin{equation*}
                \left\langle \bG - \bG_r, \funder{\xi_{\tlambda_j}}{\overline{\tlambda}} + \bH \right\rangle_{\origspace} = 0 \quad \left(\bH \in \myspan \big\{ \xi_{\tlambda_0}, \ldots, \xi_{\tlambda_{r-1}} \big\}\right) 
            \end{equation*}
            and
            \begin{equation*}
                \bG \big( \tlambda_j \big) =  \bG_r \big( \tlambda_j \big) \quad \textrm{and} \quad  \bG' \big( \tlambda_j \big) =  \bG_r' \big( \tlambda_j \big),
            \end{equation*}
            for $j=0,\ldots,r-1$, respectively. The latter conditions match the classical Meier--Luenberger conditions (see \Cref{thm:MLc}) exactly, when $\origspace = H_2(\C^{+})$ or $\origspace = H_2(\D)$.  
        
            \item Notice that \Cref{thm:orthocond} does not make the assumption $\bG \in \fomspaceo$. Indeed, it suffices that the inner product in~\eqref{eq:orthocond} exists. Consequently, Theorems~\ref{thm:orthocond} and~\ref{thm:MLcRKHS} hold for any $\bG \in \origspace$ and $r < n$ as well.
            \item \Cref{thm:orthocond} can also be used in cases where $\bG \in \origspace$, but $\origspace$ is only assumed to be a Hilbert space (not necessarily an \ac{rkhs}). Indeed, consider an orthonormal system $\left\{ \varphi_j \right\}_{j=0}^{r-1} \subseteq \origspace$ and the kernel
            \begin{equation}
                \label{eq:hker}
                \xi(x, \lambda) = \sum_{j=0}^{r-1} \varphi_j(x) \overline{\varphi_j(\lambda)} \quad (x, \lambda \in \Omega).
            \end{equation}
            
            Then, considering for simplicity $\left\{ \tlambda_j \right\}_{j=0}^{r-1} \subset \Omeganeq{r}$, by Theorems~\ref{thm:findimRKHS} and~\ref{thm:posdef},
            \begin{equation*}
                \tVn{r}{\tilde{\blambda}} = \myspan \left\{ \xi_{\tlambda_0}, \ldots, \xi_{\tlambda_{r-1}} \right\} \subseteq \origspace
            \end{equation*}
            is an $r$--dimensional \ac{rkhs}. Furthermore, by \Cref{thm:embeddedRKHS}, the inner product of this \ac{rkhs} matches $\langle \cdot,\cdot\rangle_{\origspace}$. Considering all possible $\tilde{\blambda} \in \Omeganeq{r}$, we can define $\romspaceo \subseteq \origspace$ according to \Cref{def:gmodspace}. Notice that the proof of \Cref{thm:orthocond} does not rely on $\bG$ being in an \ac{rkhs}; rather, it only requires that the inner products 
            \begin{equation*}
                \left\langle \bG - \bG_r, \funder{\xi_{\lambda}}{\olambda} \right\rangle_{\origspace}
            \end{equation*}
            exist. Hence, \Cref{thm:orthocond} holds for an arbitrary $\bG \in \origspace$, where $\romspaceo$ is defined using the kernel~\eqref{eq:hker}. Note, however, that in this case, choosing different orthonormal bases $\left\{ \varphi_j \right\}_{j=0}^{r-1}$ will lead to a different reduced model space $\romspaceo \subseteq \origspace$. There is no natural choice, since the full space $\origspace$ does not possess a kernel of its own. The choice of this basis can therefore influence the quality of the \ac{rom}. Finally, we note that the \ac{rkhs}--based interpolatory Meier--Luenberger conditions given in \Cref{thm:MLcRKHS} do not follow from this, when $\origspace$ is not an \ac{rkhs}. Indeed, in this case the reproducing--kernel--based argument in the proof of \Cref{thm:MLcRKHS} cannot be used.
        \end{enumerate}
        \label{rem:opticonds}
\end{remark}

Next, we prove the equivalence between \Cref{thm:orthocond} and \Cref{thm:MLcRKHS}.
\begin{theorem}
    \label{thm:equivMlcOrtho}
    Consider the assumptions of \Cref{thm:orthocond}. The conditions~\eqref{eq:orthocond} and~\eqref{eq:MLcRKHS} are equivalent.
\end{theorem}
\begin{proof}
    Suppose first that~\eqref{eq:orthocond} holds. The proof of \Cref{thm:MLcRKHS} shows that~\eqref{eq:MLcRKHS} follows from this. Now suppose we have
    \begin{equation*}
    \bG^{(i)}\big(\tlambda_j\big) = \bG_r^{(i)}\big(\tlambda_j\big) \quad \textrm{and} \quad \bG^{(m_j)}\big(\tlambda_j\big) = \bG_r^{(m_j)}\big(\tlambda_j\big)
    \end{equation*}
    for $j=0,\ldots,k-1$ and $i=0\ldots,m_j-1$  for which $\tres_{j,i} \neq 0$. For $i < m_j$, this yields
    \begin{equation*}
        0 = \bG^{(i)}\big(\tlambda_j\big) - \bG_r^{(i)}\big(\tlambda_j\big) \stackrel{\textrm{Cor.~\ref{thm:derivKerOrder}}}{=} \left\langle \bG - \bG_r, \funder[i]{\xi_{\tlambda_{j}}}{\overline{\tlambda}}\right\rangle_{\origspace},  
    \end{equation*}
    hence $\bG - \bG_r \perp \tVn{r}{\tilde{\blambda}}$. Similarly, for $i=m_j$, we get
    \begin{equation*}
        0 = \bG^{(m_j)}\big(\tlambda_j\big) - \bG_r^{(m_j)}\big(\tlambda_j\big) \stackrel{\textrm{Cor.~\ref{thm:derivKerOrder}}}{=} \left\langle \bG - \bG_r, \funder[m_j]{\xi_{\tlambda_{j}}}{\overline{\tlambda}}\right\rangle_{\origspace}.  
    \end{equation*}
    This additional observation ensures $\bG - \bG_r \perp D \tVn{r}{\tilde{\blambda}}$, from which~\eqref{eq:orthocond} follows.
\end{proof}

The following lemma clarifies that the classical $H_2$ optimality conditions from \Cref{thm:orthoCondH2} and \Cref{thm:MLc} are special cases of the newly introduced theorems.
\begin{lemma}[Optimality conditions in $H_2\big(\C^{+} \big)$ are special cases of \Cref{thm:orthocond} and \Cref{thm:MLcRKHS}] 
\label{thm:specialCases}
    Let $\Omega = \C^{+}$, $\origspace = H_2( \C^{+} )$, $\bG \in \fomspacec$ be a \ac{fom}, and $\bG_r \in \romspacec$ be the $H_2 ( \C^{+} )$--norm--optimal \ac{rom} of order $r$ to $\bG$ given by
    \begin{equation*}
        \bG_r = \sum_{j=0}^{r-1} \tres_j \zeta_{\tlambda_j},
    \end{equation*}
    where we assume $\tilde{\blambda} \in \Omega^{r}_{\neq}$ and $\zeta$ is the Cauchy kernel from~\eqref{eq:kerh2c}. Then, \Cref{thm:orthocond} and \Cref{thm:MLcRKHS} apply to $\bG_r$ and are equivalent to \Cref{thm:orthoCondH2} and \Cref{thm:MLc}.
\end{lemma}
\begin{proof}
   The proof of the statement is given in \Cref{app:specialCases}. 
\end{proof}
\begin{remark}
    \begin{enumerate}[label=(\alph*)]
        \item  A similar result with identical reasoning can be given if $\origspace = H_2(\D)$.
        \item  Although no further formal proofs will be given, we note that recent $H_2$--optimal \ac{mor} results given by Borghi and Breiten in~\cite{borghi_mathcal_2024} are also likely special cases of the proposed theorems. Indeed, it is not difficult to see that the $H_2$ spaces defined over simply connected sets in these works are also \acp{rkhs}, which, at the very least, makes \Cref{thm:orthocond} and \Cref{thm:MLcRKHS} applicable to assess the optimality of \acp{rom}.
    \end{enumerate} 
\end{remark}

Another interesting necessary condition of optimality can be obtained if we consider \acp{rkhs} that contain polynomial--like objects. In order to introduce these properly, we consider multipliers on \acp{rkhs}.
\begin{definition}[Multipliers of $\origspace$, invariance]\label{def:multipliers}
    Let $\origspace$ be an \ac{rkhs} of functions $\Omega \to \C$. For some $\bG \in \origspace$, let
    \begin{equation*}
        D_{\bG} \coloneqq \left\{ \bH \in \origspace \ : \ \bG \cdot \bH \in \origspace \right\} \subseteq \origspace,
    \end{equation*}
    where $\bG \cdot \bH$ denotes pointwise multiplication. We call the linear operator $\cM_{\bG} : D_{\bG} \to \origspace$ a \emph{multiplier} of $\origspace$ associated with $\bG$ if
    \begin{equation*}
        \cM_{\bG}[\bH] \coloneqq \bG \cdot \bH \quad (\bH \in \origspace).
    \end{equation*}
 Additionally, when $D_{\bG} =\origspace$, we say that $\origspace$ is $\cM_{\bG}$--\emph{invariant}.
\end{definition}
By the closed graph theorem, multipliers on \acp{rkhs} are always bounded (see, e.g., Theorem 5.21 in~\cite{paulsen_introduction_2016}). Hence, we obtain a convenient notation for \acp{rkhs} that contain polynomial--like functions: we say that they are $\cM_{\myid}$--invariant, where $\myid$ is the identity operator on $\Omega$, and therefore,
\begin{equation*}
    \cM_{\myid}[\bH](\lambda) = \lambda \bH(\lambda) \quad (\lambda \in \Omega).
\end{equation*}

With these concepts, we are now ready to state the next optimality condition. It shows that in $\cM_{\myid}$--invariant \acp{rkhs}, we can describe the optimality of a \ac{rom} with a rational--like argument, similar to the results obtained in classical \ac{mor} literature for $H_2$ spaces~\cite{gugercin_mathcalh_2_2008}.

\begin{theorem}[A necessary condition for norm--optimal \acp{rom} in $\cM_{\myid}$--invariant \acp{rkhs}]
\label{thm:mzMLc}
    Suppose the assumptions of \Cref{thm:orthocond} are satisfied. In addition, suppose $\origspace$ is $\cM_{\myid}$--invariant as in \Cref{def:multipliers}.
    Let $(\lambda_0, \ldots, \lambda_{n-1}) \in \Omeganeq{n}$ and assume we have given the \ac{fom}
    \begin{equation*}
        \bG = \sum_{k=0}^{n-1} \res_k \xi_{\lambda_k} \in \fomspaceo,
    \end{equation*}
    Then, if $\big(\tlambda_0, \ldots, \tlambda_{r-1} \big) \in \Omeganeq{r}$ defines  an $\origspace$--norm--optimal \ac{rom}
    \begin{equation*}
        \bG_r = \sum_{j=0}^{r-1} \tres_j \xi_{\tlambda_j} \in \romspaceo,
    \end{equation*}
    then
    \begin{equation*}
        \left\langle \bG - \bG_r, \left( \cM_{\myid}^{\ast} - \overline{\tlambda_j} \myidH \right)\bigg|_{Y_r^{\tilde{\blambda}}}^{+} \left[\xi_{\tlambda_j} \right] \right\rangle_{\origspace} = 0 \quad (j=0,\ldots,r-1).
    \end{equation*}
    Here, $\cM_{\myid}^{\ast}$ denotes the adjoint of the multiplier $\cM_{\myid}$ and $\left( \cM_{\myid}^{\ast} - \overline{\tlambda_j} \myidH \right)\bigg|_{Y_r^{\tilde{\blambda}}}^{+}$ denotes the Moore--Penrose pseudoinverse of the operator $ \cM_{\myid}^{\ast} - \overline{\tlambda_j} \myidH$ restricted to
    \begin{equation*}
        Y_r^{\tilde{\blambda}} \coloneqq \tVn{r}{\tilde{\blambda}} \oplus D \tVn{r}{\tilde{\blambda}} =  \myspan\left\{ \xi_{\tlambda_0}, \funder{\xi_{\tlambda_0}}{\overline{\tlambda}}, \ldots,  \xi_{\tlambda_{r-1}}, \funder{\xi_{\tlambda_{r-1}}}{\overline{\tlambda}}\right\} \subseteq \origspace.
    \end{equation*}
\end{theorem}
\begin{proof}
    Since $\lambda \mapsto \xi_{\lambda}$ is differentiable, for all $\bH \in \origspace$, \Cref{thm:derivKer} gives
    \begin{equation}
        \label{eq:deriv1}
        \left\langle \bH, \cM_{\myid}^{\ast} \left[ \funder{\xi_{\lambda}}{\olambda}\right] \right\rangle_{\origspace} =  \left\langle \cM_{\myid}\left[\bH\right], \funder{\xi_{\lambda}}{\olambda} \right\rangle_{\origspace} = \left( \cM_{\myid}\left[\bH\right] \right)'(\lambda) = \bH(\lambda) + \lambda \bH'(\lambda).
    \end{equation}
    Notice, however, that $\bH(\lambda) + \lambda \bH'(\lambda)$ can be written as
    \begin{equation}
        \label{eq:deriv2}
        \bH(\lambda) + \lambda \bH'(\lambda) = \left\langle \bH, \xi_{\lambda}\right\rangle_{\origspace} + \lambda\left\langle \bH, \funder{\xi_{\lambda}}{\olambda} \right\rangle_{\origspace} = \left\langle \bH, \xi_{\lambda} + \overline{\lambda} \funder{\xi_{\lambda}}{\olambda} \right\rangle_{\origspace}.
    \end{equation}
    Equating the first equation in~\eqref{eq:deriv1} and the last equation in~\eqref{eq:deriv2} gives
    \begin{equation*}
        \left\langle \bH, \cM_{\myid}^{\ast} \left[\funder{\xi_{\lambda}}{\olambda}\right] \right\rangle_{\origspace} = \left\langle \bH, \xi_{\lambda} + \overline{\lambda} \funder{\xi_{\lambda}}{\olambda} \right\rangle_{\origspace},
    \end{equation*}
    which holds for all $\bH \in \origspace$, and thus,
    \begin{equation}
        \label{eq:adjMonDers}
        \cM_{\myid}^{\ast} \left[ \funder{\xi_{\lambda}}{\olambda} \right] = \xi_{\lambda} + \overline{\lambda} \funder{\xi_{\lambda}}{\olambda}.
    \end{equation}
    Rearranging this gives
    \begin{equation}
        \label{eq:toInvert}
        \xi_{\lambda} = \cM_{\myid}^{\ast} \left[ \funder{\xi_{\lambda}}{\olambda}\right] - \overline{\lambda}\funder{\xi_{\lambda}}{\olambda} =
        \left(\cM_{\myid}^{\ast} - \overline{\lambda} \myidH\right) \left[ \funder{\xi_{\lambda}}{\olambda} \right].
    \end{equation}
    Inverting the operator on the right--hand side is not possible in general, since for $\mu = \lambda$, we have $\left( \cM_{\myid}^{\ast} - \overline{\lambda} \myidH \right) \left[ \xi_{\mu}\right] = 0$. This follows from the observation that for any $\bH \in \origspace$, we have
    \begin{equation*}
        \left\langle \bH, \cM_{\myid}^{\ast} \left[ \xi_{\lambda} \right] \right\rangle_{\origspace}= \left\langle \cM_{\myid} \left[\bH\right], \xi_{\lambda}  \right\rangle_{\origspace} = \lambda \bH(\lambda) = \lambda \left\langle \bH, \xi_{\lambda} \right\rangle_{\origspace} = \left\langle \bH, \overline{\lambda} \xi_{\lambda} \right\rangle_{\origspace}.
    \end{equation*}
    The choice of $\bH$ here is arbitrary; consequently,
    \begin{equation}
        \label{eq:adjMonKernel}
        \cM_{\myid}^{\ast} \left[ \xi_{\lambda} \right] = \overline{\lambda} \xi_{\lambda},
    \end{equation}
    that is, the kernel slices $\xi_{\lambda}$ are eigenfunctions of $\cM_{\myid}^{\ast}$. 
    
    Due to \eqref{eq:adjMonDers} and \eqref{eq:adjMonKernel}, we know the effect of $\cM_{\myid}^{\ast}$ on the basis functions of $Y_r^{\tilde{\blambda}}$.
    In order to prove the claim, we  show that the Moore--Penrose pseudoinverse of the restricted operator $\left( \cM_{\myid}^{\ast} - \overline{\tlambda_j} \myidH \right)\bigg|_{Y_r^{\tilde{\blambda}}}$ satisfies
    \begin{equation*}
        \left( \cM_{\myid}^{\ast} - \overline{\tlambda_j} \myidH \right)\bigg|_{Y_r^{\tilde{\blambda}}}^{+} \left( \cM_{\myid}^{\ast} - \overline{\tlambda_j} \myidH \right)\bigg|_{Y_r^{\tilde{\blambda}}} \left[ \funder{\xi_{\tlambda_j}}{\overline{\lambda}} \right] = \funder{\xi_{\tlambda_j}}{\overline{\lambda}} \quad (j=0,\ldots,r-1),
    \end{equation*}
    for any $\funder{\xi_{\tlambda_j}}{\olambda} \in Y_r^{\tilde{\blambda}}$.
   Let 
   \begin{equation*}
       v_{2j} = \funder{\xi_{\tlambda_j}}{\overline{\lambda}}, \quad v_{2j + 1} = \xi_{\tlambda_j} \quad (j=0,\ldots,r-1).
   \end{equation*}
    Then, $Y_r^{\tilde{\blambda}} = \myspan \left\{ v_0, \ldots, v_{2r-1} \right\}$. 
    
From~\eqref{eq:adjMonDers} and~\eqref{eq:adjMonKernel}, we obtain the conditions
   \begin{equation*}
        \left( \cM_{\myid}^{\ast} - \overline{\tlambda_j} \myidH \right)
            \left[v_{2j}\right] = \xi_{\tlambda_j}, \quad  \left( \cM_{\myid}^{\ast} - \overline{\tlambda_j} \myidH \right)
            \left[v_{2j + 1}\right] = 0
    \end{equation*}
   for $j=0,\ldots,r-1$. 
Thus, $\left( \cM_{\myid}^{\ast} - \overline{\tlambda_j} \myidH \right)\bigg|_{Y_r^{\tilde{\blambda}}}$ can be represented by the $\C^{2r \times 2r}$ matrix
\begin{equation}
    \label{eq:Amat}
    \bA
    =
    \bI_r \otimes
    \begin{bmatrix}
        0 & 0\\
        1 & 0
    \end{bmatrix},
\end{equation}
where $\otimes$ denotes the Kronecker product and $\bI_r$ is identity matrix in $\C^{r\times r}$.
Letting $\be_k \in \C^{2r}$ denote the $k$th unit vector, for $k=0,\ldots, 2r -1$, $\bA$ satisfies
\begin{equation}
    \label{eq:Aeffect}
    \bA \be_k = \begin{cases}
        \be_{k+1}, & \text{if } k \bmod 2 = 0, \\
        0, & \textrm{otherwise}.
    \end{cases}
\end{equation}
While $\bA$ is singular, its Moore--Penrose inverse $\bA^+$ exists. We show 
\begin{equation}
    \label{eq:toshow}
    \be_{2j} = \bA^{+} \bA \be_{2j}  \quad  (j=0,\ldots,r-1).
\end{equation}
Here, $\bA^{+} \bA$ is the orthogonal projection operator that maps vectors onto the row space of $\bA$, which is the image of $\bA^*$. Since any vector $\bx \in \C^{2r}$ satisfies
\begin{equation*}
    \bx = \bx_{\mathrm{Row}} + \bx_{\mathrm{Null}},
\end{equation*}
where $\bx_{\mathrm{Row}}$ and $\bx_{\mathrm{Null}}$ belong to the row and null spaces of $\bA$, respectively, for~\eqref{eq:toshow} to hold, it suffices to show that $\be_{2j}$ is in the  row space of $\bA$, equivalently, orthogonal to its null space. This is enough, because for any $\bx_{\mathrm{Row}}$ in the row space, we have $\bx_{\mathrm{Row}} = \bA^+ \bA \bx_{\mathrm{Row}}$, since $\bA^+ \bA$ is an orthogonal projector. Notice that by~\eqref{eq:Amat},  the rank of $\bA$ is exactly $r$. Therefore, by the rank--nullity theorem, 
the dimension of the null space of $\bA$ is also $r$.
Note, however, that by~\eqref{eq:Aeffect}, $\be_{2j+1}$, $j=0,\ldots,r-1$ are exactly $r$ basis vectors spanning the  null space of $\bA$. Hence,~\eqref{eq:toshow} is equivalent to
\begin{equation*}
    \left\langle \be_{2j}, \be_{2k+1} \right\rangle_{\C^{2r}} = 0 \quad (j,k=0,\ldots,r-1),
\end{equation*}
which holds due to the orthogonality of $\be_k$ for $k=0,\ldots,2r-1$.

We recall that $\bA$ is the matrix representation of $\left( \cM_{\myid}^{\ast}  - \overline{\tlambda_j} \myidH \right)\bigg|_{Y_r^{\tilde{\blambda}}}$. The Moore--Penrose pseudoinverse of this restricted operator is the operator whose matrix representation is $\bA^{+}$ in the respective bases. In this case, by~\eqref{eq:toInvert} and~\eqref{eq:toshow}, we have 
 \begin{equation*}
        \funder{\xi_{\tlambda_j}}{\overline{\tlambda}} = \left( \cM_{\myid}^{\ast}  - \overline{\tlambda_j} \myidH \right)\bigg|_{Y_r^{\tilde{\blambda}}}^{+} \left[ \xi_{\tlambda_j} \right] \quad (j=0,\ldots,r-1)
    \end{equation*}
    and our claim follows from \Cref{thm:orthocond}.
\end{proof}
\begin{remark}
    \Cref{thm:mzMLc} is not directly used by the \ac{mor} algorithms or experiments presented later in this paper; we include it as a self--contained result of independent interest. It shows that, in $\cM_{\myid}$--invariant \acp{rkhs}, the necessary optimality condition retains a rational--like structure, similar in spirit to the classical orthogonality conditions familiar from the \ac{mor} literature, see \Cref{thm:orthoCondH2}.
\end{remark}

Finally, we formulate a necessary condition of $\origspace$ optimality in terms of an equilibrium function. This formulation will be especially useful for constructing an iterative method in \Cref{sec:irka} to obtain norm--optimal \acp{rom}.
\begin{definition}[Equilibrium function]
    \label{def:EquibFun}
    Let $\origspace$ be an \ac{rkhs} of functions $\Omega \to \C$ of dimension at least $n$ with kernel $\xi$ that satisfies the assumptions made in \Cref{thm:orthocond}, and let $\bG \in \fomspaceo$. Assuming $r \leq n$, let $\bmu = \big( \mu_0, \ldots, \mu_{r-1} \big) \in \Omega^r$ and let $m_{\mu_i} \le m$ count the repetitions of $\mu_i$ in $\bmu$ for $i=0,\ldots,r-1$. Furthermore, let $k \leq r$ and $\big( \tlambda_0, \ldots, \tlambda_{k-1}\big) \in \Omeganeq{k}$ be the $k$ pairwise distinct components of $\bmu \in \Omega^r$, i.e., assume there exist $r_0, \ldots, r_{k-1} \in \left\{ 0, \ldots, r-1 \right\}$, for which $\tlambda_j = \mu_{r_j}$ for $j=0,\ldots,k-1$.
    
    The \emph{equilibrium function} associated with $\bG$ is defined as $\cD : \Omega^r \to \C^r$ with 
    \begin{equation}
        \label{eq:EquibFun}
        \left[ \cD(\bmu) \right]_i \coloneqq \bG^{(m_{\mu_i})}(\mu_i) - \left(\cP_{\Vn{r}{\tilde{\blambda}}} \left[ \bG \right]\right)^{(m_{\mu_i})}\big(\mu_i\big) \quad (i=0,\ldots,r-1),
    \end{equation}
   where $\cP_{\Vn{r}{\tilde{\blambda}}}$ denotes the orthogonal projector onto the $r$--dimensional subspace $\Vn{r}{\tilde{\blambda}} \subseteq \origspace$ as in \eqref{eq:Xr} with $m_j \coloneqq m_{\mu_{r_j}}$, $j = 0,\ldots,k-1$.
\end{definition}
Using \Cref{def:EquibFun}, we arrive at the following optimality condition. 
\begin{theorem}
    \label{thm:equibCond}
    Consider the assumptions made in \Cref{thm:orthocond} and let $\cD : \Omega^r \to \C^r$ be the equilibrium function as given in \Cref{def:EquibFun}. If 
    \begin{equation*}
        \cD( \bmu ) = 0
    \end{equation*}
    for some $\mu \in \Omega^r$, then~\eqref{eq:MLcRKHS} and~\eqref{eq:orthocond} are satisfied by the \ac{rom} 
    \begin{equation*}
        \bG_r \coloneqq \cP_{\Vn{r}{\tilde{\blambda}}} \left[ \bG \right].
    \end{equation*}
\end{theorem}
\begin{proof}
    By our assumptions $\bG_r = \cP_{\Vn{r}{\tilde{\blambda}}} \left[ \bG \right]$, hence $\bG - \bG_r \perp \Vn{r}{\tilde{\blambda}}$ holds. Consequently, for any $\bH \in \Vn{r}{\tilde{\blambda}}$, we have $\left\langle \bG - \bG_r, \bH \right\rangle_{\origspace} = 0$. 
    From this, we obtain
    \begin{equation*}
        0 = \left\langle \bG - \bG_r, \funder[i]{\xi_{\tlambda_j}}{\overline{\tlambda}} \right\rangle_{\origspace} \stackrel{\text{Cor.}~\ref{thm:derivKerOrder}}{=} \bG^{(i)}\big(\tlambda_j\big) - \bG_r^{(i)}\big(\tlambda_j\big)
    \end{equation*}
    for $j=0,\ldots,k-1$ and $i=0,\ldots,m_{j}-1$. 

    By our assumptions that $\cD(\bmu) = 0$ and \Cref{def:EquibFun}, this means that 
    \begin{equation*}
        \bG^{( m_{j} )} \big( \tlambda_j \big) - \bG_r^{(m_j)} \big( \tlambda_j \big) = 0
    \end{equation*}
    for $j=0,\ldots,k-1$.
    This shows that~\eqref{eq:MLcRKHS} is satisfied. To show~\eqref{eq:orthocond}, we employ \Cref{thm:derivKerOrder} again to find
    \begin{equation*}
       0 = \bG^{( m_{j} )} \big( \tlambda_j \big) - \bG_r^{(m_j)} \big( \tlambda_j \big) 
       \stackrel{\text{Cor.}~\ref{thm:derivKerOrder}}{=} \left\langle \bG - \bG_r,  \funder[m_j]{\xi_{\tlambda_j}}{\overline{\tlambda}} \right\rangle_{\origspace},
    \end{equation*}
    which implies $\bG - \bG_r \perp D\Vn{r}{\tilde{\blambda}}$ with $D\Vn{r}{\tilde{\blambda}}$ as in~\eqref{eq:dXr}. 
\end{proof}
Note that in \Cref{thm:equibCond}, we only showed that finding a zero of the equilibrium function identifies a \ac{rom}, which satisfies the optimality conditions of \Cref{thm:orthocond} and \Cref{thm:MLcRKHS}. In applications, we can often assume that all kernel centers defining the \ac{rom} are distinct, i.e., $\big(\mu_0,\ldots,\mu_{r-1} \big) = \big(\tlambda_0, \ldots, \tlambda_{r-1} \big) \in \Omeganeq{r}$. This is commonly assumed for example in the reduction of \ac{lti} systems~\cite{cdcAngino, gugercin_mathcalh_2_2008}. With this additional assumption, we can show that the condition $\cD \big( \bmu \big) = \cD \big( \tilde{\blambda}\big) = 0$ not only implies, but is equivalent to \Cref{thm:orthocond} and \Cref{thm:MLcRKHS}.
\begin{theorem}
    \label{thm:equibCondDistinct}
    Suppose that the assumptions made in \Cref{thm:orthocond} are satisfied and let $\cD : \Omega^r \to \C^r$ be the equilibrium function as given in \Cref{def:EquibFun}. Assume further that $\big( \tlambda_0, \ldots, \tlambda_{r-1}\big) \in \Omeganeq{r}$. Then a \ac{rom}
    \begin{equation*}
    \bG_r \in \Vn{r}{\tilde{\blambda}} \subseteq \romspaceo,    
    \end{equation*}
    satisfies \eqref{eq:MLcRKHS} (with $m_j=1$, $j=0,\ldots,r-1$) if and only if
    \begin{equation*}
     \cD\big(\tilde{\blambda}\big) = 0.
    \end{equation*}
\end{theorem}
\begin{proof}
    Note that due to the distinct kernel center assumption, the subspaces $\Vn{r}{\tilde{\blambda}}$ from~\eqref{eq:Xr}, $D \Vn{r}{\tilde{\blambda}}$ from~\eqref{eq:dXr} coincide with the subspaces $\tVn{r}{\tilde{\blambda}}$ and $D \tVn{r}{\tilde{\blambda}}$ from \Cref{thm:orthocond}. These are given by
    \begin{equation*}
        \Vn{r}{\tilde{\blambda}} = \tVn{r}{\tilde{\blambda}} \coloneqq \myspan \left\{ \xi_{\tlambda_0}, \ldots, \xi_{\tlambda_{r-1}} \right\}
    \end{equation*}
    and
    \begin{equation*}
        D \Vn{r}{\tilde{\blambda}} = D \tVn{r}{\tilde{\blambda}} \coloneqq \myspan \left\{ \funder{\xi_{\tlambda_0}}{\overline{\tlambda}}, \ldots, \funder{\xi_{\tlambda_{r-1}}}{\overline{\tlambda}}\right\}.
    \end{equation*}
    First, assume $\bG\big(\tlambda_j \big) = \bG_r\big(\tlambda_j \big)$ and $\bG'\big(\tlambda_j \big) = \bG_r'\big(\tlambda_j \big)$ for $j=0,\ldots,r-1$, that is, \eqref{eq:MLcRKHS} with $m_j=1$, $j=0,\ldots,r-1$ is satisfied. 
    The reproducing property means
    \begin{equation*}
        \bG\big(\tlambda_j \big) - \bG_r \big(\tlambda_j \big) = 0 \iff \left\langle \bG - \bG_r, \xi_{\tlambda_j} \right\rangle_{\origspace} = 0 \quad (j=0,\ldots,r-1).
    \end{equation*}
    However, by the Hilbert projection theorem, this means $\bG - \bG_r \perp \Vn{r}{\tilde{\blambda}}$, and consequently $\bG_r = \cP_{\Vn{r}{\tilde{\blambda}}}\left[ \bG \right]$. In addition, by assumption, we have
    \begin{equation*}
        \left[\cD(\tilde{\blambda}) \right]_j = \bG'\big(\tlambda_j\big) - \left(\cP_{\Vn{r}{\tilde{\blambda}}}\left[ \bG \right]\right)' \big(\tlambda_j\big) = \bG'\big(\tlambda_j\big) - \bG_r'\big(\tlambda_j\big) = 0.
    \end{equation*}

    Conversely, assume that $\cD\big(\tilde{\blambda}\big) = 0$. By \Cref{thm:equibCond}, $\bG_r = \cP_{\Vn{r}{\tilde{\blambda}}}\left[ \bG \right]$, and we have
    \begin{multline*}
        \bG\big(\tlambda_j\big) - \bG_r \big(\tlambda_j \big) =  \bG\big(\tlambda_j \big) - \cP_{\Vn{r}{\tilde{\blambda}}}\left[ \bG \right] \big(\tlambda_j \big) = \left\langle \bG - \cP_{\Vn{r}{\tilde{\blambda}}}\left[ \bG \right], \xi_{\tlambda_j}  \right\rangle_{\origspace} = 0 \\ (j=0,\ldots,r-1).
    \end{multline*}
    Since $\cD\big(\tilde{\blambda}\big) = 0$, we then have
    \begin{equation*}
        \left[ \cD \big( \tilde{\blambda} \big) \right]_j = \bG' \big( \tlambda_j \big) - \left( \cP_{\Vn{r}{\tilde{\blambda}}}\left[ \bG \right] \right)'\big( \tlambda_j \big) = \bG'\big(\tlambda_j\big) - \bG_r'\big(\tlambda_j\big) = 0 \quad (j=0,\ldots,r-1).
    \end{equation*}
\end{proof}
\begin{remark}
    Our implementations of the algorithms introduced in \Cref{sec:irka} make the assumption $\big( \tlambda_0, \ldots, \tlambda_{r-1} \big) \in \Omeganeq{r}$. The results in \Cref{sec:exp} show that even under this assumption, high quality \acp{rom} can be obtained for a diverse family of \acp{fom}.
\end{remark}

\section{Generalized Malmquist--Takenaka systems}
\label{sec:gmt}
In this section, we introduce a novel orthonormal function system called \emph{generalized Malmquist--Takenaka (\ac{gmt}) functions} that can be constructed in any \ac{rkhs}. Notice that Theorems~\ref{thm:orthocond} and~\ref{thm:MLcRKHS} can only be applied in \acp{rkhs} with sufficiently smooth kernels. For \ac{mor} problems posed in \acp{rkhs}, where the above theorems cannot be applied, we introduce a method to transform the \ac{mor} problem to an admissible \ac{rkhs}. With the help of \ac{gmt} functions, we obtain isometries (in fact unitary maps) that can be used to conduct this transformation procedure. For an illustration of this pipeline, see \Cref{fig:pipeline}.

\subsection{Generalized MT functions in RKHSs}
Henceforth, for $\Omega \subseteq \C$, let $\ell_{\Omega}$ denote the set of sequences whose components belong to $\Omega$. Similarly, we denote by $\ell_{\N}$ the set of sequences containing natural numbers. We make use of the following assumptions throughout this section.
\begin{assumption}
    \label{ass:GMTkers}
    Let $\origspace$ be an \ac{rkhs} of functions $ \C \supseteq \Omega \to \C$ with kernel $\xi$. Suppose $\big( \lambda \mapsto \xi_{\lambda} \big) \in C^m(\Omega, \origspace)$ for some $m \in \N$ and $\funder[i]{\xi_{\lambda}}{\olambda} \in \origspace$ for $i=0,\ldots,m$. Let $\blambda = (\lambda_0,\lambda_1,\lambda_2,\ldots) \in \ell_{\Omega}$ be such that each of its unique elements appears at most $m+1$ times in the sequence. Let $\mathbf{d} = (d_0,d_1,d_2,\ldots) \in \ell_{\N}$ be defined such that $d_k$ is number of occurrences of the value $\lambda_k$ in the vector $(\lambda_0,\ldots,\lambda_{k-1})$ for all $k \in \N$. 
    
    Denote by $\left\{ \eta_k \right\}_{k=0}^{\infty} \subseteq \origspace$ the function system
    \begin{equation}
        \label{eq:kers}
        \eta_k \coloneqq 
            \funder[d_k]{\xi_{\lambda_k}}{\olambda}.
    \end{equation}
    Suppose finally that the functions $\left\{ \eta_0,\ldots,\eta_{n-1} \right\}$ are linearly independent for all admissible choices of $\blambda$ (and $\mathbf{d}$) and $n \le \dim \origspace$.
\end{assumption}

In this section, we further define
\begin{equation}
\label{eq:VnSec5}
\Vn{n}{\blambda}  \coloneqq \myspan \left\{ \eta_0, \ldots, \eta_{n-1} \right\} \subseteq \origspace.
\end{equation}
Note that this definition of $\Vn{n}{\blambda}$ corresponds to $\Vn{n}{\blambda}$ from \Cref{thm:orthocond} with the difference that now repeated entries in $\blambda$ are allowed. In addition, for ease of notation, note that $\blambda$ will denote both sequences and finite--length vectors of kernel centers depending on the context. We are now ready to introduce the \ac{gmt} basis functions.
\begin{definition}[Generalized Malmquist--Takenaka functions, \ac{gmt}--Fourier partial sum]
    \label{def:GMT}
    Suppose Assumption~\ref{ass:GMTkers} is satisfied. For $n \ge 0$, the $n$th \emph{\ac{gmt} function} is defined as
    \begin{equation*}
        \gmtn \coloneqq \frac{\eta_n - \cS_{n}^{\blambda} \left[ \eta_n \right]  }{ \left\| \eta_n - \cS_{n}^{\blambda} \left[ \eta_n \right] \right\|_{\origspace}},
    \end{equation*}
    where for $n \ge 0$,
    \begin{equation}
        \label{eq:GMTpartsum}
         \cS_{n}^{\blambda} \left[ \bG \right] \coloneqq 
            \sum_{k=0}^{n-1} \left\langle \bG, \gmtk \right\rangle_{\origspace} \gmtk   \quad \left(\bG \in \origspace \right).
    \end{equation}
    We refer to the operator $\cS_{n}^{\blambda}$ as the $n$th \emph{\ac{gmt}--Fourier partial sum}.  
\end{definition}
\begin{remark}
If we consider $d_k=0$ for $k=0,\ldots,n$ in Assumption~\ref{ass:GMTkers}, then \Cref{def:GMT} describes a Gram--Schmidt orthogonalization procedure applied to the pairwise distinct kernel slices $\xi_{\lambda_k}$. Hence, the \ac{gmt} system can be constructed in any \ac{rkhs}. If $d_k > 0$ for some $0 \leq k \leq n$, then the orthogonalization procedure acts on the kernel slices and their derivatives up to at most order $m$. Notice that the actual functions undergoing orthogonalization are parameterized by the sequence $\blambda \in \ell_{\Omega}$.
\end{remark}
\Cref{def:GMT} is a generalization of the classical Malmquist--Takenaka (\ac{mt}) system~\cite{mi_fast_2012, qian_kernel_2026, dozsa_bernoullis_2024}. Indeed, if we choose $\origspace = H_2(\D)$ and $\blambda \in \ell_{\D}$, then the above definition gives the well--known rational orthonormal system
\begin{equation}
    \label{eq:MT}
    \gmtn(z) = \frac{\sqrt{1 - |\lambda_n|^2}}{1 - \overline{\lambda_n} z} \prod_{k=0}^{n-1} \frac{z - \lambda_k}{1 - \overline{\lambda_k} z} \quad \left(z \in \D, \blambda \in \ell_{\D}\right).
\end{equation}
This system is complete in $H_2(\D)$ (see~\eqref{eq:H2D}) if and only if $\blambda \in \ell_{\D}$ satisfies the Blaschke--Sz\'asz condition~\cite{heuberger_modelling_2005}
\begin{equation}
    \label{eq:szasz}
    \sum_{n=0}^{\infty} \left( 1 - |\lambda_n| \right) = \infty. 
\end{equation}
 \ac{mt} functions have many peculiar properties that make them important for the study of Hardy spaces. In \Cref{subsec:GMTprop}, we show that several of these properties are retained by the \ac{gmt} system in \Cref{def:GMT}. \ac{mt} functions also have many applications, for example, in system identification~\cite{dozsa_bernoullis_2024, heuberger_modelling_2005, gozse_realizing_2015, dozsa_pole_2023, szabo_identification_1999}, rational interpolation~\cite{nagy-csiha_construction_2023, garcia_finite_2018}, medical signal processing~\cite{kovacs_generalized_2020}, and model order reduction~\cite{soumelidis_hyperbolic_2017, mi_fast_2012}.

We note that the classical \ac{mt} functions in~\eqref{eq:MT} form a rather general orthonormal system in $H_2(\D)$. Indeed, if we choose $\blambda = (0,0, 0,\ldots)$, then~\eqref{eq:MT} reduces to $\gmtn(z) = z^n$ $(z \in \D)$. This can be analytically continued to $\T$, where the functions then become the classical trigonometric system. Choosing $\blambda = (\lambda, \lambda, \ldots) \ (\lambda \in \D \setminus \{0\})$ yields the Laguerre orthonormal system in $H_2(\D)$ with many applications in system theory~\cite{dozsa_impulse_2026, wahlberg_system_1991}. Finally, we note that the classical \ac{mt} system has been previously generalized to other specific \acp{rkhs}. For example, in~\cite{pap_multiresolution_2011}, \ac{mt} functions are constructed in Bergman spaces. Additionally, using conformal maps, they can be defined on $H_2$ spaces of the half plane in~\cite{nagy-csiha_construction_2023, cdcAngino}.

Some attempts to define \ac{gmt}--like systems in \acp{rkhs} have also been made previously. In~\cite{muller_newton_2009}, a Newton--basis is defined for general \acp{rkhs}, which matches the above definition if $\lambda_k \neq \lambda_j$ when $j \neq k$ for all $k,j=0,\ldots,n$. In~\cite{muller_newton_2009}, the authors only investigate certain interpolatory properties of the proposed functions and make no connection to the classical rational \ac{mt} system. \Cref{def:GMT} generalizes this Newton--basis by allowing repeated kernel centers in $\blambda$ if the kernel $\xi$ is sufficiently smooth. Recently, in~\cite{qian_kernel_2026}, generalized \ac{mt} functions have been proposed in a family of \acp{rkhs} known as boundary weighted Hardy spaces. Instead of the Gram--Schmidt method, the authors of~\cite{qian_kernel_2026} consider an inner--outer factorization--based approach to generate their \ac{mt} system. This restricts their constructions to a family of $H_2$--like analytic function spaces. In addition, they do not consider \acp{rkhs} where the map $\lambda \mapsto \xi_{\lambda}$ is only $m$--times differentiable. Hence, the kernel \ac{mt} system defined in~\cite{qian_kernel_2026} is also an important special case of \Cref{def:GMT}. To the best of our knowledge, however, there have been no previous generalizations to cover repeated kernel centers in $\blambda \in \ell_{\Omega}$ in arbitrary \acp{rkhs} as given in \Cref{def:GMT}. 

Numerical concerns regarding a possible definition similar to \Cref{def:GMT} have been voiced in~\cite{qian_kernel_2026} due to the application of the Gram--Schmidt method. These concerns are partially addressed empirically in \Cref{sec:exp}, where a variety of experiments shows that, in many \acp{rkhs}, \ac{gmt} functions can in fact be safely generated. In addition, we introduce a biorthogonal system later in \Cref{subsec:GMTprop}, which can be used to transform \acp{fom} to companion spaces without the need for the explicit computation of \ac{gmt} functions. Finally, we note that the proposed biorthogonal \ac{gmt} basis can be interpreted as the analog of Lagrange basis polynomials for interpolation problems in \acp{rkhs}. 

\subsection{Properties of GMT functions}
\label{subsec:GMTprop}
We now show some elementary properties of \ac{gmt} functions. We will make the  following assumption throughout this section.
\begin{assumption}
    \label{ass:RKHSprops}
    Assume that $\origspace$ is an \ac{rkhs} of functions $\C \supseteq \Omega \to \C$ with reproducing kernel $\xi$ and let furthermore, $\blambda \in \ell_{\Omega}$. In addition, let $\big\{ \gmtk \big\}_{k=0}^{n-1}$ be the first $n$ \ac{gmt} functions generated by $\blambda$ as per \Cref{def:GMT}.
\end{assumption}

\begin{lemma}[Orthogonality of the \ac{gmt} system]
\label{thm:GMTONS}
Under \Cref{ass:RKHSprops}, for any choice of $n \in \N$ and \ac{rkhs} $\origspace$, the \ac{gmt} functions $\big\{\gmtk\big\}_{k=0}^{n-1}$ are orthonormal, that is, for $j,k=0,\ldots,n-1$, we have
\begin{enumerate}[label=(\alph*)]
\item 
    \itemeq{\left\langle \gmtj, \gmtk \right\rangle_{\origspace} = \delta_{j,k} \coloneqq \begin{cases}
        0, & \text{if } j \neq k, \\
        1, & \text{if } j = k,
    \end{cases}}   \label{eq:GMTONS}
\item 
       \itemeq{\myspan\left\{ \varphi^{\blambda}_0, \varphi^{\blambda}_1, \ldots, \varphi^{\blambda}_{n-1} \right\} = \Vn{n}{\blambda}.}
  \label{eq:GMTspan}
\end{enumerate}
\end{lemma}
\begin{proof}
    The statement follows from the observation that \Cref{def:GMT} is a Gram--Schmidt procedure in $\origspace$.
\end{proof}
\begin{remark}
    The orthogonality property is also proven for the \ac{rkhs} Newton basis introduced in~\cite[Theorem~6]{muller_newton_2009}. This corresponds to the \ac{gmt} system when the generating sequence $\blambda$ does not contain any repeated kernel centers.
\end{remark}

The next claim verifies that we can use \ac{gmt} functions to characterize the generalized model spaces from \Cref{def:gmodspace}.

The ordering of the first $n$ components in $\blambda \in \ell_{\Omega}$ does not affect the subspace generated by the first $n$ \ac{gmt} functions. Hence, there are $n!$ \ac{gmt} systems that can be used to span $\Vn{n}{\blambda}$. The next theorem gives insight into the placement of zeros of \ac{gmt} functions.
\begin{theorem}[Zeros]
    \label{thm:GMTzers}
 Under~\Cref{ass:GMTkers} and~\Cref{ass:RKHSprops}, it holds that
    \begin{equation*}
        \left(\gmtk\right)^{(d_j)}(\lambda_j) = 0 \quad (k=0,\ldots,n-1, \ j=0,\ldots, k-1).
    \end{equation*}
\end{theorem}
\begin{proof}
    Let $n > 0$ and $k \leq n$. By~\eqref{eq:GMTONS} and~\eqref{eq:GMTspan}, we have $\gmtn \perp \Vn{k}{\blambda}$ and $\eta_j = \funder[d_j]{\xi_{\lambda_j}}{\olambda} \in \Vn{k}{\blambda}$ for $j=0,\ldots,k-1$. Thus,
    \begin{equation*}
        0 = \left\langle \gmtn, \funder[d_j]{\xi_{\lambda_j}}{\olambda} \right\rangle_{\origspace} \stackrel{\textrm{Cor.~\ref{thm:derivKerOrder}}}{=}  \left(\gmtn\right)^{(d_j)}(\lambda_j).
    \end{equation*}
\end{proof}

\begin{remark}
\Cref{thm:GMTzers} also holds for the classical \ac{mt} functions and can be easily seen from the factorized form~\eqref{eq:MT}.
\end{remark}
For some applications, we need to consider real--valued \acp{rkhs}. In the real--valued case, the Chebyshev or Haar property is key for many applications such as uniform approximation of continuous functions. The following theorem establishes this property for the \ac{gmt} system under some additional assumptions.
\begin{theorem}[Chebyshev--Haar property]
    \label{thm:haar}
    Let $\origspace$ be an \ac{rkhs} of functions $\Omega \supseteq \R \to \R$. Under~\Cref{ass:GMTkers}, suppose that for any choice of $n \geq 1$ and
    \begin{equation*}
        t_0 < t_1 < \ldots < t_{n-1} \in \Omega \quad \textrm{and} \quad \lambda_0 < \lambda_1 < \ldots < \lambda_{n-1} \in \Omega,
    \end{equation*}
    we have
    \begin{equation}
        \label{eq:totalPos}
        \det\left[ \xi(t_k, \lambda_j) \right]_{k,j=0}^{n-1} > 0.
    \end{equation}
    Then, the \ac{gmt} system $\left\{ \gmtk \right\}_{k=0}^{n-1}$ has the Chebyshev--Haar property, that is, the zeros identified in \Cref{thm:GMTzers} are the only zeros.
\end{theorem}
\begin{proof}
    We note that by the assumptions of the theorem, all kernel centers in $\blambda$ are distinct. Hence, by \Cref{def:GMT} and \Cref{thm:GMTONS}, we have
    \begin{equation*}
        \Vn{n}{\blambda} = \myspan\left\{ \gmt_0, \ldots, \gmt_{n-1} \right\} = \myspan\left\{ \xi_{\lambda_0}, \ldots, \xi_{\lambda_{n-1}} \right\}.
    \end{equation*}
    From this, we observe that for any $k=0,\ldots,n-1$ we can write
    \begin{equation*}
        \gmtk = \sum_{j=0}^{k} a_{k,j} \xi_{\lambda_j},
    \end{equation*}
    where $a_{k,k} \neq 0$. Introduce the matrix $\bA = \left[ a_{k,j} \right]_{k,j=0}^{n-1}$ with the convention that $a_{k,j} = 0$ whenever $j > k$. Similarly, introduce $\Xi = \left[ \xi(t_j, \lambda_k) \right]_{k,j=0}^{n-1}.$
    Then, we can express the matrix $\Phi$, which satisfies $\left[ \Phi \right]_{k, j} = \gmtk(t_j)$ for $k,j=0,\ldots,n-1$ as $\Phi = \bA \Xi$.
    From this,
    \begin{equation*}
        \det \Phi = \det \bA \det \Xi,
    \end{equation*}
    where $\det \Xi > 0$ by~\eqref{eq:totalPos} and $\det \bA \neq 0$ because it is a triangular matrix with nonzero diagonal entries. This implies the statement. Indeed, consider $\bH \in \Vn{n}{\blambda} \setminus \{0\}$ and suppose $\bH$ has $n$ distinct zeros, i.e.,
    \begin{equation*}
        \bH(t_k) = \sum_{j=0}^{n-1} c_j \gmtj(t_k) = 0 \quad (k=0,\ldots,n-1).
    \end{equation*}
    Taking $\bc \coloneqq \left[ c_0, \ldots, c_{n-1} \right]^\top$, we can write this in matrix form $\Phi \bc = 0$.
    However, $\det \Phi \neq 0$, hence the only solution to this is $\bc = 0$, which contradicts the assumption $\bH \neq 0$. Thus, $\bH \in \Vn{n}{\blambda}$ cannot have $n$ distinct zeros. Since $\gmt_{n-1} \in \Vn{n}{\blambda}$, the $n-1$ zeros from \Cref{thm:GMTzers} are the only zeros.
\end{proof}
\begin{remark}
    \begin{enumerate}[label = (\alph*)]
        \item Notice that under the assumptions of \Cref{thm:haar}, every initial segment $\gmt_0, \ldots, \gmtk$ with $k < n$ also has the Chebyshev--Haar property.
        \item The Chebyshev--Haar property is key for various numerical methods. For example, the Remez algorithm, which constructs best approximations in the Chebyshev norm to continuous functions, has an extension to Haar systems in general~\cite{petrushev_rational_1988}.
    \end{enumerate}
\end{remark}

The next claim asserts that the \ac{gmt}--Fourier partial sums given in~\eqref{eq:GMTpartsum} are orthogonal projection operators onto the subspace $\Vn{n}{\blambda}$.
\begin{lemma}[\ac{gmt} projection]
\label{thm:orthoProjGMT}
   Under~\Cref{ass:RKHSprops} and \Cref{def:GMT}, we have 
   \begin{equation*}
       \gmtPsumn = \cP_{\Vn{n}{\blambda}}.
   \end{equation*}
\end{lemma}
\begin{proof}
 With the orthonormal basis $\left\{ \gmt_0,\ldots,\gmt_{n-1} \right\}$ of $\Vn{n}{\blambda}$, by standard linear algebra arguments we obtain
  \begin{equation*}
   \cP_{\Vn{n}{\blambda}}\left[ \bG \right] = \sum_{k=0}^{n-1} \left\langle \bG, \gmt_k \right\rangle_{\origspace} \gmt_k = \gmtPsumn \left[ \bG \right] \quad (\bG \in \origspace).
  \end{equation*}

\end{proof}

The following theorem shows that \ac{gmt}--Fourier partial sums provide a convenient way to obtain interpolants in \acp{rkhs}. This can be seen as a generalization of the rational interpolation technique proposed in~\cite{nagy-csiha_construction_2023} using classical \ac{mt} expansions. It also generalizes the Newton--basis--based interpolation in \acp{rkhs}~\cite{muller_newton_2009}, which did not allow for kernel centers with higher multiplicities. With the proposed \Cref{def:GMT}, we can claim much more. Suppose that $k \leq n-1$ is the highest index of appearance of $\lambda$. Then, by the next statement, we find that our \ac{gmt} expansions are Hermite interpolants of any function $\bG \in \origspace$ at the node $\lambda$ up to order $d_k$.

\begin{theorem}[Interpolatory property of \ac{gmt} partial sums]
    \label{thm:interpGMT} Suppose that \Cref{ass:GMTkers} and \Cref{ass:RKHSprops} hold. Then, for any $\bG \in \origspace$,
    \begin{equation*}
        \left(\gmtPsumn \left[\bG\right]\right)^{(d_k)} (\lambda_k) = \bG^{(d_k)}(\lambda_k) \quad (k=0,\ldots,n-1),
    \end{equation*}
    where $\gmtPsumn$ defines the \ac{gmt}--Fourier partial sums given in~\eqref{eq:GMTpartsum}.
\end{theorem}
\begin{proof}
    By \Cref{thm:orthoProjGMT}, $\gmtPsumn$ is an orthogonal projector onto $\Vn{n}{\blambda}$.
    Hence, for any $\bG \in \origspace$ and $k=0,\ldots,n-1$ we have
    \begin{multline*}
        0 = \left\langle \bG -  \gmtPsumn \left[\bG \right], \eta_k \right\rangle_{\origspace} = \left\langle \bG, \funder[d_k]{\xi_{\lambda_k}}{\olambda} \right\rangle_{\origspace} -  \left\langle \gmtPsumn \left[\bG \right], \funder[d_k]{\xi_{\lambda_k}}{\olambda} \right\rangle_{\origspace} \\   \stackrel{\textrm{Cor.~ \ref{thm:derivKerOrder}}}{=} \bG^{(d_k)}(\lambda_k) - \left( \gmtPsumn \left[\bG \right]\right)^{(d_k)}(\lambda_k).
    \end{multline*}
\end{proof}
With the help of \ac{gmt} functions, we can also define Lagrange interpolation schemes in \acp{rkhs}. For this, consider the following definition of generalized Lagrange basis polynomials.
\begin{definition}[Lagrange--\ac{gmt} basis functions]
\label{def:Lagrange}
   Under~\Cref{ass:GMTkers}, assume that $\lambda_0, \ldots, \lambda_{n-1}$ are pairwise distinct. Suppose also that $\xi(\lambda, \lambda) \neq 0$ for all $\lambda \in \Omega$. Let 
   \begin{equation*}
       \blambda^{l} \coloneqq \left(\lambda_0, \ldots, \lambda_{l-1}, \lambda_{l+1}, \ldots, \lambda_{n-1}, \lambda_{l}\right) 
   \end{equation*}
   for $l=0,\ldots,n-1$. The \emph{Lagrange--\ac{gmt} basis function} corresponding to the node $\lambda_l$ is given by
   \begin{equation*}
       \psi_l \coloneqq \frac{\varphi_{n-1}^{\blambda^{l}}}{\varphi_{n-1}^{\blambda^{l}}(\lambda_l)},
   \end{equation*}
   where $\varphi_{n-1}^{\blambda^{l}}$ is the $n$th \ac{gmt} function generated by $\blambda^l$ for $l=0,\ldots,n-1$.
\end{definition}
We now show some key properties of the Lagrange--\ac{gmt} basis from \Cref{def:Lagrange}.
\begin{theorem}[Lagrange property]
    \label{thm:LagrangeProps}
    The Lagrange--\ac{gmt} functions from \Cref{def:Lagrange}  satisfy the Lagrange property and moreover, $\left\{\psi_k \right\}_{k=0}^{n-1}$ forms a biorthogonal system to $\left\{ \xi_{\lambda_k} \right\}_{k=0}^{n-1}$. That is, 
    \begin{equation*}
       \left\langle \psi_k, \xi_{\lambda_j} \right\rangle_{\origspace} = \psi_k \left( \lambda_j \right) = \delta_{j,k}
    \end{equation*}
    for $k,j=0,\ldots,n-1$ (where $\delta_{j,k}$ denotes again the Kronecker symbol). 
\end{theorem}
\begin{proof}
Consider first $j=k$. Then,
    \begin{equation*}
        \left\langle \psi_k, \xi_{\lambda_k} \right\rangle_{\origspace} = \psi_k(\lambda_k) = \frac{\varphi_{n-1}^{\blambda^k} (\lambda_k)}{\varphi_{n-1}^{\blambda^k} (\lambda_k)} = 1.
    \end{equation*}
    If $k \neq j$, by \Cref{thm:GMTzers} and \Cref{def:Lagrange}, we have
    \begin{equation*}
        \left\langle \psi_k, \xi_{\lambda_j} \right\rangle_{\origspace} = \psi_k(\lambda_j) = \frac{\varphi_{n-1}^{\blambda^k} (\lambda_j)}{\varphi_{n-1}^{\blambda^k} (\lambda_k)} = 0.
    \end{equation*}
\end{proof}

We now turn our attention to more general properties of the \ac{gmt} system. First, we discuss the question of completeness in \acp{rkhs} and raise a point about the Blaschke--Sz\'asz property for \ac{gmt} systems in general \acp{rkhs}. Then, we introduce a \ac{gmt}--like biorthogonal system in $\cM_{\myid}$--invariant \acp{rkhs} to express $\cP_{\Vn{n}{\blambda}}$.

{\sloppy The next theorem clarifies the existence of complete \ac{gmt} systems in infinite--dimensional \acp{rkhs}.\par}
\begin{theorem}[Complete \ac{gmt} systems]
   Let $\origspace$ be an infinite--dimensional separable \ac{rkhs} of functions $\C \supseteq \Omega \to \C$. Suppose~\Cref{ass:GMTkers} holds and let $\left\{ \gmtk \right\}_{k=0}^{\infty}$ denote the \ac{gmt} system. Then there exists a $\blambda \in \ell_{\Omega}$ such that $\left\{ \gmtn \right\}_{n=0}^{\infty}$ forms a complete orthonormal system in $\origspace$.
\end{theorem}
\begin{proof}
    The statement is a direct consequence of \Cref{thm:kerdens}.
\end{proof}
\begin{remark}
    A natural question is whether there are analogs of the Blaschke--Sz\'asz condition (see~\eqref{eq:szasz}) for \ac{gmt} functions in a general separable \ac{rkhs}. We do not attempt to answer this question in this paper; however, we note that, in addition to $H_2(\D)$, similar conditions have already been defined in other \acp{rkhs}. We highlight the Bergman space $A_2(\D)$, which contains absolutely square--integrable holomorphic functions of the open unit disk~\cite{duren_bergman_2024}. For this space, a special case of the \ac{gmt} functions has been defined in~\cite{pap_multiresolution_2011}. In addition, an analog of the Blaschke--Sz\'asz condition is also known~\cite{horowitz_zeros_1974}. While the condition in this Hilbert space looks very similar to~\eqref{eq:szasz}, the proofs of these theorems are quite different. In particular, we note that the proof in Bergman spaces does not rely on inner--outer factorizations (unlike the proof of~\eqref{eq:szasz}, see~\cite[Chapter 2]{noauthor_theory_1970}), as these do not exist in Bergman spaces. We refer to~\cite[Chapter 4]{duren_bergman_2024}  and~\cite{horowitz_zeros_1974} for a proof of the Blaschke--Sz\'asz condition for Bergman spaces. It is therefore an interesting question whether a Blaschke--Sz\'asz condition applicable to a wide class of \ac{rkhs} spaces can be proved. An analog of the Blaschke--Sz\'asz condition is also shown in~\cite{qian_kernel_2026} for boundary weighted Hardy Spaces.
\end{remark}

In~\cite{qian_kernel_2026}, the authors show that \ac{mt} functions can be efficiently generated in a family of $H_2$--like \acp{rkhs} known as boundary weighted Hardy spaces using inner--outer factorizations. Many \acp{rkhs} of interest, such as Bergman spaces, \ac{rbf} spaces, and others, however, do not belong to this class. The next theorem defines a biorthogonal system in \acp{rkhs} with $\cM_{\myid}$ invariance. This class of \acp{rkhs} is much wider than boundary weighted Hardy spaces; for example, it also contains Bergman spaces and several non--analytic \acp{rkhs}. The following biorthogonal expansions offer a way to transform \ac{mor} problems to admissible \acp{rkhs} as introduced in \Cref{subsec:GMTtrans} when $\origspace$ is $\cM_{\myid}$--invariant. They are inspired by the Lagrange--\ac{gmt} functions from \Cref{def:Lagrange}.

\begin{theorem}[Biorthogonal \ac{gmt} systems]
    \label{thm:shiftInv}
    Let~\Cref{ass:RKHSprops} be satisfied and assume furthermore that $\origspace$ is $\cM_{\myid}$--invariant as in~\Cref{def:multipliers}. Suppose the first $n \in \N$ elements of $\blambda$ are pairwise distinct, that is, $d_k = 0$ for $k=0,\ldots,n-1$. Finally, suppose that $\xi(\lambda_k, \lambda_k) \neq 0$ for $k=0,\ldots,n-1$. Let
    \begin{equation}
        \label{eq:BGMT}
        \vartheta_k = \xi_{\lambda_k} \cdot \ell_{k} \quad (k = 0,\ldots,n-1),
    \end{equation}
    where $\ell_{k}$ is the Lagrange basis polynomial satisfying
    \begin{equation*}
        \ell_{k}(x) = \prod_{\substack{j=0 \\ j \neq k}}^{n-1} \frac{x - \lambda_j}{\lambda_k - \lambda_j}.
    \end{equation*}
    Then, the $k$th \ac{bgmt} function
    \begin{equation*}
        \bior_k(x) \coloneqq \frac{\vartheta_k}{\xi(\lambda_k, \lambda_k)}
    \end{equation*}
    is biorthogonal to the system $\left\{ \xi_{\lambda_k} \right\}_{k=0}^{n-1}$, that is, 
        $\left\langle \bior_k, \xi_{\lambda_j} \right\rangle_{\origspace} = \delta_{j,k}$
    for $j,k=0,\ldots,n-1$.
\end{theorem}
\begin{proof}
   Let $k,j \in \{0, 1,\ldots, n-1\}$. Since $\vartheta_k \in \origspace$ due to $\cM_{\myid}$ invariance we have
   \begin{equation*}
       \left\langle \vartheta_k, \xi_{\lambda_j} \right\rangle_{\origspace} = \left\langle \xi_{\lambda_k} \ell_k, \xi_{\lambda_j} \right\rangle_{\origspace} = \xi_{\lambda_k}(\lambda_j) \ell_k(\lambda_j) = \begin{cases}
           \xi_{\lambda_k}(\lambda_k), & \text{if } k = j, \\
           0, & \text{if } k \neq j.
       \end{cases}
   \end{equation*}
\end{proof}
\begin{remark}
    \begin{enumerate}[label = (\alph*)]
        \item \Cref{thm:shiftInv} could easily be extended to include repeated entries in $\blambda$, but nonvanishing assumptions on higher--order derivatives of $\xi(\lambda_k, \lambda_k)$ for $k=0,\ldots,n-1$ would have to be imposed.

        \item \Cref{thm:shiftInv} covers a much wider collection of \acp{rkhs} than the traditional Hardy--like spaces. Indeed, the \ac{rkhs} $\origspace$ is only required to be $\cM_{\myid}$--invariant, but does need not to have a Szeg\H{o}--like kernel, nor even contain analytic functions. For example, Bergman spaces~\cite{duren_bergman_2024} satisfy the conditions of \Cref{thm:shiftInv}, but fall outside the class of boundary weighted Hardy spaces, for which efficient \ac{gmt} constructions are proposed in~\cite{qian_kernel_2026}. Other good examples are Fock spaces~\cite{zhu_analysis_2012}, which do not even contain analytic functions but are frequently used in signal processing~\cite{de_figueiredo_generalized_1983}.
    \end{enumerate}
\end{remark}

The following lemma notes how \ac{bgmt} functions can be used to represent models in $\cM_{\myid}$--invariant \acp{rkhs}.
\begin{lemma}
    \label{thm:biorExp}
   Suppose the assumptions of \Cref{thm:shiftInv} are satisfied and let $\bG \in \fomspaceo \subseteq \origspace$. We have
    \begin{equation*}
        \bG = \sum_{k=0}^{n-1} \left\langle \bG, \bior_k \right\rangle_{\origspace} \xi_{\lambda_k}.
    \end{equation*}
\end{lemma}
\begin{proof}
    Since for any $\bG \in \fomspaceo$, we have
    \begin{equation*}
        \bG = \sum_{k=0}^{n-1} \res_k \xi_{\lambda_k},
    \end{equation*}
    it suffices to show $\res_k = \left\langle \bG, \bior_k \right\rangle_{\origspace}$.
    This is indeed the case, since
    \begin{equation*}
        \left\langle \bG, \bior_k \right\rangle_{\origspace} =  \left\langle \sum_{j=0}^{n-1} \res_j \xi_{\lambda_j}, \bior_k \right\rangle_{\origspace} = \sum_{j=0}^{n-1} \res_j \left\langle \xi_{\lambda_j}, \bior_k \right\rangle_{\origspace} \stackrel{\text{Thm.}~{\ref{thm:shiftInv}}}{=} \res_k.
    \end{equation*}
\end{proof}
\begin{corollary}
For $\bG \in \origspace$, the operator $\bgmtPsumn : \origspace \to \Vn{n}{\blambda}$ given by
\begin{equation*}
    \bgmtPsumn \left[ \bG \right] = \sum_{k=0}^{n-1} \left\langle \bG, \bior_k \right\rangle_{\origspace} \xi_{\lambda_k}
\end{equation*}
is the orthogonal projector to $\Vn{n}{\blambda}$, that is,
\begin{equation*}
    \bgmtPsumn = \cP_{\Vn{n}{\blambda}}.
\end{equation*}
\end{corollary}

We finish our discussion on the properties of \ac{gmt} functions with a note on numerical aspects. Prior works have refrained from the pursuit of \Cref{def:GMT}, citing the numerical instability of Gram--Schmidt orthogonalization in \acp{rkhs}~\cite{qian_kernel_2026}. While it is certainly true that the standard Gram--Schmidt procedure becomes numerically unstable when applied to basis vectors that are nearly linearly dependent, in \Cref{sec:exp} we give empirical evidence that computing several hundred and in some cases several thousand \ac{gmt} functions is entirely possible, with standard Cholesky--factorization--based orthogonalization~\cite{terao_lu-cholesky_2020}, which we use in our implementation to obtain \ac{gmt} systems from kernel slices. Our experiments cover a variety of \acp{rkhs}, and in a few cases, for example due to the distinct support of kernel slices, the orthogonalization procedure behaves extremely well even for large $n$. We also note that randomized orthogonalization techniques based on sketching have recently enjoyed much attention from the research community~\cite{wei_random_2024, jang_randomized_2025}. Documenting these results, and adapting them to the \ac{rkhs} setting if needed, should be a future research priority.

Finally, we note that the biorthogonal system in \Cref{thm:shiftInv} can be used to address numerical concerns of the \ac{mor} methods proposed in this study. Indeed, the Lagrange basis polynomials $\ell_k$ in~\eqref{eq:BGMT} can be efficiently computed using a barycentric--formula--based approach~\cite{nagy-csiha_construction_2023, berrut_barycentric_2004}, even for large $n$. Then, \Cref{thm:biorExp} can be used to express models in $\fomspaceo$ in a numerically safe manner. We also note that an analog of the \ac{gmt} transformation introduced in the next section can be defined with the \ac{bgmt} system in~\eqref{eq:BGMT} to transform \ac{mor} problems to--and--from $\cM_{\myid}$--invariant \acp{rkhs}.

Finally, in many specific \acp{rkhs}, efficient algorithms have already been developed to evaluate special cases of the \ac{gmt} functions. In $H_2(\D)$, discrete orthogonal variants have been considered in~\cite{dozsa_bernoullis_2024}. It has also been shown that special sampling schemes within $\Omega=\D$ induce a \ac{gmt}--based multi--resolution of $H_2(\D)$~\cite{pap_hyperbolic_2011} and the Bergman space~\cite{pap_multiresolution_2011}. For Laguerre bases, fast methods for computing expansion coefficients have recently been proposed in~\cite{dozsa_impulse_2026}. A more general construction based on the \ac{fft} for rational systems is given in~\cite{schipp_rational_2003}. Generalizing some of these results for \acp{rkhs} could bypass the need for computing $\gmtn$ in a kernel basis. Instead, using an appropriate discrete sampling of a $\bG \in \origspace$, one could obtain the \ac{gmt} expansion coefficients of $\bG$ directly.

\subsection{The GMT transformation}\label{subsec:GMTtrans}

We are now ready to introduce a family of unitary transformations that will allow us to develop the desired \ac{mor} methods in \Cref{sec:irka}. First, we make the following set of assumptions.
\begin{assumption}
    \label{ass:GMTtrans}
    Let $\Omega, \hat{\Omega} \subseteq \C$ and suppose that $\origspace$ and $\ambspace$ are \acp{rkhs} of functions $\Omega \to \C$ and $\hat{\Omega} \to \C$, respectively. Let $\gamma : \Omega \to \hat{\Omega}$ be a bijection. Denote by $\xi$ and $\hat{\xi}$ the reproducing kernels of $\origspace$ and $\ambspace$. Finally, let $\left\{\gmtk \right\}_{k=0}^{n-1}$ and $\left\{ \tgmtk \right\}_{k=0}^{n-1}$ be two \ac{gmt} systems spanning 
    \begin{equation}\label{eq:defVn}
            \Vn{n}{\blambda} = \myspan \left\{ \gmtk \right\}_{k=0}^{n-1} \subseteq \origspace, \quad \hVn{n}{\gamma(\blambda)} = \myspan \left\{ \tgmtk \right\}_{k=0}^{n-1} \subseteq \ambspace.
    \end{equation}
\end{assumption}

\begin{definition}[\ac{gmt} transformation]
    \label{def:GMTtransform}
    Consider Assumption~\ref{ass:GMTtrans}.We define the \emph{\ac{gmt} transformation} $\gmtTransVn : \Vn{n}{\blambda} \to \hVn{n}{\gamma(\blambda)}$ as $\gmtTransVn \left[ \gmtk \right] \coloneqq \tgmtk$
    for $k=0,\ldots,n-1$, and for an arbitrary $\bG \in \Vn{n}{\blambda}$ as
    \begin{multline}
         \label{eq:GMTtransfrom}
         \gmtTransVn \left[\bG\right] = \gmtTransVn \left[ \sum_{k=0}^{n-1} \left\langle \bG, \gmtk \right\rangle_{\origspace} \gmtk  \right] = \sum_{k=0}^{n-1} \left\langle \bG, \gmtk \right\rangle_{\origspace} \gmtTransVn \left[ \gmtk \right] \\ =  \sum_{k=0}^{n-1} \left\langle \bG, \gmtk \right\rangle_{\origspace} \tgmtk = \hat{\bG} \in \hVn{n}{\gamma(\blambda)} \subseteq \ambspace.
    \end{multline}
\end{definition}
From \Cref{def:GMTtransform}, we immediately obtain the following property.
\begin{theorem}
    \label{thm:unitary}
    Under~\Cref{ass:GMTtrans}, the \ac{gmt} transformation in \Cref{def:GMTtransform} is linear and unitary, the latter meaning that
    \begin{equation*}
        \left\langle \bG, \bH \right\rangle_{\origspace} = \left\langle \gmtTransVn \left[\bG\right], \gmtTransVn \left[\bH\right] \right\rangle_{\ambspace} \quad (\bG, \bH \in \Vn{n}{\blambda} \subseteq \origspace).
    \end{equation*}
\end{theorem}
\begin{proof}
    Linearity holds, because by~\eqref{eq:GMTtransfrom}, for any $\bG,\bH \in \Vn{n}{\blambda}$ and $w \in \C$ we have
    \begin{align*}
     \gmtTransVn \left[ \bG + w \bH \right] &= \gmtTransVn \left[ \sum_{k=0}^{n-1} \left\langle \bG, \gmtk \right\rangle_{\origspace} \gmtk + w  \sum_{k=0}^{n-1} \left\langle \bH, \gmtk \right\rangle_{\origspace}\gmtk \right] \\ 
      & =  \sum_{k=0}^{n-1} \left( \left\langle \bG, \gmtk \right\rangle_{\origspace} + w\left\langle \bH, \gmtk \right\rangle_{\origspace} \right) \gmtTransVn \left[ \gmtk \right]  \\ &= \sum_{k=0}^{n-1} \left( \left\langle \bG, \gmtk \right\rangle_{\origspace} + w\left\langle \bH, \gmtk \right\rangle_{\origspace} \right) \tgmtk \\ &= \sum_{k=0}^{n-1} \left\langle \bG, \gmtk \right\rangle_{\origspace} \tgmtk  + w  \sum_{k=0}^{n-1} \left\langle \bH, \gmtk \right\rangle_{\origspace} \tgmtk = \gmtTransVn \left[ \bG \right] + w\gmtTransVn \left[ \bH \right].
    \end{align*}
    The unitarity follows from the fact that $\gmtTransVn$ maps an orthonormal system in $\Vn{n}{\blambda}$ to one in $\hVn{n}{\gamma(\blambda)}$. Indeed, suppose that $\bG,\bH \in \Vn{n}{\blambda} \subseteq \origspace$. Then, 
    \begin{align*}
    \left\langle \gmtTransVn \left[\bG\right], \gmtTransVn \left[\bH\right] \right\rangle_{\ambspace} &= \left\langle \sum_{k=0}^{n-1} \left\langle \bG, \gmtk \right\rangle_{\origspace} \tgmtk,  \sum_{j=0}^{n-1} \left\langle \bH, \gmtj \right\rangle_{\origspace} \tgmtnj \right\rangle_{\ambspace} \\ &= \sum_{k=0}^{n-1}  \sum_{j=0}^{n-1} \left\langle \bG, \gmtk \right\rangle_{\origspace}\overline{\left\langle \bH, \gmtj \right\rangle}_{\origspace} \left\langle \tgmtk, \tgmtnj \right\rangle_{\ambspace} \\ &\stackrel{\text{\eqref{eq:GMTONS}}}{=} \sum_{k=0}^{n-1} \left\langle \bG, \gmtk \right\rangle_{\origspace}  \overline{\left\langle \bH, \gmtk \right\rangle}_{\origspace} \\&= \sum_{k=0}^{n-1} \sum_{j=0}^{n-1} \left\langle \bG, \gmtk \right\rangle_{\origspace} \overline{\left\langle \bH, \gmtj \right\rangle}_{\origspace} \left\langle \gmtk, \gmtj \right\rangle_{\origspace} = \left\langle \bG, \bH \right\rangle_{\origspace}.
    \end{align*}    
\end{proof}
Using \Cref{def:GMTtransform}, we can transform a \ac{mor} problem posed in an \ac{rkhs} $\origspace$ to another one, where we might be able to solve the problem more easily. To this end, we first define the concept of a weak model space from \Cref{def:weak} generated by the \ac{gmt} transformation. 

For the remainder of this study, we note that the  definition of the weak model spaces in \Cref{def:weak} can in a straight--forward manner be generalized to the case of repeated kernel centers. In the rest of this work, we assume that the unitary maps $\cF$ in this definition are chosen as \ac{gmt} transformations.


\begin{example}[\ac{mor} and the \ac{gmt} transformation]
Suppose we are given a \ac{fom} $\bG$ in the \ac{rkhs} $\origspace$ defined as
\begin{equation*}
	\origspace \coloneqq \left\{ \bF \in L^2(\Omega) \ : \ \bF' \in L^2(\Omega), \ \bF(0)^2 + \int_{0}^{1} \bF'(t)^2\, \du t < \infty \right\},
\end{equation*}
where $\Omega \coloneqq [0,1]$ and differentiation is understood in the weak sense. This is indeed an \ac{rkhs} of functions $\Omega \to \R$  with kernel
\begin{equation*}
\xi(t, \lambda) \coloneqq 1 + \min \{t, \lambda\} \quad (t, \lambda \in \Omega),
\end{equation*}
see, e.g.,~\cite[Section 1.2]{Wahba_1990}. This \ac{rkhs} contains first--order splines and also certain deep ReLU networks. For a deeper discussion on $\origspace$, we refer to \Cref{subsec:relu}. Here, we only note that for any $\bF \in \origspace$, we have
\begin{equation}
	\label{eq:snorm}
	\| \bF \|_{\origspace} \coloneqq \sqrt{ \bF(0)^2 + \int_{0}^{1} \bF'(t)^2 \, \du t }. 
\end{equation}
Let the \ac{fom} be defined as
\begin{equation}
	\label{eq:exfom}
	\bG \coloneqq \xi_{\frac{1}{4}} + \xi_{\frac{3}{4}},
\end{equation}
that is, we consider a \ac{fom} of order $n=2$ with kernel centers $\lambda_0 = \frac{1}{4}$ and $\lambda_1 = \frac{3}{4}$. Suppose that the \ac{rom} is of order $r=1$, i.e., 
\begin{equation*}
	\bG_r \coloneqq \tres_0 \xi_{\tlambda_0}
\end{equation*}
is an $\origspace$--norm--optimal approximation of $\bG$. We clearly cannot apply our optimality conditions from \Cref{sec:orthCond} to verify that $\bG_r$ fulfills at least the proposed necessary conditions of optimality. Indeed, each of the applicable conditions needs to evaluate the first--order derivative of $\bG_r$ (or at least of $\xi_{\tlambda_0}$) at $\tlambda_0$, which does not exist. Hence, it becomes necessary to transform the problem to an admissible \ac{rkhs} $\ambspace$.  

Since $\origspace$ is an \ac{rkhs}, we can apply \Cref{def:GMT} to write $\bG$ in a \ac{gmt} expansion. A quick calculation  using~\eqref{eq:snorm} and \Cref{def:GMT} verifies that the \ac{gmt} system defined by $\lambda_0$ and $\lambda_1$ can be written as
\begin{equation*}
	\gmt_0(t) = \frac{2}{\sqrt{5}} + \frac{2}{\sqrt{5}} \min \left\{t, \frac{1}{4} \right\} 
\end{equation*}
and
\begin{equation*}
	\gmt_1(t) = \sqrt{2} \min \left\{t, \frac{3}{4}\right\} - \sqrt{2} \min \left\{t, \frac{1}{4}\right\}.
\end{equation*}
By~\eqref{eq:exfom}, the reproducing property, and the fact that functions in $\origspace$ are real--valued, we obtain
\begin{equation*}
	\left\langle \bG, \gmt_0 \right\rangle_{\origspace} = \sqrt{5}, \quad \left\langle \bG, \gmt_1 \right\rangle_{\origspace} = \frac{1}{\sqrt{2}}.
\end{equation*}
Hence, we can write
\begin{equation*}
	\bG = \sqrt{5} \gmt_0 + \frac{1}{\sqrt{2}} \gmt_1.
\end{equation*}
The map $\gamma(t) = 2t - 1$
is a bijection between $\Omega$ restricted to $(0,1)$ and $\mathbb{I} = (-1, 1)$. This allows us to choose the real analytic space $\realHtwo$ from~\eqref{eq:realH2} as the companion space $\ambspace$. Recall that this is an \ac{rkhs} with kernel 
\begin{equation}
    \label{eq:kerh2I}
	\rho(x, \lambda) = \frac{1}{1 - x \lambda} \quad (x, \lambda \in \mathbb{I}).
\end{equation}
Using the \ac{gmt} transformation from \Cref{def:GMTtransform}, we obtain a surrogate \ac{fom} $\hat{\bG} \in \ambspace$ as 
\begin{equation*}
	\hat{\bG} = \gmtTransVn \left[\bG\right] = \sqrt{5} \tgmt_0 + \frac{1}{\sqrt{2}} \tgmt_1,
\end{equation*}
where $\gamma\big( \frac{1}{4} \big) = -\frac{1}{2}$, $\gamma \big( \frac{3}{4} \big) = \frac{1}{2}$, and by~\eqref{eq:MT},
\begin{equation*}
    \tgmt_0(t) = \frac{\sqrt{1 - \left| -\frac{1}{2}\right|^2}}{1 + \frac{t}{2}} = \frac{\sqrt{3}}{2 + t}, \quad \tgmt_1(t) = \frac{\sqrt{1 - \left| \frac{1}{2}\right|^2}}{1 - \frac{t}{2}} \frac{t + \frac{1}{2}}{1 + \frac{t}{2}} = \frac{2 \sqrt{3} t + \sqrt{3}}{4 - t^2}.
\end{equation*}
Now suppose that, using Algorithm~\ref{alg:hybrid_equilibrium}, we identified an $\realHtwo$--norm--optimal \ac{rom} to $\hat{\bG}$ denoted by
\begin{equation}
    \label{eq:GrHatExample}
	\hat{\bG}_r = c_0 \hat{\varphi}^{\gamma(\tilde{\blambda})}_0, 
\end{equation}
where by \Cref{def:GMT},
\begin{equation}
\label{eq:varphiHatExample}
\hat{\varphi}^{\gamma(\tilde{\blambda})}_0 = \frac{\rho_{\gamma(\tlambda_0)}} {\sqrt{\rho\big(\gamma(\tlambda_0), \gamma(\tlambda_0)\big)}}
\end{equation}
holds for some $\gamma(\tlambda_0) \in \mathbb{I}$. This is possible because the optimality conditions of \Cref{sec:orthCond} can be applied, as the kernel $\rho$ is sufficiently smooth in both variables. The norm optimality of $\hat{\bG}_r$ implies that $\hat{\bG}_r = \cP_{\hVn{1}{\gamma(\tilde{\blambda})}} \big[ \hat{\bG} \big]$, which by~\eqref{eq:kerh2I},~\eqref{eq:GrHatExample}, and~\eqref{eq:varphiHatExample} is equivalent to
\begin{equation*}
	c_0 = \sqrt{1 - \big| \gamma(\tlambda_0) \big|^2} \hat{\bG} \big( \gamma(\tlambda_0) \big).
\end{equation*}
A natural idea to obtain an optimal \ac{rom} to $\bG$ in $\origspace$ would now be to take 
\begin{equation*}
	\bG_r = \gmtTransVr^{-1} \big[ \hat{\bG}_r \big] = \underbrace{\sqrt{1 - \big| \gamma(\tlambda_0) \big|^2} \hat{\bG} \big( \gamma(\tlambda_0) \big)}_{=c_0} \rgmt_0.
\end{equation*}
Unfortunately, this choice cannot guarantee that $\left\| \bG - \bG_r \right\|_{\origspace} = \big\| \hat{\bG} - \hat{\bG}_r \big\|_{\ambspace}$, because the function system $\left\{ \hat{\varphi}^{\gamma(\tilde{\blambda})}_0, \tgmt_0, \tgmt_1 \right\}$ is not orthogonal in $\ambspace$ and therefore, the transformation mapping these to $\left\{ \rgmt_0, \gmt_0, \gmt_1 \right\}$ is no longer unitary between the subspaces $\origspace$ and $\ambspace$ that contain $\bG,\bG_r$ and $\hat{\bG}, \hat{\bG}_r$, respectively. Instead, we have to write $\hat{\bG}_r$ in a new \ac{gmt} basis generated by the kernel centers
\begin{equation*}
	\gamma(\boldsymbol{\mu}) = \left( \gamma(\lambda_0), \gamma(\lambda_1), \gamma(\tlambda_0) \right) \in \mathbb{I}^3 = \hat{\Omega}^3.
\end{equation*}
Using this basis, we can fully describe both $\hat{\bG}$ and $\hat{\bG}_r$ with the first three, now orthogonal, \ac{gmt} functions generated by $\gamma(\boldsymbol{\mu})$. Thus, the norm--optimal approximant to $\bG$ in $\origspace$ satisfying $\rkhsError{\xi}{r}(\bG_r) = \hat{\cE}_{\rho}^{r}\big(\hat{\bG}_r\big) \coloneqq \min_{\hat{\bH} \in \fG_{\rho}^r} \big\|\hat{\bG} - \hat{\bH}\big\|_{\hat{\cH}}$ is given by
\begin{equation}
	\label{eq:weakSolEx}
	\Gweak{r} \coloneqq \gmtTransVnr^{-1} \big[ \hat{\bG}_r \big] = \left\langle \hat{\bG}_r, \hat{\varphi}_{0}^{\gamma(\boldsymbol{\mu})} \right\rangle_{\ambspace} \varphi^{\boldsymbol{\mu}}_{0} +  \left\langle \hat{\bG}_r, \hat{\varphi}_{1}^{\gamma(\boldsymbol{\mu})} \right\rangle_{\ambspace} \varphi^{\boldsymbol{\mu}}_{1} + \left\langle \hat{\bG}_r, \hat{\varphi}_{2}^{\gamma(\boldsymbol{\mu})} \right\rangle_{\ambspace} \varphi^{\boldsymbol{\mu}}_{2}.
\end{equation}
Notice, however, that this $\Gweak{r}$ belongs to a weak solution space in $\origspace$ as defined in \Cref{def:weak} generated by the \ac{gmt} transformation. Indeed, nothing guarantees that the linear combination in~\eqref{eq:weakSolEx} belongs to a one--dimensional subspace in $\origspace$ spanned by a single kernel slice $\xi_{\tlambda}$ for some $\tlambda \in \Omega$. Below, we investigate how weak solutions can still be used to perform \ac{mor}, and how they can be approximated by strong counterparts belonging to $\romspaceo \subseteq \origspace$.
\end{example} 

For a \ac{mor} problem posed in a general \ac{rkhs} $\origspace$, we can choose an admissible $\ambspace$ such that its kernel $\hat{\xi}$ satisfies the differentiability conditions needed for \Cref{thm:orthocond}. More generally, suppose a \ac{fom} $\bG$ belongs to the $n$--dimensional subspace $\Vn{n}{\blambda} \subseteq \fomspaceo$ in the \ac{rkhs} $\origspace$. Suppose further that for some $\blambda \in {\Omega}^n$ we have $\bG = \gmtPsumn\left[ \bG \right]$, where the operator $\gmtPsumn$ is defined according to~\eqref{eq:GMTpartsum}. Applying the \ac{gmt} transformation defined by this tuple yields $\gmtTransVn \left[\bG \right] = \hat{\bG} \in \hVn{n}{\gamma(\blambda)} \subseteq \fomspaceg \subseteq \ambspace$. Suppose that we obtain a \ac{rom} $\hat{\bG}_r \in \hVn{r}{\gamma(\tilde{\blambda})} \subseteq \romspaceg$, where the first $r$ components of the tuple $\gamma(\tilde{\blambda})$ describe an $r$--dimensional subspace $\hVn{r}{\gamma(\tilde{\blambda})}\in \romspaceg \subseteq \ambspace$ (see \Cref{def:gmodspace}). 

In order to recover the corresponding \ac{rom} in $\origspace$, consider first the extended kernel center tuple
\begin{equation}
    \label{eq:extendedKerCenters}
    \boldsymbol{\mu} \coloneqq \left( \mu_0, \mu_1, \ldots, \mu_{n+r-1} \right) \coloneqq \left( \lambda_0, \ldots, \lambda_{n-1}, \tlambda_{0}, \ldots, \tlambda_{r-1} \right) \in \Omega^{n+r}.
\end{equation}
The kernel slices defined by the first $n+r$ components of $\boldsymbol{\mu}$ and $\gamma(\boldsymbol{\mu})$, respectively, define the $(n+r)$--dimensional subspaces
\begin{equation*}
        \Vn{n+r}{\boldsymbol{\mu}} = \myspan \left\{ \gmtkj  \right\}_{k=0}^{n+r-1} \subseteq \origspace, \quad
        \hVn{n+r}{\gamma(\boldsymbol{\mu})} = \myspan \left\{ \tgmtkj  \right\}_{k=0}^{n+r-1} \subseteq \ambspace.
\end{equation*}

Notice that we have $\bG \in \Vn{n+r}{\boldsymbol{\mu}}$ and $\hat{\bG}, \hat{\bG}_r, \hat{\bG} - \hat{\bG}_r \in \hVn{n+r}{\gamma(\boldsymbol{\mu})}$. Due to the ordering of $\boldsymbol{\mu}$ in~\eqref{eq:extendedKerCenters}, the first $n$ functions of these systems span $\Vn{n}{\blambda}$ and $\hVn{n}{\gamma(\blambda)}$. Using \Cref{def:GMTtransform}, we can define a unitary \ac{gmt} transformation between $\Vn{n+r}{\boldsymbol{\mu}}$ and $\hVn{n+r}{\gamma(\boldsymbol{\mu})}$ denoted by $\gmtTransVnr$. Again, because of the ordering chosen in~\eqref{eq:extendedKerCenters}, restricting this operator to $\Vn{n}{\blambda}$ yields $\gmtTransVn$. Thus, we have
\begin{equation}
    \label{eq:sameBasis}
        \hat{\bG} = \gmtTransVn \left[ \bG \right] = \gmtTransVnr \left[ \bG \right], \quad \bG = \gmtTransVn^{-1} \big[ \hat{\bG} \big] = \gmtTransVnr^{-1} \big[ \hat{\bG} \big].
\end{equation}
We can now define the corresponding \ac{rom} in $\origspace$ as
\begin{equation}
    \label{eq:weakROM}
    \Gweak{r} \coloneqq \gmtTransVnr^{-1} \big[ \hat{\bG}_r \big].
\end{equation}
Consider the subspace 
\begin{equation*}
    \hVn{r}{\gamma(\tilde{\blambda})} \coloneqq \myspan\left\{ \tgmtj \right\}_{j=0}^{r-1} = \myspan \left\{  \hat{\eta}_j \right\}_{j=0}^{r-1} \subseteq \ambspace,
\end{equation*}
where $\left\{ \hat{\eta}_j \right\}_{j=0}^{r-1}$ are defined according to~\eqref{eq:kers} in the companion \ac{rkhs} $\ambspace$. Furthermore, define the subspace $\Vn{r}{\tilde{\blambda}} \subseteq \origspace$ spanned by $\left\{ \eta_j \right\}_{j=0}^{r-1}$ with corresponding \ac{gmt} basis $\left\{ \rgmtj \right\}_{j=0}^{r-1}$. Unfortunately, restricting $\gmtTransVnr$ to $\Vn{r}{\tilde{\blambda}}$ does not yield the \ac{gmt} transformation $\gmtTransVr$. Indeed, the \ac{rom} defined in~\eqref{eq:weakROM} belongs to the $r$--dimensional subspace
\begin{equation*}
    W_r \coloneqq \myspan \left\{ \gmtTransVnr^{-1} \left[ \tgmtj \right]  \right\}_{j=0}^{r-1} \subseteq \origspace.
\end{equation*}
In general, functions in $W_r$ cannot be expressed as an $r$--term linear combination of kernel slices $\eta_j$ for $j=0,\ldots,r-1$. On the other hand, \Cref{thm:findimRKHS} guarantees that the $r$--dimensional $\origspace$ subspace $W_r \subseteq \wromspaceo$ is an \ac{rkhs} in its own right; therefore, its elements, in particular the \ac{rom} $\Gweak{r}$, can be expressed as
\begin{equation}
    \label{eq:weakROMnativeBasis}
    \Gweak{r} = \sum_{j=0}^{r-1} \wtres_j \xi_{\tlambda_j}^{W_r} 
\end{equation}
where $\xi^{W_r}$ denotes the kernel associated with the \ac{rkhs} $W_r$ and $\wtres_j \in \C$, $\tlambda_j \in \Omega$ for $j=0,\ldots,r-1$. We summarize the steps of \ac{gmt}--transformation--based \ac{mor} in \acp{rkhs}, which can be used to obtain a norm--optimal weak solution:
\begin{enumerate}
    \item Given \ac{fom} $\bG \in \fomspaceo$, companion \ac{rkhs} $\ambspace$ and reduced order $r$, write $\bG$ in a \ac{gmt} expansion, i.e.,
    \begin{equation*}
        \bG = \sum_{k=0}^{n-1} c_k \gmtk \in \Vn{n}{\blambda}.
    \end{equation*}
    \item Obtain corresponding \ac{fom} in $\ambspace$, i.e., 
    \begin{equation*}
        \hat{\bG} = \gmtTransVn [ \bG ] = \sum_{k=0}^{n-1} c_k \tgmtk \in \hVn{n}{\gamma(\blambda)}.
    \end{equation*}
    \item Use an appropriate \ac{mor} method in the companion space to obtain \ac{rom} kernel centers $\gamma\big(\tlambda_0\big), \ldots, \gamma(\tlambda_{r-1}) \in \hat{\Omega}$. Suppose that the \ac{rom} is given by $\hat{\bG}_r = \gmtPsumrT \big[ \hat{\bG} \big]$ and it satisfies one of the optimality conditions given in \Cref{sec:orthCond}.
    \item Construct an extended \ac{gmt} basis spanning $\big\{ \tgmtkj \big\}_{k=0}^{n+r-1}$ denoted by $\hVn{n+r}{\gamma(\boldsymbol{\mu})}$. The first $n$ functions in this system match the \ac{gmt} functions spanning $\hVn{n}{\gamma(\blambda)}$ to ensure condition~\eqref{eq:sameBasis}.
    \item Obtain a weak \ac{rom} $\bG_r$ according to~\eqref{eq:weakROM}. 
    \item If needed, compute the kernel $\xi^{W_r}$. This is possible using a feature map expression from \Cref{def:features}, because the functions $\left\{\gmtTransVnr^{-1} \left[\tgmtj\right]\right\}_{j=0}^{r-1}$ form a basis in $W_r$. Finally, write $\Gweak{r}$ according to~\eqref{eq:weakROMnativeBasis}.
\end{enumerate}
Then, by \Cref{thm:unitary}, we have $\left\| \bG - \Gweak{r} \right\|_{\origspace} = \big\| \hat{\bG} - \hat{\bG}_r \big\|_{\ambspace}$ with $\hat{\bG}_r$ satisfying one of the optimality conditions given in \Cref{sec:orthCond}. 

One might be tempted to reverse the order of kernel centers in~\eqref{eq:extendedKerCenters}. In this case, we obtain~\eqref{eq:sameBasis} for the \ac{rom}, instead of the \ac{fom}. However, this choice leads to the issue that usually
\begin{equation*}
\bG \neq \gmtTransVnr^{-1}  \underbrace{\left[ \gmtTransVn\left[ \bG \right] \right]}_{\hat{\bG}}.
\end{equation*}
Consequently, $\left\| \bG - \bG_r \right\|_{\origspace} = \big\| \hat{\bG} - \hat{\bG}_r \big\|_{\ambspace}$ would no longer hold, where $\bG_r \coloneqq \gmtTransVr^{-1} \big[ \hat{\bG}_r \big]$.

Finally, we note that for some special choices of $\origspace$, the weak and strong model spaces seem to coincide under the \ac{gmt} transformation pipeline. That is, the above transformation pipeline always produces $\Gweak{r} \in \romspaceo$. This seems to be the case, for example, if $\origspace=H_2(\D)$ and $\ambspace = H_2(\C^{+})$. In the current study, we refrain from pursuing sufficient conditions for this to hold in general. We note, however, that if $\origspace$ and $\ambspace$ are both $H_2$ Hardy--Hilbert spaces, the \ac{gmt} transformation can alternatively be realized using a conformal map between their domains, instead of the basis--matching construction of \Cref{def:GMTtransform}. In this case, weak and strong solutions are guaranteed to coincide.

Practically, if an application requires a strong \ac{rom} in $\romspaceo$, one way to proceed is to project the weak solution to $\Vn{r}{\tilde{\blambda}} = \myspan\left\{ \rgmtj \right\}_{j=0}^{r-1} \subseteq \origspace$. The next theorem clarifies error bounds on this projection.
\begin{theorem}[Error of approximate strong models]
    \label{thm:ProjErr}
    Suppose \Cref{ass:GMTtrans} is satisfied, and let $\Vn{n}{\blambda}$, $\Vn{r}{\tilde{\blambda}}$, $\hVn{n}{\gamma(\blambda)}$, and $\hVn{r}{\gamma(\tilde{\blambda})}$ be built as in \eqref{eq:defVn}. Suppose that $\Vn{n}{\blambda} \cap \Vn{r}{\tilde{\blambda}} = \{ 0 \}$. Let $\dotplus$ denote the direct sum and define the subspaces
    \begin{equation*}
        \Vn{n+r}{\boldsymbol{\mu}} \coloneqq \Vn{n}{\blambda} \dotplus \Vn{r}{\tilde{\blambda}} = \myspan \left\{ \gmtkj \right\}_{k=0}^{n+r-1}, \quad \hVn{n+r}{\gamma(\boldsymbol{\mu})} \coloneqq \hVn{n}{\gamma(\boldsymbol{\lambda})} \dotplus \hVn{r}{\gamma(\tilde{\blambda})} = \myspan \left\{ \tgmtkj \right\}_{k=0}^{n+r-1}
    \end{equation*}
    along with
    \begin{equation*}
        \left(\Vn{n}{\blambda}\right)^{\perp} \coloneqq \Vn{n+r}{\boldsymbol{\mu}} \ominus \Vn{n}{\blambda}, \quad  \left(\hVn{n}{\gamma(\blambda)}\right)^{\perp} \coloneqq \hVn{n+r}{\gamma(\boldsymbol{\mu})} \ominus \hVn{n}{\gamma(\blambda)},
    \end{equation*}
    and the \ac{gmt} transformations
    \begin{equation*}
        \gmtTransVnr : \Vn{n+r}{\boldsymbol{\mu}} \to \hVn{n+r}{\gamma(\boldsymbol{\mu})}, \quad \gmtTransVn : \Vn{n}{\boldsymbol{\lambda}} \to \hVn{n}{\gamma(\boldsymbol{\lambda})}
    \end{equation*}
    generated by $\boldsymbol{\mu} \coloneqq \big(\lambda_0, \ldots, \lambda_{n-1}, \tlambda_{0}, \ldots, \tlambda_{r-1} \big)$ and $\blambda \coloneqq \big( \lambda_0, \ldots, \lambda_{n-1} \big)$, respectively. Then, the following are true:
    \begin{enumerate}[label=(\alph*)]
    \item\label{it:a} The operators
        \begin{equation}
            \label{eq:opers}
            \begin{aligned}
             & \cQ \coloneqq \cP_{\Vn{n}{\blambda}} \circ \left(\cP_{\big(\Vn{n}{\blambda}\big)^{\perp}}\Big|_{\Vn{r}{\tilde{\blambda}}}\right)^{-1}  : \left( \Vn{n}{\blambda} \right)^{\perp} \to \Vn{n}{\blambda}, \\ & \hat{\cQ} \coloneqq \cP_{\hVn{n}{\gamma(\blambda)}} \circ \left( \cP_{\left( \hVn{n}{\gamma(\blambda)}\right)^{\perp}}\Big|_{\hVn{r}{\gamma(\tilde{\blambda})}} \right)^{-1}  \ : \ \left( \hVn{n}{\gamma(\blambda)} \right)^{\perp} \to \hVn{n}{\gamma(\blambda)}
            \end{aligned}
        \end{equation} 
        are well defined.
    \item\label{it:b} For all
    \begin{equation*}
        \Fweak{r} \in W_r \coloneqq \myspan \left\{ \gmtTransVnr^{-1} \big[ \hat{\xi}_{\gamma(\tlambda_0)}\big], \ldots, \gmtTransVnr^{-1} \big[ \hat{\xi}_{\gamma(\tlambda_{r-1})}\big] \right\},
    \end{equation*}
    with $\gmtTransVn^{\perp} \coloneqq \gmtTransVnr\big|_{\left(\Vn{n}{\blambda}\right)^{\perp}}$, we have
    \begin{align}
        \label{eq:VrErr}
        \begin{split}
        \operatorname{dist} \left( \Fweak{r}, \Vn{r}{\tilde{\blambda}} \right) & \coloneqq \inf_{\bV \in \Vn{r}{\tilde{\blambda}}} \left\| \Fweak{r} - \bV \right\|_{\origspace} \\ &\leq \left\| \hat{\cQ} \circ \gmtTransVn^{\perp} - \gmtTransVn \circ \cQ \right\|_{\cL\left(\left( \Vn{n}{\blambda} \right)^{\perp}, \hVn{n}{\gamma(\blambda)}\right)} \left\| \Fweak{r} \right\|_{\cH},
        \end{split}
    \end{align}
    \end{enumerate}
\end{theorem}
\begin{proof}
We first observe that the \ac{gmt} transformation $\gmtTransVnr$ is a block--diagonal linear operator. Indeed, because of the ordering given by $\boldsymbol{\mu}$ determining the Gram--Schmidt steps, we have 
\begin{equation*}
\Vn{n}{\blambda} = \myspan \left\{ \gmtmu_0, \ldots, \gmtmu_{n-1} \right\} = \myspan \left\{ \gmt_0, \ldots, \gmt_{n-1} \right\}    
\end{equation*}
 and 
 \begin{equation*}
 \left(\Vn{n}{\blambda}\right)^{\perp} = \myspan \left\{ \gmtmu_{n}, \ldots, \gmtmu_{n+r-1} \right\}.    
 \end{equation*}
 An analogous observation can be made for $\hVn{n}{\gamma(\blambda)}$ and $\left(\hVn{n}{\gamma(\blambda)}\right)^{\perp}$. Hence, 
 \begin{equation*}
     \gmtTransVnr = \begin{bmatrix}
         \gmtTransVn & 0 \\
         0 & \gmtTransVn^{\perp}
     \end{bmatrix}
 \end{equation*}
 with respect to the orthogonal splittings
 \begin{equation*}
     \Vn{n+r}{\boldsymbol{\mu}} = \Vn{n}{\blambda} \oplus \left( \Vn{n}{\blambda}\right)^{\perp} \quad \textrm{and} \quad \hVn{n+r}{\gamma(\boldsymbol{\mu})} = \hVn{n}{\gamma(\boldsymbol{\lambda})} \oplus \left( \hVn{n}{\gamma(\boldsymbol{\lambda})} \right)^{\perp}.
 \end{equation*}
 It follows that $\gmtTransVn : \Vn{n}{\blambda} \to \hVn{n}{\gamma(\blambda)}$ and $\gmtTransVn^{\perp} : \left(\Vn{n}{\blambda}\right)^{\perp} \to \left(\hVn{n}{\gamma(\blambda)}\right)^{\perp}$ are unitary.

 We now show statement~\ref{it:a}, i.e., the existence of the operators $\cQ$ and $\hat{\cQ}$ in~\eqref{eq:opers}. Let $\bV_0 \in \Vn{r}{\tilde{\blambda}}$ and suppose $\cP_{\left( \Vn{n}{\blambda} \right)^{\perp}} [ \bV_0 ] = 0$. Since we can write $\bV_0$ as
 \begin{equation*}
     \bV_0 = \cP_{\Vn{n}{\blambda}} \left[ \bV_0 \right] + \underbrace{\cP_{\left(\Vn{n}{\blambda}\right)^{\perp}} \left[ \bV_0 \right]}_{=0} = \cP_{\Vn{n}{\blambda}} \left[ \bV_0 \right],
 \end{equation*}
 this gives $\bV_0 \in \Vn{n}{\blambda}$. On the other hand, by the  assumption $\Vn{r}{\tilde{\blambda}} \cap \Vn{n}{\blambda} = \{ 0 \}$, hence, $\bV_0 = 0$ follows. This shows that $\cP_{\left( \Vn{n}{\blambda} \right)^{\perp}} \Big|_{\Vn{r}{\tilde{\blambda}}}$ is injective. However, since $\dim \Vn{r}{\tilde{\blambda}} = \dim  \left(\Vn{n}{{\blambda}}\right)^{\perp} = r$, the restricted orthogonal projector $\cP_{\left( \Vn{n}{\blambda} \right)^{\perp}} \Big|_{\Vn{r}{\tilde{\blambda}}}$ is also bijective and therefore, $\cQ$ in~\eqref{eq:opers} is well--defined. An analogous argument can be used to show the existence of $\hat{\cQ}$.

Let us show statement~\ref{it:b}. We make use of the fact that any function $\bV \in \Vn{r}{\tilde{\blambda}}$ can be written as $\cP_{\Vn{n}{\blambda}} \left[ \bV \right] + \cP_{\left(\Vn{n}{\blambda}\right)^{\perp}} \left[ \bV \right]$. This, together with~\eqref{eq:opers}, and the bijective nature of $\cP_{\left( \Vn{n}{\blambda} \right)^{\perp}} \Big|_{\Vn{r}{\tilde{\blambda}}}$  gives
\begin{equation}
    \label{eq:vr}
    \Vn{r}{\tilde{\blambda}} = \left\{ \cQ [ \bH ] + \bH \ : \ \bH \in \left(\Vn{n}{\blambda}\right)^{\perp} \right\}.
\end{equation}
Observe that for each $\bH \in \left( \Vn{n}{\blambda} \right)^\perp$, there exists a $\bV \in \Vn{r}{\tilde{\blambda}}$ such that $\cQ [ \bH ] = \cP_{\Vn{n}{\blambda}} [ \bV ]$ and $\bH = \cP_{\left(\Vn{n}{\blambda}\right)^{\perp}} [ \bV ]$. By the same argument in the companion space, we also obtain 
\begin{equation}
\label{eq:vrhat}
\hVn{r}{\gamma(\tilde{\blambda})} = \left\{ \hat{\cQ}\big[ \hat{\bH} \big] + \hat{\bH} \ : \ \hat{\bH} \in \left( \hVn{n}{\gamma(\blambda)} \right)^{\perp}\right\}.
\end{equation}

Next, let $\Fweak{r} \in W_r$ be fixed. Then we also have $\Fweak{r} = \cP_{\Vn{n}{\blambda}} [ \Fweak{r} ] + \cP_{\left(\Vn{n}{\blambda}\right)^{\perp}} [ \Fweak{r} ]$. In addition, let $\hFweak{r} \coloneqq \gmtTransVnr [ \Fweak{r} ] \in \hVn{r}{\gamma(\tilde{\blambda})}$. Exploiting these observations and the block--diagonal structure $\gmtTransVnr^{-1} = \operatorname{diag}\bigg(\gmtTransVn^{-1}, \left( \gmtTransVn^{\perp}\right)^{-1}\bigg)$, we can express the set $W_r = \gmtTransVnr^{-1} \Big[ \hVn{r}{\gamma(\tilde{\blambda})} \Big]$ as
\begin{equation*}
    W_r = \left\{ \gmtTransVn^{-1} \big[ \hat{\bF} \big] + \left( \gmtTransVn^{\perp} \right)^{-1} \big[ \hat{\bF} \big] \ : \ \hat{\bF} \in \hVn{r}{\gamma(\tilde{\blambda})} \right\}.
\end{equation*}
Considering~\eqref{eq:vrhat}, we find that for all $\hat{\bF} \in \hVn{r}{\gamma(\tilde{\blambda})}$, we have
\begin{equation*}
    \hat{\bF} = \underbrace{\hat{\cQ}\big[ \hat{\bH}\big]}_{\in \hVn{n}{\gamma(\blambda)}} + \underbrace{\hat{\bH}}_{\in \left( \hVn{n}{\gamma(\blambda)} \right)^{\perp}}
\end{equation*}
for some $\hat{\bH} \in \left(\hVn{n}{\gamma(\blambda)}\right)^{\perp}$. Accordingly, we can write
\begin{equation*}
    \Fweak{r} = \gmtTransVnr^{-1} \big[ \hFweak{r} \big] = \left(\gmtTransVnr^{-1}  \circ \hat{\cQ} \right) \big[ \hat{\bH}_0 \big] + \gmtTransVnr^{-1} \big[ \hat{\bH}_0 \big]
\end{equation*}
for some $\hat{\bH}_0 \in \left(\hVn{n}{\gamma(\blambda)}\right)^{\perp}$. By exploiting the block diagonal nature of $\gmtTransVnr$, this gives
\begin{equation}
    \label{eq:firstBg}
    \Fweak{r} = \left( \gmtTransVn^{-1}  \circ \hat{\cQ} \right)\big[ \hat{\bH}_0 \big]  + \left(\gmtTransVn^{\perp} \right)^{-1} \big[ \hat{\bH}_0 \big].
\end{equation}
Now recall that $\hat{\bH}_0 \in \left(\hVn{n}{\gamma(\blambda)}\right)^{\perp}$ and so $\bH_0 \coloneqq \gmtTransVnr^{-1} \big[ \hat{\bH}_0\big] = \left(\gmtTransVn^{\perp}\right)^{-1} \big[ \hat{\bH}_0 \big] \in \left(\Vn{n}{\blambda} \right)^\perp$. Applying this substitution to~\eqref{eq:firstBg} yields
\begin{equation*}
    \Fweak{r} = \left( \gmtTransVn^{-1}  \circ \hat{\cQ}  \circ \gmtTransVn^{\perp} \right) [ \bH_0 ]   + \bH_0.
\end{equation*}
We find that we can describe the subspace $W_r$ by
\begin{equation}
     \label{eq:wr}
     W_r = \left\{ \left( \gmtTransVn^{-1}  \circ \hat{\cQ}  \circ \gmtTransVn^{\perp} \right) [ \bH ]  + \bH \ : \ \bH \in \left( \Vn{n}{\blambda} \right)^{\perp} \right\}.
\end{equation}
We are now ready to give the estimate in~\eqref{eq:VrErr}. Recall, that for $\Fweak{r} \in W_r$, the distance from $\Vn{r}{\tilde{\blambda}}$ is defined as
\begin{equation*}
    \operatorname{dist} \left( \Fweak{r}, \Vn{r}{\tilde{\blambda}} \right) \coloneqq \inf_{\bV \in \Vn{r}{\tilde{\blambda}}} \left\| \Fweak{r} - \bV \right\|_{\origspace}.
\end{equation*}
Hence, any particular $\bV_0 \in \Vn{r}{\tilde{\blambda}}$ satisfies
\begin{equation}
    \label{eq:overest}
     \operatorname{dist}\left( \Fweak{r}, \Vn{r}{\tilde{\blambda}} \right) \leq \left\| \Fweak{r} - \bV_0 \right\|_{\origspace}.
\end{equation}
We can now exploit~\eqref{eq:vr} and~\eqref{eq:wr} by observing that $\Fweak{r} = \left( \gmtTransVn^{-1}  \circ \hat{\cQ}  \circ \gmtTransVn^{\perp} \right) [ \bH_0 ] + \bH_0$ for some $\bH_0 \in \left( \Vn{n}{\blambda} \right)^{\perp}$ and choosing $\bV_0 = \cQ [ \bH_0 ] + \bH_0 \in \Vn{r}{\tilde{\blambda}}$. This allows us to write
\begin{align*}
    \Fweak{r} - \bV_0 &= \left(\gmtTransVn^{-1}  \circ \hat{\cQ}  \circ \gmtTransVn^{\perp} \right) [ \bH_0 ]  + \bH_0 - \cQ [ \bH_0 ] - \bH_0 \\ &= \left( \gmtTransVn^{-1}  \circ \hat{\cQ}  \circ \gmtTransVn^{\perp} \right) [ \bH_0 ] - \cQ [ \bH_0 ] \in \Vn{n}{\blambda},
\end{align*}
and therefore,
\begin{equation*}
    \left\| \Fweak{r} - \bV_0 \right\|_{\origspace} = \left\| \left(\gmtTransVn^{-1}  \circ \hat{\cQ}  \circ \gmtTransVn^{\perp} \right) [ \bH_0 ]  - \cQ [ \bH_0 ] \right\|_{\cH}.
\end{equation*}
To obtain the formula, we can exploit that $\gmtTransVn = \gmtTransVnr\Big|_{\Vn{n}{\blambda}}$ and that $\gmtTransVn$ and $\gmtTransVnr$ are unitary and find
\begin{align*}
    \left\| \Fweak{r} - \bV_0 \right\|_{\origspace} &= \left\| \gmtTransVn \left[\Fweak{r} - \bV_0\right] \right\|_{\ambspace} \\ &= \left\| \gmtTransVn \left[\left( \gmtTransVn^{-1} \circ \hat{\cQ}  \circ \gmtTransVn^{\perp}\right) [ \bH_0 ]  - \cQ [ \bH_0 ] \right] \right\|_{\ambspace} \\ &= \left\| \left(\hat{\cQ}  \circ \gmtTransVn^{\perp}\right) [ \bH_0 ]  - \left(\gmtTransVn  \circ \cQ \right) [ \bH_0 ]  \right\|_{\ambspace}.
\end{align*}
This, together with~\eqref{eq:overest} gives
\begin{align*}
    \operatorname{dist}\left( \Fweak{r}, \Vn{r}{\tilde{\blambda}} \right) & \leq \left\| \Fweak{r} - \bV_0 \right\|_{\origspace} \\&= \left\| \left(\hat{\cQ}  \circ \gmtTransVn^{\perp}\right) [ \bH_0 ] - \left( \gmtTransVn  \circ \cQ \right) [ \bH_0 ]   \right\|_{\ambspace} \\ &\leq \left\| \hat{\cQ}  \circ \gmtTransVn^{\perp}  - \gmtTransVn  \circ \cQ  \right\|_{\cL\left(\left( \Vn{n}{\blambda} \right)^{\perp}, \hVn{n}{\gamma(\blambda)}\right)} \left\| \bH_0 \right\|_{\origspace} \\ &\leq \left\| \hat{\cQ}  \circ \gmtTransVn^{\perp}  - \gmtTransVn  \circ \cQ  \right\|_{\cL\left(\left( \Vn{n}{\blambda} \right)^{\perp}, \hVn{n}{\gamma(\blambda)}\right)} \left\| \Fweak{r} \right\|_{\origspace},
\end{align*}
where the last inequality follows from~\eqref{eq:wr}, which we can apply to $\bH_0$ to obtain $\left\| \bH_0 \right\|_{\origspace} = \left\| \cP_{\left(\Vn{n}{\blambda}\right)^{\perp}} [ \Fweak{r} ] \right\|_{\origspace} \leq \left\| \Fweak{r} \right\|_{\origspace}$.
 
\end{proof}

\begin{corollary}
    \label{thm:weakProj}
    Suppose the assumptions of \Cref{thm:ProjErr}. Then, for any weak \ac{rom} $\Gweak{r} \in W_r$, the infimum in $\operatorname{dist}\left( \Gweak{r}, \Vn{r}{\tilde{\blambda}}\right)$ is attained and the best approximation in $\Vn{r}{\tilde{\blambda}}$ is given by $\cP_{\Vn{r}{\tilde{\blambda}}} \big[ \Gweak{r} \big]$. In addition,
    \begin{align}
        \label{eq:weakProjRom}
        \begin{split}
        \operatorname{dist}\left( \Gweak{r}, \Vn{r}{\tilde{\blambda}}\right) &= \left\| \Gweak{r} - \cP_{\Vn{r}{\tilde{\blambda}}} \big[ \Gweak{r} \big]\right\|_{\origspace} \\ &\leq \left\| \hat{\cQ}  \circ \gmtTransVn^{\perp}   - \gmtTransVn  \circ \cQ   \right\|_{\cL\left(\left( \Vn{n}{\blambda} \right)^{\perp}, \hVn{n}{\gamma(\blambda)}\right)} \left\| \Gweak{r} \right\|_{\cH}.
        \end{split}
    \end{align}
\end{corollary}
\begin{proof}
    The attainment of the infimum and the best approximation are consequences of the Hilbert projection theorem, while~\eqref{eq:weakProjRom} is an application of \Cref{thm:ProjErr}.
\end{proof}

The previous statement bounds the discrepancy between the weak solution and its projection to $\Vn{r}{\tilde{\blambda}}$. To obtain this projected solution, one has to first compute the weak \ac{rom} given by $\Gweak{r} = \gmtTransVnr^{-1}\big[ \hat{\bG}_r \big]$, and then project it to $\Vn{r}{\tilde{\blambda}}$. This is computationally inefficient, because finding $\gmtTransVnr^{-1}\big[ \hat{\bG}_r \big]$ involves generating the first $n+r$ terms of the \ac{gmt} system in the companion space described by $\gamma(\boldsymbol{\mu})$. In the experiments, we found that it is much more economical to find the kernel centers $\gamma\big(\tlambda_0\big), \ldots, \gamma\big(\tlambda_{r-1}\big)$ describing the \ac{rom} $\hat{\bG}_r$ in the companion space and then consider the projected \ac{rom} 
\begin{equation}
    \label{eq:dirProjection}
    \Gweakproj{r} \coloneqq \cP_{\Vn{r}{\tilde{\blambda}}} \left[ \bG \right]
\end{equation}
in the original \ac{rkhs} $\origspace$. In addition to being computationally more efficient, by the Hilbert projection theorem, we also have
\begin{equation*}
    \left\| \bG - \cP_{\Vn{r}{\tilde{\blambda}}}\left[ \bG \right] \right\|_{\origspace} \leq  \left\| \bG - \left( \gmtTransVn^{-1} \circ \cP_{\hVn{r}{\gamma(\tilde{\blambda})}}  \right) \big[ \hat{\bG} \big] \right\|_{\origspace},
\end{equation*}
that is, the economical projection~\eqref{eq:dirProjection} is guaranteed to have an error to the \ac{fom} $\bG$ that is lower than or equal to that of the projected weak solution. We bound the $\origspace$ error of~\eqref{eq:dirProjection} with the following theorem.
\begin{theorem}[Error of projected \ac{rom}]
    \label{thm:ProjErrEco}
    Suppose the assumptions of \Cref{thm:ProjErr} hold and let $\Gweakproj{r}$ be defined according to~\eqref{eq:dirProjection}. Then, 
    \begin{equation*}
        \left\| \bG - \Gweakproj{r} \right\|_{\origspace} \leq \rkhsError{\xi}{r}\big( \Gweak{r} \big) + \left\| \hat{\cQ}  \circ \gmtTransVn^{\perp}   - \gmtTransVn  \circ \cQ  \right\|_{\cL\left(\left( \Vn{n}{\blambda} \right)^{\perp}, \hVn{n}{\gamma(\blambda)}\right)} \left\| \Gweak{r} \right\|_{\cH}.
\end{equation*}
    where $\Gweak{r} = \gmtTransVnr^{-1} \big[ \hat{\bG}_r \big]$ and $\rkhsError{\xi}{r}(\cdot)$ as  in \Cref{def:Objective}.
\end{theorem}
\begin{proof}
    For ease of notation, denote 
    the projection of $\Gweak{r}$ to $\Vn{r}{\tilde{\blambda}}$ by $\bG_r^{\mathrm{weak\text{-}proj}}$.
    By the Hilbert projection theorem and since $\Gweakproj{r} = \cP_{\Vn{r}{\tilde{\blambda}}} \left[ \bG \right]$, we know that 
    \begin{equation}
        \label{eq:ecoProjIsBetter}
        \left\| \bG - \Gweakproj{r} \right\|_{\origspace} \leq \big\| \bG - \bG_r^{\mathrm{weak\text{-}proj}} \big\|_{\origspace}.
    \end{equation}
     On the other hand, since the \ac{gmt} transformation is unitary, we have
     \begin{equation*}
         \rkhsError{\xi}{r}\big( \Gweak{r} \big) = \left\| \bG - \Gweak{r} \right\|_{\origspace} = \left\| \hat{\bG} - \hat{\bG}_r \right\|_{\ambspace} = \hat{\cE}_{\hat{\xi}}^{r}\big( \hat{\bG}_r \big),
     \end{equation*}
     and by \Cref{thm:ProjErr} and \Cref{thm:weakProj},
     \begin{equation*}
         \big\| \Gweak{r} - \bG_r^{\mathrm{weak\text{-}proj}} \big\|_{\origspace} \leq \left\| \hat{\cQ}  \circ \gmtTransVn^{\perp}   - \gmtTransVn  \circ \cQ   \right\|_{\cL\left(\left( \Vn{n}{\blambda} \right)^{\perp}, \hVn{n}{\gamma(\blambda)}\right)} \left\| \Gweak{r} \right\|_{\cH}.
     \end{equation*}
     Hence, the triangle inequality gives
     \begin{equation*}
         \big\| \bG - \bG_r^{\mathrm{weak\text{-}proj}} \big\|_{\origspace} \leq \rkhsError{\xi}{r} \big(\Gweak{r} \big) + \left\| \hat{\cQ}  \circ \gmtTransVn^{\perp}   - \gmtTransVn  \circ \cQ   \right\|_{\cL\left(\left( \Vn{n}{\blambda} \right)^{\perp}, \hVn{n}{\gamma(\blambda)}\right)} \left\| \Gweak{r} \right\|_{\cH}.
     \end{equation*}
     Combining this with~\eqref{eq:ecoProjIsBetter} gives the statement.
\end{proof}

\Cref{fig:trfPipeline} summarizes the \ac{gmt} transformation pipeline introduced above and the subspaces in which the \ac{fom} and the two considered \acp{rom} reside. The \ac{fom} $\bG \in \Vn{n}{\blambda} \subseteq \origspace$ is mapped via the unitary \ac{gmt} transformation $\gmtTransVn$ to $\hat{\bG} \in \hVn{n}{\gamma(\blambda)} \subseteq \ambspace$, where \ac{mor} is carried out to obtain the companion \ac{rom} $\hat{\bG}_r \in \hVn{r}{\gamma(\tilde{\blambda})}$. The kernel centers describing $\hat{\bG}_r$ determine the subspace $\Vn{r}{\tilde{\blambda}} \subseteq \origspace$ by $\gamma^{-1}$. The subspace $\Vn{r}{\tilde{\blambda}}$ contains the projected \ac{rom} $\Gweakproj{r} = \cP_{\Vn{r}{\tilde{\blambda}}} \left[ \bG \right]$ from~\eqref{eq:dirProjection}. Applying instead the extended inverse transformation $\gmtTransVnr^{-1}$ directly to $\hat{\bG}_r$ yields the weak \ac{rom} $\Gweak{r} \in W_r$ from~\eqref{eq:weakROM}. Note that $\hVn{n+r}{\gamma(\boldsymbol{\mu})} = \hVn{n}{\gamma(\blambda)} \dotplus \hVn{r}{\gamma(\tilde{\blambda})}$ is only a (not necessarily orthogonal) direct sum. This is also true for $\Vn{n+r}{\boldsymbol{\mu}} = \Vn{n}{\blambda} \dotplus \Vn{r}{\tilde{\blambda}}$.

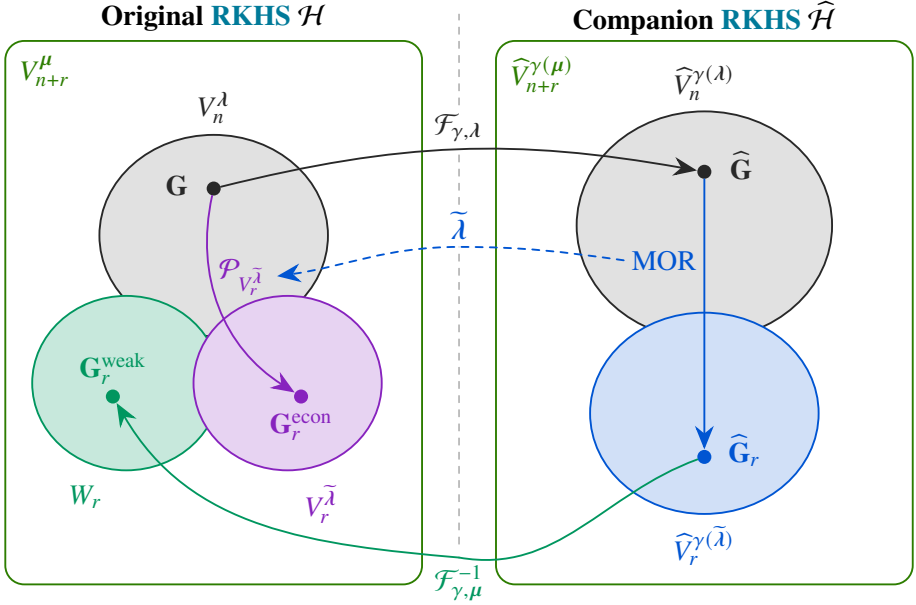
\begin{figure}[tbp]
    \centering
    \resizebox{\linewidth}{!}{
\usetikzlibrary{arrows.meta}
\definecolor{fomgray}{RGB}{40,40,40}
\definecolor{weakgreen}{RGB}{0,150,95}
\definecolor{econpurple}{RGB}{135,35,190}
\definecolor{ambblue}{RGB}{0,85,210}
\definecolor{spacegreen}{RGB}{45,125,0}

\begin{tikzpicture}[>={Stealth[length=10pt, width=8pt]}, line cap=round, line join=round,
    every node/.style={text opacity=1},
    elabel/.style={font=\large, inner sep=3pt}]

\draw[spacegreen, thick, rounded corners=8pt] (-0.3,-0.9) rectangle (5.9,7.2);
\draw[spacegreen, thick, rounded corners=8pt] (7.0,-0.9) rectangle (13.2,7.2);

\node[font=\bfseries\large] at (2.8,7.55) {Original \ac{rkhs} $\origspace$};
\node[font=\bfseries\large] at (10.1,7.55) {Companion \ac{rkhs} $\ambspace$};

\node[spacegreen, font=\large, anchor=north west] at (-0.2,7.1) {$\Vn{n+r}{\boldsymbol{\mu}}$};
\node[spacegreen, font=\large, anchor=north west] at (7.1,7.1) {$\hVn{n+r}{\gamma(\boldsymbol{\mu})}$};


\filldraw[fomgray, fill=fomgray!14, thick] (2.8,4.3) ellipse [x radius=1.7, y radius=1.5];
\filldraw[weakgreen, fill=weakgreen!22, thick] (1.5,2.1) ellipse [x radius=1.4, y radius=1.3];
\filldraw[econpurple, fill=econpurple!20, thick] (3.9,2.1) ellipse [x radius=1.4, y radius=1.3];

\node[fomgray, font=\large, anchor=south] at (2.8,5.85) {$\Vn{n}{\blambda}$};
\node[weakgreen, font=\large, anchor=north] at (0.9,0.7) {$W_r$};
\node[econpurple, font=\large, anchor=north] at (4.4,0.7) {$\Vn{r}{\tilde{\blambda}}$};

\node[circle, fill=fomgray, inner sep=2.1pt] (G) at (2.8,5.0) {};
\node[fomgray, font=\large] at ($(G)+(-0.55,0.05)$) {$\bG$};

\node[circle, fill=weakgreen, inner sep=2.1pt] (Gweak) at (1.3,1.9) {};
\node[weakgreen, font=\large] at ($(Gweak)+(0,0.42)$) {$\Gweak{r}$};

\node[circle, fill=econpurple, inner sep=2.1pt] (Gecon) at (4.1,1.9) {};
\node[econpurple, font=\large] at ($(Gecon)+(0,-0.45)$) {$\Gweakproj{r}$};

\filldraw[fomgray, fill=fomgray!14, thick] (10.1,4.45) ellipse [x radius=1.9, y radius=1.7];
\filldraw[ambblue, fill=ambblue!18, thick] (10.1,1.65) ellipse [x radius=1.7, y radius=1.5];

\node[fomgray, font=\large, anchor=south] at (10.1,6.2) {$\hVn{n}{\gamma(\blambda)}$};
\node[ambblue, font=\large, anchor=north] at (10.1,0.1) {$\hVn{r}{\gamma(\tilde{\blambda})}$};

\node[circle, fill=fomgray, inner sep=2.1pt] (Ghat) at (10.1,5.25) {};
\node[fomgray, font=\large] at ($(Ghat)+(0.55,0.1)$) {$\hat{\bG}$};

\node[circle, fill=ambblue, inner sep=2.1pt] (Ghatr) at (10.1,1.0) {};
\node[ambblue, font=\large] at ($(Ghatr)+(0.6,0.1)$) {$\hat{\bG}_r$};

\draw[->, thick, fomgray] (G) to[bend left=12] node[midway, above, elabel] (Flab) {$\gmtTransVn$} (Ghat);

\draw[->, thick, ambblue] (Ghat) -- node[pos=0.3, left, elabel, align=right] (mor) {MOR} (Ghatr);

\coordinate (Fwinv) at (6.45,-0.5);
\draw[weakgreen, thick] (Ghatr) to[out=200, in=-10] (Fwinv);
\draw[->, thick, weakgreen] (Fwinv) to[out=175, in=-60] (Gweak);
\node[weakgreen, elabel, anchor=north] at (Fwinv) {$\gmtTransVnr^{-1}$};

\draw[->, thick, econpurple] (G) to[bend right=35] node[pos=0.35, right, elabel] (proj) {$\cP_{\Vn{r}{\tilde{\blambda}}}$} (Gecon);

\coordinate (lamX) at (6.45,4.13);
\draw[->, thick, dashed, ambblue, shorten >=2pt] (mor.west) to[out=172, in=0] (lamX) to[out=180, in=8] (proj.east);
\node[ambblue, elabel, anchor=south] (lamlab) at (lamX) {$\tilde{\blambda}$};

\draw[dashed, gray] (6.45,-0.3) -- (lamX);
\draw[dashed, gray] (lamlab.north) -- (lamX |- Flab.south);
\draw[dashed, gray] (lamX |- Flab.north) -- (6.45,7.6);

\end{tikzpicture}
    }
    \caption{The \ac{gmt} transformation pipeline. The \ac{fom} $\bG$ and its companion counterpart $\hat{\bG}$ are shown in gray/black; the companion \ac{rom} $\hat{\bG}_r$ obtained by \ac{mor} in blue. Inverting the \ac{gmt} transformation directly gives the weak \ac{rom} $\Gweak{r} \in W_r$ (green), while projecting $\bG$ onto the subspace $\Vn{r}{\tilde{\blambda}}$ determined by the kernel centers of $\hat{\bG}_r$ (via $\gamma^{-1}$) gives the projected \ac{rom} $\Gweakproj{r}$ (purple).}
    \label{fig:trfPipeline}
\end{figure}

Before using the \ac{gmt} transformation to show some existence theorems of the \ac{mor} problem in \acp{rkhs}, we show a more practical approach to compute it. In particular, we directly obtain \ac{gmt} coefficients of $\bG$ given the kernel centers $\lambda_k$ and linear weights $\res_k$ without having to compute the \ac{gmt} functions themselves.

\begin{theorem}[\ac{gmt} transformation without basis change]
\label{thm:BGMTtransfrom}
Suppose that \Cref{ass:GMTtrans} is satisfied and let $(\lambda_0, \ldots, \lambda_{n-1}) \in \Omeganeq{n}$. Let
\begin{align*}
\bG = \sum_{k=0}^{n-1} \res_k \xi_{\lambda_k}
\in \Vn{n}{\blambda}
\subseteq \origspace, \quad 
\gmtTransVn\left[\bG\right]
= \hat{\bG}
= \sum_{k=0}^{n-1} \hat{\res}_k
\hat{\xi}_{\gamma(\lambda_k)}
\in \hVn{n}{\gamma(\blambda)}
\subseteq \ambspace.
\end{align*}
Using the biorthogonal basis from \Cref{def:Lagrange}, let 
\begin{equation*}
\bc =  
[\res_0,\res_1,\ldots,\res_{n-1}]^\top
\coloneqq 
\left[
\left\langle \bG,\psi_0\right\rangle_{\origspace},
\ldots,
\left\langle \bG,\psi_{n-1}\right\rangle_{\origspace}
\right]^\top
\in \C^n .
\end{equation*}
Moreover, let $\boldsymbol{\Xi} = \bL\bL^*$ and $\hat{\boldsymbol{\Xi}} = \hat{\bL}\hat{\bL}^*$ be the Cholesky factorizations of the Gramian matrices
\begin{equation*}
\boldsymbol{\Xi} \coloneqq  
\left[\left\langle
\xi_{\lambda_k},
\xi_{\lambda_j}
\right\rangle_{\origspace}\right]_{k,j=0}^{n-1},
\quad
\hat{\boldsymbol{\Xi}}
\coloneqq
\left[\left\langle
\hat{\xi}_{\gamma(\lambda_k)},
\hat{\xi}_{\gamma(\lambda_j)}
\right\rangle_{\ambspace}\right]_{k,j=0}^{n-1}.
\end{equation*}
Then, with $\frak{P} \coloneqq
\hat{\bL}^{-\top}\bL^\top
\in \C^{n\times n}$, the linear coefficients of the \ac{rom} are given by
\begin{equation}
\label{eq:BGMTtransform}
\hat{\res}_k = \left[ \frak{P} \bc \right]_k \quad (k=0,\ldots,n-1).
\end{equation}
\end{theorem}

\begin{proof}
For all $k=0,\ldots,n-1$, we have $\gmt_k\in \Vn{n}{\blambda}$ and
\begin{equation}
\label{eq:GMTwithKerSlices}
\gmtk
=
\sum_{j=0}^{k} r_{k,j}\xi_{\lambda_j}.
\end{equation}
Hence,
\begin{equation*}
\left\langle \bG,\gmtk \right\rangle_{\origspace}
=
\left\langle
\bG,
\sum_{j=0}^{k} r_{k,j}\xi_{\lambda_j}
\right\rangle_{\origspace} 
=
\sum_{j=0}^{k}
\overline{r_{k,j}}
\left\langle
\bG,\xi_{\lambda_j}
\right\rangle_{\origspace} 
=
\sum_{j=0}^{k}
\overline{r_{k,j}}\bG(\lambda_j).
\end{equation*}
The same relation holds between $\tgmtk$ and
$\hat{\xi}_{\gamma(\lambda_k)}$, namely,
\begin{equation*}
\tgmtk
=
\sum_{i=0}^{k}
\hat{r}_{k,i}
\hat{\xi}_{\gamma(\lambda_i)}.
\end{equation*}
Substituting these expressions into \Cref{def:GMTtransform}, we obtain
\begin{align*}
\gmtTransVn \left[\bG \right]
&=
\sum_{k=0}^{n-1}
\left\langle\bG,\gmtk\right\rangle_{\origspace}
\tgmtk \\
&=
\sum_{k=0}^{n-1}
\left(
\sum_{j=0}^{k}
\overline{r_{k,j}}\bG(\lambda_j)
\right)
\left(
\sum_{i=0}^{k}
\hat{r}_{k,i}
\hat{\xi}_{\gamma(\lambda_i)}
\right) 
=
\sum_{i=0}^{n-1}
\left(
\sum_{j=0}^{n-1}
\left(
\sum_{k=0}^{n-1}
\hat{r}_{k,i}
\overline{r_{k,j}}
\right)
\bG(\lambda_j)
\right)
\hat{\xi}_{\gamma(\lambda_i)},
\end{align*}
where we use that $r_{k,i} = 0$ whenever $i > k$. Define 
\begin{equation*}
\frak{B}
\coloneqq
\left[\sum_{k=0}^{n-1}
\hat{r}_{k,i}\overline{r_{k,j}}\right]_{i,j=0}^{n-1} \in\C^{n\times n}.
\end{equation*}
Then
\begin{equation*}
\gmtTransVn\left[\bG\right]
=
\sum_{k=0}^{n-1}
[\frak{B}\tilde{\bc}]_k
\hat{\xi}_{\gamma(\lambda_k)}, \quad \text{where} \quad \tilde{\bc}
\coloneqq
\left[\bG(\lambda_0),\ldots,\bG(\lambda_{n-1}) \right]^\top.
\end{equation*}
Next, we show that
\begin{equation*}
    \frak{B}
    =
    \hat{\bL}^{-\top}\overline{\bL}^{-1},
\end{equation*}
where $\hat{\bL}$ and $\bL$ are obtained from the Cholesky factorizations of $\hat{\boldsymbol{\Xi}}$ and $\boldsymbol{\Xi}$, respectively. The Cholesky factorizations exist because the Gram matrices $\hat{\boldsymbol{\Xi}}$ and $\boldsymbol{\Xi}$ are positive definite under our assumptions.

Introducing the vector of basis functions $\Phi
    =
    \left[\gmt_0,\ldots,\gmt_{n-1}\right]^\top$, by~\eqref{eq:GMTwithKerSlices} we obtain $\Phi = \bR \left[\xi_{\lambda_0},\ldots,\xi_{\lambda_{n-1}}\right]^\top,$
where $\bR = [r_{k,j}]_{k,j=0}^{n-1}$.
By \Cref{def:GMT}, the diagonal coefficients of $\bR$ satisfy
 $   r_{k,k}
    =\left\|
    \xi_{\lambda_k}-\cS_k^{\blambda}\left[\xi_{\lambda_k}\right]
    \right\|_{\origspace}^{-1}>0$ for $k=0,\ldots,n-1$.
Since
$\left\langle\gmt_k,\gmt_j\right\rangle_{\origspace}
=\delta_{j,k}$. This gives
\begin{equation*}
    \bI_n = \left[ \left\langle\gmt_k,\gmt_j\right\rangle_{\origspace} \right]_{k,j=0}^{n-1} = \bR \left[ \left\langle \xi_{\lambda_k},\xi_{\lambda_j}\right\rangle_{\origspace} \right]_{k,j=0}^{n-1}  \bR^* = \bR\Xi\bR^*.
\end{equation*}
Substituting the Cholesky factorization $\boldsymbol{\Xi}=\bL\bL^*$ yields $\bR\bL\bL^* \bR^*
= (\bR\bL)(\bR\bL)^* = \bI_n$. Thus, $\bR\bL$ is unitary. Since both $\bR$ and $\bL$ are lower triangular, $\bR\bL$ is lower triangular. Moreover, since $r_{k,k}>0$ by the Gram--Schmidt construction and $[\bL]_{k,k}>0$ by the construction of the Cholesky factorization, its diagonal entries are $ [\bR\bL]_{k,k}
    =    r_{k,k}[\bL]_{k,k}$ are positive for $k=0,\dots,n-1$. Therefore, $\bR\bL=\bI_n$, i.e.,  $  \bR=\bL^{-1}$.
Analogously, $\hat{\bR}=\hat{\bL}^{-1}$.
By the definition of $\frak{B}$, we have $\frak{B} = \hat{\bL}^{-\top}\overline{\bL}^{-1} = 
    \hat{\bR}^\top 
    \overline{\bR}$.
It remains to relate the evaluation coefficients $\tilde{\bc}$ to the kernel coefficients $\bc$. By the reproducing property,
\begin{equation*}
    \tilde c_j \coloneqq [\tilde{\bc}]_j
    = \bG(\lambda_j) = \left\langle
        \bG,\xi_{\lambda_j}
    \right\rangle_{\origspace}.
\end{equation*}
Since
    $\bG = \sum_{k=0}^{n-1} \res_k\xi_{\lambda_k}$,
it follows that
\begin{equation*}
    \tilde c_j = \sum_{k=0}^{n-1} \res_k \left\langle \xi_{\lambda_k}, \xi_{\lambda_j} \right\rangle_{\origspace} = \sum_{k=0}^{n-1} \boldsymbol{\Xi}_{k,j}\res_k \quad (j=0,\ldots,n-1).
\end{equation*}
Therefore, $
    \tilde{\bc}
    =
    \boldsymbol{\Xi}^\top \bc, 
    $
and using $\Xi=\bL\bL^*$, we obtain
    $\tilde{\bc}
    =
    \boldsymbol{\Xi}^\top \bc
    =
    \overline{\bL}\bL^\top\bc$.
Combining this with the expression for $\frak{B}$ gives $
    \hat{\bc}
    =  \frak{B}\tilde{\bc}
    = \hat{\bL}^{-\top}
    \overline{\bL}^{-1}
    \overline{\bL}\bL^\top \bc
    = \hat{\bL}^{-\top}\bL^\top \bc$,
which proves~\eqref{eq:BGMTtransform}.
\end{proof}
\begin{remark}
    While \Cref{thm:BGMTtransfrom} allows for the computation of $\hat{\bG} = \gmtTransVn \left[ \bG \right]$ without having to express $\bG$ in a \ac{gmt} basis in $\origspace$, or generate the corresponding \ac{gmt} basis in $\ambspace$, it does not remove all of the computational difficulty for \ac{gmt} transformations. In particular, if the kernel slices defined by $\gamma(\lambda_0), \ldots, \gamma(\lambda_{n-1})$ are nearly linearly dependent in the companion \ac{rkhs} $\ambspace$, then inverting $\hat{\bL}$ from \Cref{thm:BGMTtransfrom} can still pose a challenge for large $n$. However, if one can identify a family of \acp{rkhs} where such a computation remains stable, even if the points $\gamma(\lambda_k)$ for $k=0,\ldots,n-1$ are clustered in $\hat{\Omega}$, then one could transform \ac{mor} problems to these spaces without further numerical issues. We do not further pursue the rigorous identification of such nicely behaved \acp{rkhs}; however, based on our presented experiments in \Cref{sec:exp}, we empirically find several candidate spaces.
\end{remark}

Finally, we clarify the existence of local minimum points of $\rkhsError{\xi}{r}$ in \acp{rkhs} with sufficiently smooth kernels. 

\begin{theorem}
    \label{thm:ExistLocalOpt}
    Let $\Omega \subseteq \C$ be a compact set and $\origspace$ be an \ac{rkhs} of functions $\Omega \to \C$ with kernel $\xi$. Let $\bG \in \fomspaceo$, and assume $\left(\lambda \mapsto \xi_{\lambda}\right) \in C^r(\Omega, \origspace)$ for some $r < n$. Finally, suppose that $\funder[j]{\xi_{\lambda}}{\olambda} \in \origspace$ for $j=0,\ldots,r-1$ and all $\lambda \in \Omega$. Then, there exists $\tilde{\blambda}^{\mathrm{opt}} \in \Omega^r$ such that
    \begin{equation*}
        \rkhsError{\xi}{r}\big( \cP_{\Vn{r}{\tlambda^{\mathrm{opt}}}}[\bG] \big) = \min_{\tilde{\blambda} \in \Omega^r} \rkhsError{\xi}{r}\big( \cP_{\Vn{r}{\tilde{\blambda}}}[\bG] \big).
    \end{equation*}
    Moreover, if $\tilde{\blambda}^{\mathrm{opt}}$ lies in the interior of $\Omega^r$, then $\tilde{\blambda}^{\mathrm{opt}}$ is a zero of the equilibrium function given in \Cref{def:EquibFun}. 
\end{theorem}
\begin{proof}
  We show that the map $\tilde{\blambda} \mapsto \cP_{\Vn{r}{\tilde{\blambda}}} [\bG]$ is continuous with respect to the norm induced on $\cH$. Recall that by \Cref{thm:orthoProjGMT}, we have $\cP_{\Vn{r}{\tilde{\blambda}}} = \gmtPsumr$.
  This shows that if $\tilde{\blambda}$ has no repeated kernel centers, then there exists a neighborhood of $\tilde{\blambda}$ where continuity is given by the continuity of the inner product on $\origspace$.  Otherwise, if $\tilde{\blambda} \in \Omega^r$ has repeated kernel centers, without loss of generality, assume that these repetitions are grouped. Formally, we assume that
  \begin{equation*}
   \tilde{\blambda} = \big( \underbrace{\tlambda_0,\ldots,\tlambda_0}_{m_0 \text{ times}}, \ldots, \underbrace{\tlambda_{k-1},\ldots,\tlambda_{k-1}}_{m_{k-1}\text{ times}}\big)
  \end{equation*}

  
  Consider the corresponding divided--difference basis
  
  \begin{align}
    \label{eq:divDiffBasis}
         \begin{split}
             \tau\big[\tlambda_j\big] &= \xi_{\tlambda_j} \quad (j = 0,\ldots, r-1) \\
            \tau\left[\tlambda_j, \ldots, \tlambda_{j+q}\right] &\coloneqq \begin{cases}
             \frac{\tau\left[\tlambda_{j+1}, \ldots, \tlambda_{j+q}\right] - \tau\left[\tlambda_j, \ldots, \tlambda_{j+q-1}\right]}{\tlambda_{j+q} - \tlambda_j}, & \text{if } \tlambda_j \neq \tlambda_{j+q}, \\
             \frac{1}{q!} \funder[q]{\xi_{\tlambda_j}}{\overline{\tlambda}}, & \text{if } \tlambda_j = \tlambda_{j+q} 
         \end{cases} \quad (j + q \leq r-1).
         \end{split}
  \end{align}
  Notice first~\cite[Proposition 35]{deBoor_divided_2005} that
  \begin{equation*}
    \Vn{r}{\tilde{\blambda}} = \myspan\left\{ \tau\left[\tlambda_0\right], \tau\left[\tlambda_0, \tlambda_1\right], \ldots, \tau\left[\tlambda_0, \ldots, \tlambda_{r-1}\right] \right\} = \myspan \left\{  \rgmtj  \right\}_{j=0}^{r-1} \subseteq \origspace
  \end{equation*}
  for any $\tilde{\blambda} \in \Omega^r$, where $\gmtj$ denote the \ac{gmt} functions constructed according to \Cref{def:GMT}. Consequently, $\tau\left[\tlambda_0\right], \ldots, \tau\left[\tlambda_0, \ldots, \tlambda_{r-1}\right]$ form a linearly independent system. The Hermite--Genocchi formula for divided differences (see, e.g.,~\cite[Theorem 3.3]{atkinson_introduction_2008} or~\cite{deBoor_divided_2005}) asserts the  continuity of $\tau\left[\tlambda_i, \ldots, \tlambda_{i+k}\right]$, even with repeated kernel centers. Then, we can express the projection operator $\cP_{\Vn{r}{\tilde{\blambda}}}$ in terms of the basis in~\eqref{eq:divDiffBasis}, i.e.,
  \begin{equation*}
    \cP_{\Vn{r}{\tilde{\blambda}}}[\cdot] = \sum_{k,j=0}^{r-1} \left[\bU 
    \big(\tilde{\blambda} \big)^{-1}\right]_{k,j} \left\langle \cdot, \tau\left[\tlambda_0, \ldots, \tlambda_{j}\right] \right\rangle_{\origspace} \tau\left[\tlambda_0, \ldots, \tlambda_k\right],
  \end{equation*}
  where $\bU \big(\tilde{\blambda} \big)$ denotes the Gramian matrix
  \begin{equation*}
      \bU 
    \big(\tilde{\blambda} \big) = \left[\left\langle  \tau\big[\tlambda_0, \ldots, \tlambda_k\big],  \tau\big[\tlambda_0, \ldots, \tlambda_j\big] \right\rangle_{\origspace}\right]_{k,j = 0}^{r-1}. 
  \end{equation*}
  We observe that:
  \begin{enumerate}
      \item The scalar--valued functions $\left[\bU (\cdot)^{-1}\right]_{k,j}$, $k,j=0,\ldots,r-1$, are continuous.
      \item The operators $\cT_{k,j}\big(\tilde{\blambda}\big)[\cdot] \coloneqq \left\langle \cdot, \tau\left[\tlambda_0, \ldots, \tlambda_{j}\right] \right\rangle_{\origspace} \tau\left[\tlambda_0, \ldots, \tlambda_k\right]$, $k,j=0,\ldots,r-1$, are continuous in the operator norm. Indeed, we have 
      \begin{equation*}
          \big\| \cT_{k,j}\big(\tilde{\blambda}\big) - \cT_{k,j}\big(\hat{\blambda}\big) \big\|_{\cL(\origspace)} \to 0
      \end{equation*}
      as $\tilde{\blambda} \to \hat{\blambda}$ (in any $p$--vector norm). This follows directly from the continuity of the divided differences $\tau$ with respect to the $\origspace$ norm.
  \end{enumerate}
  Consequently, the map $\tilde{\blambda} \mapsto \cP_{\Vn{r}{\tilde{\blambda}}}$ is continuous in the operator norm on $\Omega^r$, so is $\tilde{\blambda} \mapsto \rkhsError{\xi}{r}\big( \cP_{\Vn{r}{\tilde{\blambda}}}[\bG] \big)$. Since $\Omega^r$ is compact, the extreme value theorem guarantees the existence of $\tilde{\blambda}^{\mathrm{opt}} \in \Omega^r$ attaining $\min_{\tilde{\blambda} \in \Omega^r} \rkhsError{\xi}{r}\big( \cP_{\Vn{r}{\tilde{\blambda}}}[\bG] \big)$, which proves the first statement.

  For the second statement, suppose $\tilde{\blambda}^{\mathrm{opt}}$ lies in the interior of $\Omega^r$. Then there exists a $\delta > 0$ such that the open ball of radius $\delta$ around $\tilde{\blambda}^{\mathrm{opt}}$ is contained in $\Omega^r$, so $\bG_r^{\mathrm{opt}} \coloneqq \cP_{\Vn{r}{\tlambda^{\mathrm{opt}}}}[\bG]$ is a locally optimal \ac{rom} in the sense of \Cref{def:locmin}. By \Cref{thm:equibCond}, this yields $\cD\big(\tilde{\blambda}^{\mathrm{opt}}\big) = 0$.
\end{proof}

\section{MOR methods}
\label{sec:irka}

We now propose \ac{mor} algorithms in \acp{rkhs}. First, we consider a method for \ac{mor} problems posed in \acp{rkhs} that satisfy the conditions of Theorems~\ref{thm:orthocond} and~\ref{thm:MLcRKHS}. Then, we use \ac{gmt} transformations to solve the \ac{mor} problem in \acp{rkhs} with kernels that are not sufficiently smooth.

To clarify the setting of our proposed method, we make the following assumption. 
\begin{assumption}
    \label{ass:IRKAass} 
    Suppose that $\origspace$ is an \ac{rkhs} of functions $\Omega \to \C$ with kernel $\xi$ such that $\xi \in C^2(\Omega^2,\C)$.
\end{assumption}
Under Assumption~\ref{ass:IRKAass}, we propose a norm--optimal \ac{rkhs}--\ac{mor} method in Algorithm~\ref{alg:irkaRKHS}. Here, we assume that the initial and all intermediate kernel centers are pairwise different. If that was not the case, we would need higher regularity in the second kernel variable. 

\begin{algorithm}[tbp]
\caption{$\origspace$--optimal MOR in \acp{rkhs} with $C^2$ kernels}
\begin{algorithmic}[1]
\Require \ac{rkhs} $\origspace$ with kernel $\xi$, \ac{fom} $\bG \in \fomspaceo$, reduced order $r < n$, initial \ac{rom} kernel centers $\tilde{\blambda}^{(0)} \in \Omeganeq{r}$, termination tolerance $\varepsilon > 0$.
\Ensure \ac{rom} $\bG_r \in \romspaceo$.
\State $m \gets 0$.
\State $\varepsilon^{(0)}\gets  \infty$.
\While{$\varepsilon^{(m)}> \varepsilon$}
    \State $\bJ \gets \left[\frac{\partial \cD \left( \tilde{\blambda}^{(m)} \right)_i}{\partial \overline{\tlambda_j}} \right]_{i,j=0}^{r-1}$. \Comment{$\cD$ as in \Cref{def:EquibFun}.}
    \State $\tilde{\blambda}^{(m+1)} \gets \tilde{\blambda}^{(m)} -\bJ^{-1} \cD \left( \tilde{\blambda}^{(m)} \right)$.
    \State $\varepsilon^{(m)} \gets \left\| \tilde{\blambda}^{(m+1)} - \tilde{\blambda}^{(m)} \right\|_{\infty}. $
    \State $m \gets m+1$.
\EndWhile
\State $\bG_r \gets \cP_{\Vn{r}{\tilde{\blambda}^{(m)}}}\left[ \bG \right]$.
\State \Return $\bG_r$.
\end{algorithmic}
\label{alg:irkaRKHS}
\end{algorithm}
\begin{remark}
    \begin{enumerate}[label=(\alph*)]
        \item Algorithm~\ref{alg:irkaRKHS} is a Newton--like iteration to find a zero of $\cD$ in \Cref{def:EquibFun}. The Jacobian $\bJ$ in Algorithm~\ref{alg:irkaRKHS} is well--defined due to Assumption~\ref{ass:IRKAass}. One could potentially use a zero finding scheme that does not require the computation of $\bJ$ such as the secant method. In this case, \Cref{ass:IRKAass} can be weakened to only assume $\xi \in C \big( \Omega^2, \origspace\big)$.
        \item If upon convergence $\bJ$ is not singular, we satisfy \Cref{thm:equibCond}. 
        \item \Cref{thm:ExistLocalOpt} guarantees the existence of a local minimizer $\tilde{\blambda}^{\mathrm{opt}} \in \Omega^r$ of $\tilde{\blambda} \mapsto \rkhsError{\xi}{r} \big( \cP_{V_r^{\tilde{\blambda}}}[\cdot] \big)$, and, provided $\tilde{\blambda}^{\mathrm{opt}}$ lies in the interior of $\Omega^r$, that it is a zero of $\cD$. This can only be applied if we extend Assumption~\ref{ass:IRKAass} such that the kernel $\xi$ is $r$--times continuously differentiable with respect to the second variable.
    \end{enumerate}
\end{remark}
While Algorithm~\ref{alg:irkaRKHS} performs well under Assumption~\ref{ass:IRKAass}, in practice it has several limitations. First, it can only be applied in \acp{rkhs} equipped with a sufficiently smooth kernel. Several classes of \acp{fom}, however, cannot be embedded into generalized model spaces of such \acp{rkhs}. For example, ReLU networks considered in \Cref{sec:exp} do not meet this requirement. In addition, similarly to classical \ac{irka}~\cite{gugercin_mathcalh_2_2008}, the update steps in Algorithm~\ref{alg:irkaRKHS} do not guarantee that $\tilde{\blambda}$ remains inside $\Omega^r$ throughout the iteration. In the $H_2(\D)$ setting, this problem is addressed in~\cite{cdcAngino}. The proposed solution modifies the Newton update in Algorithm~\ref{alg:irkaRKHS} to be compatible with the hyperbolic geometry induced by the pseudohyperbolic metric on the manifold $\D$. In particular, the Newton update is performed along hyperbolic geodesics, which force the parameter vector $\tilde{\blambda}$ to remain inside $\D^r$ in each iteration of the method. 

For the above reasons, we propose a \ac{gmt}--transformation--based \ac{mor} method, that can be applied to produce "norm--optimal"  weak solutions. 
\begin{assumption}
    \label{ass:CDCRKHS}
    Let $\Omega \subseteq \C$ such that there exists a bijection $\gamma : \Omega \to \D$. Let $\origspace$ be an \ac{rkhs} of functions $\Omega \to \C$ with reproducing kernel $\xi$.
\end{assumption}
A general \ac{mor} method applicable to \acp{fom} defined in \acp{rkhs} with non--smooth kernels is given in Algorithm~\ref{alg:cdcRKHS}. Here, we assume Assumption~\ref{ass:CDCRKHS}.
\begin{algorithm}[btp]
\caption{\ac{gmt} transformation based \ac{rkhs}--\ac{mor} method.}
\begin{algorithmic}[1]
\Require \ac{rkhs} $\origspace$ with kernel $\xi$, \ac{fom} $\bG \in \fomspaceo$ defined by the kernel centers $\blambda \in \Omega^n$, reduced order $r < n$, initial \ac{rom} kernel centers $\tilde{\blambda}^{(0)} \in \Omega^r$, bijection $\gamma : \Omega \to \D$.
\Ensure weak \ac{rom} $\Gweak{r} \in \wromspaceo$.
\State $\hat{\bG} \gets \sum_{k=0}^{n-1} \left\langle \bG, \gmtk \right\rangle_{\origspace} \tgmtk \in H_2(\D)$. \Comment{\ac{gmt} transformation.}
\State $\hat{\bG}_r \gets$ output of \Cref{alg:hybrid_equilibrium} with inputs $\hat{\bG}$, $r$, and $\gamma\left(\tilde{\blambda}^{(0)}\right)$.
\State $\Gweak{r} \gets \gmtTransVnr^{-1} \big[ \hat{\bG}_r \big] \in \wromspaceo$. \Comment{$\boldsymbol{\mu}$ as in~\eqref{eq:extendedKerCenters}.}
\State \Return $\Gweak{r}$.
\end{algorithmic}
\label{alg:cdcRKHS}
\end{algorithm}

\begin{remark}
    \begin{enumerate}[label=(\alph*)]
        \item By \Cref{thm:unitary}, Algorithm~\ref{alg:cdcRKHS} guarantees that 
        \begin{equation*}
        \rkhsError{\xi}{r} \big( \Gweak{r} \big) =  \left\| \bG - \Gweak{r} \right\|_{\origspace} = \big\| \hat{\bG} - \hat{\bG}_r \big\|_{H_2(\D)} = \hat{\cE}_{\rho}^{r} \big( \hat{\bG}_r \big),
        \end{equation*}
        where $\hat{\bG}_r$ satisfies the optimality condition in~\eqref{eq:MLC} in $H_2(\D)$ and $\rho$ denotes the $H_2(\D)$ kernel in~\eqref{eq:kerh2d}. In this way, norm--optimal reduced--order models can be obtained in \acp{rkhs}, where Theorems~\ref{thm:orthocond} and~\ref{thm:MLcRKHS} cannot be applied. 
        \item In general, $\Gweak{r} \in \wromspaceo$ (see \Cref{def:weak}). We note, however, that 
        \begin{equation*}
            \Gweak{r} = \sum_{j=0}^{r-1} c_j \gmtTransVnr^{-1} \left[ \tgmtj \right] = \sum_{k=0}^{n+r-1} \left\langle \hat{\bG}, \tgmtkj \right\rangle_{H_2(\D)} \gmtkj
        \end{equation*}
         and that we can compute the coefficients $c_j$,  $j=0,\ldots,r-1$, given 
         \begin{equation*}
            \left\langle \hat{\bG}, \tgmtkj \right\rangle_{H_2(\D)} \quad  (k=0,\ldots,n+r-1).     
         \end{equation*}
         In this way, $\Gweak{r}$ can be fully recovered given the kernel centers of the \ac{fom} $\blambda$, the kernel centers of the \ac{rom} $\tilde{\blambda}$ and the coefficients $\left(c_j\right)_{j=0}^{r-1}$. Hence, we need $n + 2\cdot r$ parameters to reconstruct a weak \ac{rom}. Hence, $\Gweak{r}$ achieves a compression of the \ac{fom} as long as $r < n/2$, since the \ac{fom} is represented by $2n$ parameters in total ($n$ kernel centers $\blambda$ and $n$ linear coefficients $\res$).
        \item The requirement for the existence of $\gamma : \Omega \to \D$ can be weakened to $\gamma : \Omega \to \hat{\Omega} \subseteq \C$, if we suppose the existence of an \ac{rkhs} $\ambspace$ with a sufficiently smooth kernel such that \Cref{thm:orthocond} can be applied. In this case, we can consider a \ac{gmt} transformation $ \gmtTransVn\left[ \bG \right] = \hat{\bG} \in \ambspace$ and obtain a norm--optimal \ac{rom} $\hat{\bG}_r \in \romspaceg$ using Algorithm~\ref{alg:irkaRKHS}. 

        Note however, that in Algorithm~\ref{alg:irkaRKHS}, no formal guarantees are given that in each iteration, the \ac{rom} parameters $\tlambda_j^{(m)} \ (j=0,\ldots,r-1, \ m=0,1,\ldots)$ remain strictly inside $\hat{\Omega}$. In this regard, we note that if $\hat{\Omega}$ is a Riemannian manifold, then the Newton step in Algorithm~\ref{alg:irkaRKHS} can be modified to respect the geometry induced on it. However, such an approach would require the development of the Riemannian Newton step for each specific choice of $\hat{\Omega}$. If a bijection $\gamma : \Omega \to \D$  is available, then using Algorithm~\ref{alg:cdcRKHS} is much simpler from the perspective of applications.

        In \Cref{sec:exp}, we consider some \ac{mor} problems, where the functions in the original space $\origspace$ are defined over a real interval, that is, $\Omega \subseteq \R$. In such cases, no bijection between $\Omega$ and $\D$ can exist. Therefore, we use the real--analytic space $\realHtwo$ as defined in~\eqref{eq:realH2}. Note that for this space, Algorithm~\ref{alg:hybrid_equilibrium} can be applied with the same theoretical guarantees as in $H_2(\D)$. Indeed, if $\hat{\bG}$ is real valued on $\mathbb{I}$ and the initial kernel centers satisfy $\gamma \left(\tilde{\blambda}^{(0)}\right) \in \mathbb{I}^r$, then the \ac{rom} produced by Algorithm~\ref{alg:hybrid_equilibrium} will satisfy $\bG_r \in \realHtwo$.

        \item The error $\rkhsError{\xi}{r}\left( \bG_r \right)$ can be computed by
        \begin{equation*}
            \rkhsError{\xi}{r}\left( \bG_r \right) = \sqrt{ \sum_{k=r}^{n+r-1} \left| \left\langle \hat{\bG}, \tgmtkj \right\rangle_{H_2(\D)} \right|^2 },
        \end{equation*}
        where $\tgmtkj$ denote the classical \ac{mt} functions from~\eqref{eq:MT} defined by the parameter sequence
        \begin{equation*}
            \boldsymbol{\mu} = \big(\lambda_0, \ldots, \lambda_{n-1}, \tlambda_0, \tlambda_1, \ldots, \tlambda_{r-1}\big) \in \Omega^{n+r},
        \end{equation*}
        and where $\left(\gamma\big(\lambda_k\big)\right)_{k=0}^{n-1}$ and $\left(\gamma(\tlambda_j)\right)_{j=0}^{r-1}$ are kernel centers associated with the \ac{fom} $\hat{\bG}$ and the \ac{rom} $\hat{\bG}_r$, respectively. This is a consequence of \Cref{thm:unitary}.

        \item For many applications, obtaining the weak solution as proposed in Algorithm~\ref{alg:cdcRKHS}, is computationally expensive. Hence, for practical applications we can use the efficient projected \ac{rom} as defined in \Cref{thm:ProjErrEco}. That is, after retrieving the \ac{rom} kernel centers $\gamma\big( \tilde{\blambda}\big)$ in the companion space via Algorithm~\ref{alg:hybrid_equilibrium}, we can determine a \ac{rom} $\bG_r \in \romspaceo$ by $\Gweakproj{r} \coloneqq \cP_{V_r^{\tilde{\blambda}}} \left[ \bG \right]$. \Cref{thm:ProjErrEco} gives a theoretical error bound on this \ac{rom} when compared to the weak solution, while our experiments in \Cref{sec:exp} empirically demonstrate its effectiveness. 
    \end{enumerate}
\end{remark}

\section{Experiments}
\label{sec:exp}
We consider a number of different experiments to show applications of the norm--optimal  \ac{rkhs}--\ac{mor} methods discussed in \Cref{sec:irka}. The experiments purposely consider a diverse array of \acp{rkhs} in order to assert the generality of the introduced methods. We emphasize that these experiments do not encompass the full spectrum of possible applications, nor do they aim to present market--ready solutions. Instead, our objective is to present several applications where norm--optimal \ac{mor} is (to the best of our knowledge) currently not possible. The four experiments below are not equally mature: the polynomial fitting example in \Cref{subsec:pol} is an illustrative academic exercise, the \ac{svm} and ReLU network experiments in Sections~\ref{subsec:svm} and~\ref{subsec:relu} are fully developed case studies, and the \ac{pde} discretization example in \Cref{subsec:pdes} is a proof of concept for a new application of the framework rather than a finished method. The presented experiments are implemented in MATLAB and Python and are fully reproducible (see code availability statement at the end of the manuscript). The presented quantitative results are obtained on a MacBook Pro with $32$ GB of RAM and Apple M4 CPU, using MATLAB version R2025a and Python 3.9.6 with PyTorch 2.8.0.

\subsection{Polynomial fitting}
\label{subsec:pol}

We begin with a simple application involving fitting polynomials to discrete data points. In particular, consider the $n$--dimensional \ac{rkhs} $\mathbb{P}_n(\Omega)$ containing the real--coefficient polynomials of degree at most $n-1$ defined on the real interval $\Omega \coloneqq (-1,1) \subseteq \R$ equipped with the inner product
\begin{equation*}
    {\langle \bG, \bF \rangle}_{\mathbb{P}_n(\Omega)} \coloneqq \sum_{k=0}^{n-1} g_k f_k, \quad \text{where} \quad \bG(x) \coloneqq \sum_{k=0}^{n-1}g_k x^k, \, \bF(x) \coloneqq \sum_{k=0}^{n-1}f_k x^k.
\end{equation*}
This is an \ac{rkhs} with reproducing kernel
\begin{equation}
    \label{eq:polker}
    \xi(x, y) = \sum_{k=0}^{n-1} x^k y^k = \frac{1 - (xy)^{n}}{1 - xy} \quad (x,y \in \Omega).
\end{equation}
Notice that the monomials $1, x, x^2, \ldots, x^{n-1}$ form an orthonormal basis in $\mathbb{P}_n(\Omega)$. Indeed, let $\phi_k(x) \coloneqq x^k$ $(k=0,\ldots,n-1)$ and observe that by the reproducing property we have
\begin{equation*}
    \left\langle \phi_j, \xi(\cdot, \lambda) \right\rangle_{\mathbb{P}_n(\Omega)} = \lambda^j \quad (j=0,\ldots,n-1, \lambda \in \Omega).
\end{equation*}
Indeed, by exploiting the bilinearity of the inner product and~\eqref{eq:polker}, this gives
\begin{equation*}
     \left\langle \phi_j, \xi(\cdot, \lambda) \right\rangle_{\mathbb{P}_n(\Omega)} =  \left\langle \phi_j,  \sum_{k=0}^{n-1} \phi_k(\cdot) \phi_k(\lambda) \right\rangle_{\mathbb{P}_n(\Omega)} = \sum_{k=0}^{n-1} \phi_k(\lambda)  \left\langle \phi_j,  \phi_k  \right\rangle_{\mathbb{P}_n(\Omega)} = \lambda^j.
\end{equation*}
In the last equation, the expressions on both sides are polynomials in $\lambda$. Hence, they can only be equal if their algebraic coefficients match. Remembering that $\phi_k(x) = x^k$ $(x \in \Omega)$, this gives
    $\left\langle \phi_j, \phi_k \right\rangle_{\mathbb{P}_n(\Omega)} = \delta_{j,k}$.
Let $\{x_j\}_{j=0}^{n-1}\subseteq (-1,1)$ be a grid of pairwise distinct points and let data $\frak{D} \coloneqq \left(\left\{ x_0,y_0 \right),\ldots,\left(x_{n-1},y_{n-1}\right)\right\} \subseteq \Omega 
\times \R$ be given. Recall that there exists a unique polynomial $\bG$ (of \textit{degree} at most $n-1$) such that 
\begin{equation*}
    \bG(x_m) = y_m \quad (m=0,\ldots,n-1),
\end{equation*}
which, in our setting, plays the role of the \ac{fom}. Therefore, for a polynomial $\bG \in \mathbb{P}_n(\Omega)$, there exist kernel centers $\lambda_0, \ldots, \lambda_{n-1} \in (-1, 1)$ and linear coefficients $\res_0, \ldots, \res_{n-1} \in \R$ such that
\begin{equation*}
    \bG(x_m) = \sum_{k=0}^{n-1} \res_k \xi_{\lambda_k}(x_m) = y_m \quad (m=0,\ldots,n-1).
\end{equation*}

Now consider the problem of fitting an order--$r$ polynomial model to the dataset $\frak{D}$. The \textit{order} of the polynomial model to be fitted is defined according to \Cref{def:modelord} and should not be confused with the \textit{degree} of the polynomial, which is understood as the highest exponent appearing in the algebraic form. 
We can pose the problem of fitting an order--$r$ polynomial
\begin{equation*}
    \bG_r = \sum_{j=0}^{r-1} \tres_j \xi_{\tlambda_j} \in \romspaceo \subseteq \mathbb{P}_n(\Omega)
\end{equation*}
to $\frak{D}$ as a $\mathbb{P}_n(\Omega)$--norm--optimal \ac{mor} task. The following lemma clarifies why this is appropriate.
\begin{lemma}[Error of $\mathbb{P}_n(\Omega)$--norm--optimal polynomial fit]
\label{thm:errInterp}
Let $\Omega = (-1, 1)$ and $\bG \in \fomspaceo \subseteq \mathbb{P}_n(\Omega)$. Then, for any $x \in \Omega$, $r < n$, and $\bG_r \in \romspaceo$, we have
\begin{equation*}
    \left| \bG(x) - \bG_r(x) \right| \leq \sqrt{n} \left\| \bG - \bG_r \right\|_{\mathbb{P}_n(\Omega)}.
\end{equation*}
\end{lemma}
\begin{proof}
   The proof of this lemma is given in Appendix~\ref{app:errInterp}.
\end{proof}
Note that the kernel $\xi$ in~\eqref{eq:polker} is infinitely many times differentiable. The smoothness of $\xi$ allows us to perform \ac{mor} using Algorithm~\ref{alg:irkaRKHS} without having to transform the problem to another \ac{rkhs} using the \ac{gmt} method. We denote the \ac{rom} obtained with Algorithm~\ref{alg:irkaRKHS} by $\Gdirect{r}$. In order to compare their effectiveness, we also use Algorithm~\ref{alg:cdcRKHS} to obtain a \ac{rom} $\Gweak{r}$. In this experiment, we consider the bijection $\gamma : \Omega \to \hat{\Omega} \coloneqq (-1, 1)$ defined as $\gamma(x) = x \ (x \in \Omega)$ and choose the companion \ac{rkhs} $\realHtwo$ from~\eqref{eq:realH2}. We recall the discussion below Algorithm~\ref{alg:hybrid_equilibrium}, which clarifies that Algorithm~\ref{alg:hybrid_equilibrium} is closed in the space $\realHtwo$. Since our \ac{gmt}--transformation--based \ac{mor} scheme in Algorithm~\ref{alg:cdcRKHS} relies on Algorithm~\ref{alg:hybrid_equilibrium} to perform the \ac{mor} in the companion space, this practically means that if we initialize \ac{rom} parameters in $(-1, 1)$, then we obtain  $\tlambda_j \in (-1, 1)$ for $j=0,\ldots,r-1$ with Algorithm~\ref{alg:cdcRKHS}.

Note that Algorithm~\ref{alg:cdcRKHS} generally guarantees that $\Gweak{r} \in \wromspaceo$. We can obtain a strong \ac{rom} $\Gweakproj{r} = \cP_{V_r^{\tilde{\blambda}}} \left[ \bG \right]$, where we assume that $\tilde{\blambda} \in \Omeganeq{r}$ denote the kernel centers that describe $\Gweak{r}$ returned by Algorithm~\ref{alg:cdcRKHS}. A formal upper bound for $\left\| \Gweak{r} - \Gweakproj{r} \right\|_{\mathbb{P}_n(\Omega)}$ is given in \Cref{thm:ProjErrEco}. In this experiment, we empirically confirm that both $\Gweak{r}$ and $\Gweakproj{r}$ provide very high quality \acp{rom} and can be used effectively to solve the polynomial fitting task described above.

To evaluate the error of the different approximations, given a dataset $\frak{D}$, we consider the \ac{mse}
\begin{equation}
    \label{eq:MSE}
    \cE^{\frak{D}}_{\mathrm{MSE}}(\bG_r) \coloneqq \frac{1}{n} \sum_{m=0}^{n-1} \left( \bG_r(x_m) - y_m \right)^2 \quad (\bG_r \in \mathbb{P}_n(\Omega)).
\end{equation}
The best \textit{degree}--$(r-1)$ least--squares fit $\bG_r^{\mathrm{LSQ}}$ is also evaluated as a baseline model. This model satisfies
\begin{equation}
    \label{eq:lsqFit}
    \bG_r^{\mathrm{LSQ}}(x) = \sum_{k=0}^{r-1} p_k x^k \quad (x \in \Omega),
\end{equation}
where the coefficients are given by $\bp = [p_0,\ldots,p_{r-1}]^\top = \bX^+ \by$. Here, $\bX^+\in \R^{r \times n}$ denotes the Moore--Penrose pseudoinverse of  
\begin{equation*}
    \bX = \begin{bmatrix}
        1 & x_0 & \ldots & x_{0}^{r-1} \\
        1 & x_1 & \ldots & x_{1}^{r-1} \\
        \vdots & \vdots & \ddots & \vdots \\
        1 & x_{n-1} & \ldots & x_{n-1}^{r-1}
    \end{bmatrix} \in \R^{n \times r}
\end{equation*}
{\sloppy and $\by = [y_0, \ldots, y_{n-1}]^\top \in \R^{n}$. Note that $\bG_r^{\mathrm{LSQ}}, \Gdirect{r}, \Gweak{r}$, and $\Gweakproj{r}$ all live in $r$--dimensional subspaces of $\mathbb{P}_n(\Omega)$. In order to evaluate the proposed models, however, we need to store at least $2r$  parameters ($r$ kernel centers and $r$ coefficients). In contrast, $\bG_r^{\mathrm{LSQ}}$ can be fully evaluated given the $r$ coefficients from~\eqref{eq:lsqFit}. To ensure a fair comparison, we also consider the least squares fit model $\bG_{2r}^{\mathrm{LSQ}}$, which requires the same number ($2r$) of parameters for reconstruction but exists in a $2r$--dimensional subspace of $\mathbb{P}_n(\Omega)$. \Cref{fig:polyfit} illustrates the resulting polynomials using $r=4$. \Cref{fig:polyerr} shows the point--wise absolute error $\left| \bG_r(x_k) - y_k \right|^2$ $(k=0,\ldots,n-1)$ of each model on a logarithmic scale. Finally, \Cref{tab:polyEval} shows the reconstruction errors (in terms of~\eqref{eq:MSE}) for different choices of $r$ and $n=100$. In these experiments, we choose the points $x_k$ equidistantly in $[-0.99, 0.99]$. The corresponding values $y_k$ are randomly sampled from a uniform random distribution.\par}
\begin{figure}[tb]
    \centering
    \includegraphics[width=\linewidth]{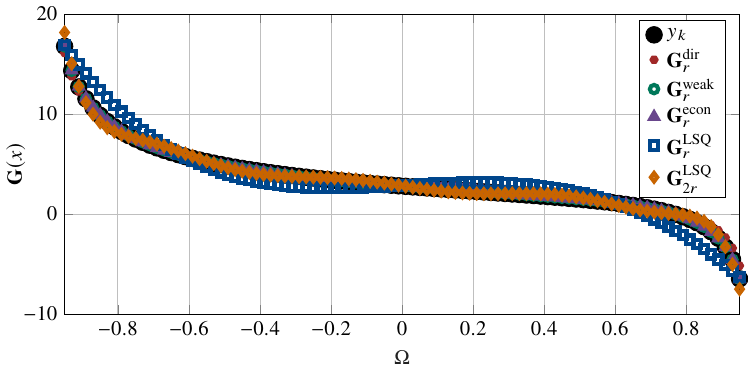}
    \caption{Polynomial fits obtained with the different approximation strategies for $r=4$ and $n=100$.}
    \label{fig:polyfit}
\end{figure}

\begin{figure}[tb]
    \centering
\includegraphics[width=\linewidth]{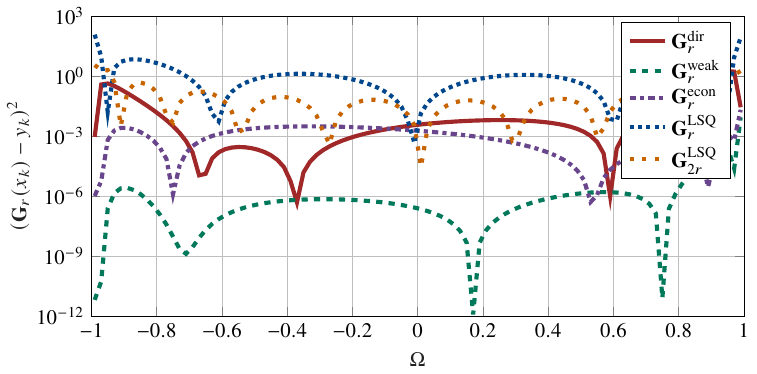}
    \caption{Pointwise square fitting errors obtained with the different approximation strategies for $r=4$ and $n=100$.}
    \label{fig:polyerr}
\end{figure}

\begin{table}[tbp]
\centering
\caption{\ac{mse} and \ac{rkhs} errors for the polynomial fitting application ($n=100$). The best \acp{rom} per reduced order are indicated by boldface font.}
\label{tab:polyEval}
\small
\setlength{\tabcolsep}{4pt}
\renewcommand{\arraystretch}{1.08}
\begin{tabular}{@{}lcccccccc@{}}
\toprule
\textbf{Model}
&
\makebox[0pt][l]{\hspace{1.7cm}$\cE^{\mathfrak{D}}_{\mathrm{MSE}}$}
& & & &
\makebox[0pt][l]{\hspace{1.7cm}$\rkhsError{\xi}{r}$}
& & & \\
\midrule
&
$r=2$ & $r=3$ & $r=4$ & $r=5$
&
$r=2$ & $r=3$ & $r=4$ & $r=5$ \\
\midrule
$\Gdirect{r}$
& $9.687$ & $0.134$ & $0.093$ & $0.060$
& $3.713$ & $1.046$ & $1.132$ & $0.962$ \\

$\Gweak{r}$
& $1.640$ & $\mathbf{0.038}$ & $\mathbf{2.6{\cdot}10^{-4}}$ & $\mathbf{6.5{\cdot}10^{-5}}$
& $2.408$ & $\mathbf{0.417}$ & $\mathbf{0.190}$ & $\mathbf{0.113}$ \\

$\Gweakproj{r}$
& $\mathbf{1.568}$ & $0.076$ & $2.0{\cdot}10^{-3}$ & $1.2{\cdot}10^{-3}$
& $\mathbf{2.362}$ & $0.438$ & $0.255$ & $0.189$ \\

$\bG_r^{\mathrm{LSQ}}$
& $9.784$ & $9.20$ & $3.528$ & $3.485$
& $10.527$ & $10.627$ & $17.631$ & $17.703$ \\

$\bG_{2r}^{\mathrm{LSQ}}$
& $3.528$ & $1.361$ & $0.513$ & $0.184$
& $17.631$ & $47.712$ & $157.146$ & $532.577$ \\
\bottomrule
\end{tabular}
\end{table}
The results in \Cref{tab:polyEval} verify the practical benefits of the proposed \ac{mor} based method for polynomial fitting. Indeed, the proposed \acp{rom} $\Gdirect{r}$, $\Gweak{r}$, and $\Gweakproj{r}$  consistently provide the best \ac{mse} scores regardless of the reduced order $r$. In fact, for $r=3,4,5$, even the worst performing proposed model provides $\geq 3$ times reduced \ac{mse} compared to $\bG_{2r}^{\mathrm{LSQ}}$. The only exception is $r=2$, where Algorithm~\ref{alg:irkaRKHS} failed to produce a good local optimum. We note that even in this case, $\Gweak{r}$ and $\Gweakproj{r}$ outperform the model $\bG_{2r}^{\mathrm{LSQ}}$ by a factor of at least $2.15$ in terms of \ac{mse}. Compared to $\bG_r^{\mathrm{LSQ}}$, the error reduction factor climbs to at least $5.97$. We note that even in this worst--case scenario, the \ac{mse} obtained by $\Gdirect{r}$ is comparable to the model  $\bG_r^{\mathrm{LSQ}}$. 

The above experiments provide insight into several properties of \ac{mor}--based polynomial fitting. First, we observe that even though it solves the problem in $\mathbb{P}_n(\Omega)$, Algorithm~\ref{alg:irkaRKHS} is consistently outperformed by the \ac{gmt}--transformation--based models obtained using Algorithm~\ref{alg:cdcRKHS}. This can be explained by the fact that Algorithm~\ref{alg:irkaRKHS} does not exploit the underlying geometry of $\Omega$. In contrast, $\Gweak{r}$ and $\Gweakproj{r}$ are obtained by first transforming the problem into $\realHtwo$ and then using a geometry--aware Newton's method (see Algorithm~\ref{alg:hybrid_equilibrium}).

Another interesting observation is that $\Gweak{r}$ and $\Gweakproj{r}$ provide similar results. While the weak \ac{rom} $\Gweak{r}$ achieves better overall scores both in terms of \ac{mse} and \ac{rkhs} error, $\Gweakproj{r}$ also provides high quality approximations of the data and consistently outperforms every model other than $\Gweak{r}$. This is especially useful if the application scenario does not allow us to store the $2r + n$ parameters needed to evaluate the weak \ac{rom} (see \Cref{subsec:GMTtrans}). Hence, while it provides the best overall approximation of the data, using the weak model $\Gweak{r}$ is not practical. Especially in the current experiment, we can express $\Gweak{r}$ using $n$ algebraic coefficients due to~\eqref{eq:polker}. While this is less than the theoretically required $n+2r$ parameters for weak model reconstruction discussed in \Cref{subsec:GMTtrans}, it matches the number of coefficients needed to reconstruct the interpolating polynomial $\bG$. We also note that according to \Cref{tab:polyEval}, the \acp{rom} $\Gdirect{r}, \Gweak{r}$, and $\Gweakproj{r}$ significantly outperform $\bG_r^{\mathrm{LSQ}}$ and $\bG_{2r}^{\mathrm{LSQ}}$ in terms of the \ac{rkhs} error $\rkhsError{\xi}{r}$. This is expected, as the latter models do not attempt to minimize $\rkhsError{\xi}{r}$. We recall that aiming for a small \ac{rkhs}--norm--based error is beneficial for the data fitting problem, because \Cref{thm:errInterp} provides a computable upper bound for the pointwise error. 

Finally, we consider the runtime of each algorithm in \Cref{tab:runtimesPol}. We observe that the proposed models $\Gweak{r}$ and $\Gweakproj{r}$ obtained using Algorithm~\ref{alg:cdcRKHS} can be computed in less than $1$ second for every experiment. In contrast, computing $\Gdirect{r}$ with Algorithm~\ref{alg:irkaRKHS} takes $6.34$ seconds on average. We note that this increase is likely not due to theoretical reasons, rather, it is a consequence of our chosen implementation. Indeed, in order to make Algorithm~\ref{alg:irkaRKHS} as easy to use as possible, we exploit the automatic differentiation feature of PyTorch to compute the derivatives appearing in the iteration. In particular, our MATLAB--based main script calls a PyTorch routine every time we need to compute the Jacobian associated with the equilibrium function defined in~\eqref{eq:EquibFun}. This significantly slows execution, but allows for a plug and play experience for the users, who can easily try Algorithm~\ref{alg:irkaRKHS} in a sufficiently smooth \ac{rkhs} of their choice. 
Finally, we note that in the case of the \ac{gmt}--based models, we use Algorithm~\ref{alg:hybrid_equilibrium} to obtain the \acp{rom} in the companion \ac{rkhs} $\realHtwo$. For this algorithm, the necessary derivatives can be efficiently computed, and no automatic differentiation or transition from MATLAB to PyTorch is required. This explains the improved execution times.

\begin{table}[tbp]
    \caption{Recorded runtimes (seconds) of the considered MOR schemes. $\bG_r^{\mathrm{LSQ}}$ and $\bG_{2r}^{\mathrm{LSQ}}$ are obtained with MATLAB's \texttt{polyfit} routine.}
    \centering
    \begin{tabular}{cccccc}
    \toprule
        $r$ & $\Gdirect{r}$ & $\max \left\{\Gweak{r}, \Gweakproj{r} \right\}$ & $\bG_r^{\mathrm{LSQ}}$ & $\bG_{2r}^{\mathrm{LSQ}}$\\
        \midrule
        2 & $6.10$ & $0.53$ & $0.02$ & $3.05 \cdot 10^{-4}$  \\
        3 & $6.15$ & $0.28$ & $4.62 \cdot 10^{-3}$ & $2.96 \cdot 10^{-4}$  \\
        4 & $6.76$ & $0.24$ & $4.21 \cdot 10^{-3}$ & $2.52 \cdot 10^{-4}$   \\
        5 & $6.33$ & $0.28$ &  $3.99 \cdot 10^{-3}$ & $2.72 \cdot 10^{-4}$  \\
        \bottomrule
    \end{tabular}
    \label{tab:runtimesPol}
\end{table}

\subsection{MOR on support vector machines}
\label{subsec:svm}

 Next, we consider \ac{mor} for \acp{svm} with \ac{rbf} kernels. The experiment uses the brake disc deformation data from~\cite{dozsa_detection_2026}, where brake fluid pressure and acceleration signals are considered to detect the presence of brake disc runout in passenger vehicles. This fault detection task is posed as a binary classification problem where measurement segments are classified as either healthy or deformed. The preprocessed segments correspond to active braking, with the vehicle speed between $90$ and $30$ km/h. We use the empirical standard deviation features introduced in~\cite{dozsa_brake_2024} to reduce the dimensionality of the data entries. In the present experiment, only two of the four sensor--based features are retained. 
 We use distinct training and test datasets, denoted by $\frak{D}_{\mathrm{tr}}$ and $\frak{D}_{\mathrm{te}}$, respectively.
 The training data is used to construct the \ac{fom}, whereas the test data is reserved for the subsequent evaluation of the \ac{fom} and its reduced--order approximations. More precisely, for pairwise different points $x_0,\ldots,x_{N_s-1}$, $s\in\{\mathrm{tr},\mathrm{te}\}$, we write
 $$
\frak{D}_{s}\coloneqq\left\{\left(x_0^{(s)},y_0^{(s)}\right),\ldots,\left(x_{N_s-1}^{(s)},y_{N_s-1}^{(s)}\right)\right\}
\subseteq \mathbb{R}^2\times\{-1,1\},
\quad s\in\{\mathrm{tr},\mathrm{te}\}.
$$

The classification procedure yields $N_{\mathrm{tr}}=5841$ training samples and
$N_{\mathrm{te}}=649$ test samples, with both datasets balanced between the deformed and healthy classes.

Here, $x_{m,1}^{(s)}$ and $x_{m,2}^{(s)}$ denote, respectively,
the empirical standard deviations of the acceleration signal measured
by an accelerometer attached to the vehicle chassis and of the brake
fluid pressure signal measured by a pressure sensor.

 Each feature vector $x_m^{(s)}$ is identified with the complex--valued input
$
z_m^{(s)}
\coloneqq
x_{m,1}^{(s)}+\mathrm{i}x_{m,2}^{(s)}
\in \Omega \coloneqq \D$.
Accordingly, we denote the corresponding complex--valued inputs by $$\mathcal{Z}_s
\coloneqq
\left(z_0^{(s)},\ldots,z_{N_s-1}^{(s)}\right)
\in \Omega^{N_s},
\quad
s\in\{\mathrm{tr},\mathrm{te}\}.$$

The inclusion in $\D$ is obtained by applying the same normalization to
the training and test inputs. This entails no restriction for the present
\ac{rbf}--\ac{svm} formulation, since a global rescaling of the input coordinates can
be compensated by a corresponding adjustment of the \ac{rbf} kernel parameter.
The kernel parameter is therefore selected in the normalized coordinates.
In particular, the support vector locations $\lambda_k$, being selected
from the training inputs, also belong to $\Omega=\D$.

Let the trained \ac{rbf}--\ac{svm} be denoted by
\begin{equation}
    \label{eq:svmFOM}
    \bD(z) = \bG(z) + b, \quad
    \bG(z) = \sum_{k=0}^{n-1} \res_k \xi(z,\lambda_k),
    \quad
    \xi(z,\lambda) = \exp\left(-\sigma |z-\lambda|^2\right),
\end{equation}
where $\sigma>0$ is the \ac{rbf} kernel parameter,  $\lambda_k \in \Omega$ denote the kernel centers, $\res_k$ are the signed dual coefficients, and $b\in\R$ is the bias term. Thus, the nonconstant part of the \ac{svm} decision function is an element $\bG \in \fomspaceo \subseteq \origspace$ of the \ac{rkhs} generated by the \ac{rbf} kernel $\xi$. The order $n$ of the \ac{fom} is equal to the number of support vectors. For the present experiment, the trained \ac{svm} contains
$n=N_{\mathrm{SV}}=5841$ support vectors.
The same bias term $b$ is used for the \ac{fom} and all \acp{rom}. We therefore apply the reduction to the nonconstant component $\bG$, seeking an approximation $\bG_r\in\romspaceo$. The corresponding reduced--order decision function is
$\bD_r \coloneqq \bG_r + b$.

We compare two reduction strategies already introduced in
\Cref{subsec:pol}. The direct approach computes a strong \ac{rom}
$\Gdirect{r}\in\romspaceo$ directly in the original \ac{rbf}--\ac{rkhs}
$\origspace$. The \ac{gmt}--based approach instead yields the projected strong \ac{rom} of the form
    $\Gweakproj{r}
    =
    \cP_{\Vn{r}{\tilde{\blambda}}}[\bG]$ for some kernel centers $\tilde{\blambda} \in \Omega^r$.
We do not report the weak \ac{gmt}--\ac{rom} in this experiment, since our focus is on reduced models that also substantially reduce inference costs. Indeed, the weak representation retains an $n$--dependent contribution, whereas the projected \ac{rom} depends only on the reduced order $r$.

The map $\gamma:\Omega\to\hat{\Omega}$ used for the \ac{gmt} transformation is set to $\gamma(z)\coloneqq z$. This choice of $\gamma$ is valid because, with the preprocessing steps described above, the kernel centers of the \ac{fom} belong to $\D$.

For any $z\in\Omega$, the predicted label of the \ac{rom} is given by
$\sgn (\bD_r(z))$, using the same convention as for the \ac{fom}.
Classification accuracy is evaluated separately on the training and test
datasets.
In addition, for $s\in\{\mathrm{tr},\mathrm{te}\}$, we report the pointwise
decision function error
\begin{equation*}
    \cE_{\mathrm{pt}}(\bD_r;\mathcal{Z}_s)
    \coloneqq
    \max_{z_m^{(s)}\in\mathcal{Z}_s}
    \left|
        \bD\big(z_m^{(s)}\big)-\bD_r\big(z_m^{(s)}\big)
    \right|
    =
    \max_{z_m^{(s)}\in\mathcal{Z}_s}
    \left|
        \bG\big(z_m^{(s)}\big)-\bG_r\big(z_m^{(s)}\big)
    \right|.
\end{equation*}
In addition to the dataset--dependent pointwise error, we consider the
\ac{rkhs} approximation error $\rkhsError{\xi}{r}$ from~\eqref{eq:Objective}.
Unlike $\cE_{\mathrm{pt}}$, this quantity measures the \ac{fom}--\ac{rom}
approximation over the underlying \ac{rkhs} and is independent of the finite
training and test datasets.

Because the \ac{fom} and \ac{rom} share the same bias, we also write
\begin{equation*}
    \rkhsError{\xi}{r}(\bD_r)
    \coloneqq
    \left\|\bD-\bD_r\right\|_{\origspace}
    =
    \left\|\bG-\bG_r\right\|_{\origspace}
    =
    \rkhsError{\xi}{r}(\bG_r),
\end{equation*}
where the \ac{rkhs} norm is well defined because
$\bD-\bD_r=\bG-\bG_r\in\origspace$.
For reporting purposes, we additionally consider the relative \ac{rkhs} error
\begin{equation}
    \label{eq:relErr}
    \cE_{\xi,\mathrm{rel}}^{r}(\bG_r)
    \coloneqq
    \frac{
        \rkhsError{\xi}{r}(\bG_r)
    }{
        \|\bG\|_{\origspace}
    }
    =
    \frac{
        \|\bG-\bG_r\|_{\origspace}
    }{
        \|\bG\|_{\origspace}
    }.
\end{equation}
The relative \ac{rkhs} error provides a scale--independent measure of the
\ac{fom}--\ac{rom} approximation quality and, like the absolute \ac{rkhs} error, is
independent of the finite training and test datasets.

The main objective of this experiment is to approximate the \ac{fom} decision
function with fewer stored parameters while preserving its classification
behavior.
This is consistent with the \ac{rkhs} approximation objective. Indeed, since $\xi(z,z)=1$ for the \ac{rbf} kernel, the reproducing property and the Cauchy--Schwarz inequality give
\begin{equation}\label{eq:cond_agree_uniform}
    \left|\bD(z)-\bD_r(z)\right|
    = \left|\bG(z)-\bG_r(z)\right|
    \leq \left\|\bG-\bG_r\right\|_{\origspace}
    =\rkhsError{\xi}{r}(\bD_r) \quad (z\in\Omega).
\end{equation}
Consequently, a \ac{rom} close to \ac{fom} in \ac{rkhs} norm implies similar decision functions over $\Omega$. Whether this approximation
yields the same classification at a given point depends on its spatial position
relative to the decision boundary.

\begin{lemma}[Simple pointwise classification certificate] \label{thm:svmRomMarginAgreement} Let $\bD$ and $\bD_r$ be the \ac{fom} and \ac{rom} decision functions~\eqref{eq:svmFOM}, respectively, and let $z \in \Omega$. If 
\begin{equation}\label{eq:cond_agree}
|\bD_r(z)|>|\bD (z)-\bD_r (z)|, 
\end{equation}
then \[ \sgn (\bD(z))=\sgn (\bD_r(z)). \] Hence, \ac{fom} and \ac{rom} assign the same label to $z$. \end{lemma}\begin{proof}
    For the proof, we refer to~\Cref{app:svmMargin}.
\end{proof}

\begin{corollary}[Uniform classification certificate]
\label{cor:svmRomUniformMarginAgreement}
Let $\bD$ and $\bD_r$ be the \ac{fom} and \ac{rom}
decision functions, respectively, and let $z\in\Omega$.
If
\begin{equation}\label{eq:cond_agree_rkhs}
    |\bD_r(z)| > \rkhsError{\xi}{r}(\bD_r),
\end{equation}
then
\[
    \sgn(\bD(z)) = \sgn(\bD_r(z)).
\]
Hence, \ac{fom} and \ac{rom} assign the same label to $z$.
\end{corollary}

\begin{proof}
By the uniform \ac{rkhs} error bound~\eqref{eq:cond_agree_uniform},
condition~\eqref{eq:cond_agree_rkhs} implies
condition~\eqref{eq:cond_agree}.
The result therefore follows directly from
\Cref{thm:svmRomMarginAgreement}.
\end{proof}
\Cref{cor:svmRomUniformMarginAgreement} shows that the \ac{rom} classifier is reliable at all points where its own decision value is sufficiently far from zero. 
Exploiting~\ref{cor:svmRomUniformMarginAgreement}, once the \ac{rkhs} error between the \ac{rom} and the \ac{fom} is known, partial agreement between \ac{fom} and \ac{rom} can be certified a priori at no additional cost for every $z\in \Omega$ that satisfies the above margin condition.
 
Hence, \Cref{thm:svmRomMarginAgreement} and \Cref{cor:svmRomUniformMarginAgreement} ensure that the signs match between $\bD(z)$ and $\bD_r(z)$ if $z$ is sufficiently far away from the decision boundary. In order to obtain a useful check for cases when $|\bD_r(z)| \leq \rkhsError{\xi}{r}(\bD_r)$, consider the following lemma.

\begin{lemma}[Projection--based classification certificate]
    \label{thm:svmRomMarginAgreementV2}
    Let $\bD=\bG+b$ and $\bD_r=\bG_r+b$ be the \ac{fom} and \ac{rom} decision functions~\eqref{eq:svmFOM}, respectively, with $\bG \in \fomspaceo$ and $\bG_r \in \romspaceo$. 
    Under \Cref{ass:GMTkers}, let $\tilde{\blambda} \in \Omega^r$ be the tuple of kernel centers defining the \ac{rom} $\bG_r$, and $\cP_{\Vn{r}{\tilde{\blambda}}}$ and
    $\cP_{\left(\Vn{r}{\tilde{\blambda}}\right)^\perp}$ denote the
    orthogonal projectors onto $\Vn{r}{\tilde{\blambda}}$ defined according to~\eqref{eq:VnSec5} and its
    orthogonal complement in $\cH$, respectively. Assume that $\bG_r = \cP_{\Vn{r}{\tilde{\blambda}}}\left[ \bG \right]$.
    
Then, for any $z\in \Omega$, if
    \begin{equation}
        \label{eq:projest}
        |\bD_r(z)| > \rkhsError{\xi}{r}(\bD_r)\left\| \cP_{\left(\Vn{r}{\tilde{\blambda}}\right)^\perp}\left[ \xi_z \right] \right\|_{\origspace} \eqqcolon \cE_{\xi,\mathrm{proj}}^r(\bD_r;z),
    \end{equation}
    then $\sgn (\bD(z)) = \sgn (\bD_r(z))$.
\end{lemma}
\begin{proof}
    The proof is given in~\Cref{app:svmMarginV2}.
\end{proof}

\begin{remark}
    Note that we can make the subspace $\Vn{r}{\tilde{\blambda}}$ in \Cref{thm:svmRomMarginAgreementV2} smaller by considering $\tVn{r}{\tilde{\blambda}}$ from \Cref{thm:orthocond} instead. However, for ease of notation we do not pursue this route.
\end{remark}

Assume that $\tilde{\blambda} \in \Omeganeq{r}$. The projection--dependent factor in~\eqref{eq:projest}
vanishes whenever $\xi_z\in \Vn{r}{\tilde{\blambda}}$.
In particular, this occurs at the reduced kernel centers $z=\tlambda_j$, $j=0,\ldots,r-1$.
Hence, the factor accounts for the local relationship between the kernel slice
$\xi_z$ and the reduced space $\Vn{r}{\tilde{\blambda}}$, yielding a substantially tighter
pointwise error estimate than the uniform \ac{rkhs} bound.

To obtain a practically useful sign check, we need to compute $\cP_{\left(V_r^{\tilde{\blambda} }\right)^\perp}[\xi_z]$ for all $z \in \Omega$.
We obtain an elegant and computationally economical method using \ac{gmt} expansions. Let $\left\{\rgmtj \right\}_{j=0}^{r-1}$ be the \ac{gmt} system obtained by orthogonalizing $\xi_{\tlambda_j}$, where the kernel centers $\tlambda_j$, $j=0,\ldots,r-1$ define the \ac{rom} and $V_r^{\tilde{\blambda}}$. Then,
\begin{equation}
    \label{eq:psumForRBF}
    \cP_{\left( V_r^{\tilde{\blambda}} \right)^\perp} \left[ \xi_z \right] = \xi_z - \cP_{ V_r^{\tilde{\blambda}}} \left[ \xi_z \right] = \xi_z - \gmtPsumr\left[ \xi_{z} \right] = \xi_z - \sum_{j=0}^{r-1} \left\langle \xi_z, \rgmtj \right\rangle_{\origspace} \rgmtj.
\end{equation}
Using the reproducing property, and since the \ac{rbf} kernel $\xi$ is real--valued, we find
\begin{equation*}
    \left\langle \xi_z, \rgmtj \right\rangle_{\origspace} = \overline{\left\langle \rgmtj, \xi_z \right\rangle}_{\origspace} = \overline{\rgmtj(z)} = \rgmtj(z), 
\end{equation*}
Plugging this into~\eqref{eq:psumForRBF}, using $\xi_z(z) = 1$, the producing property, and the orthogonality of $\left\{\rgmtj \right\}_{j=0}^{r-1}$, we obtain the formula
\begin{equation*}
    \left\| \cP_{\left(V_r^{\tilde{\blambda}}\right)^\perp} \left[ \xi_z \right] \right\|_{\origspace}^2
= \left\langle \xi_z - \sum_{j=0}^{r-1} \rgmtj(z) \rgmtj(\cdot),  \xi_z - \sum_{j=0}^{r-1} \rgmtj(z) \rgmtj(\cdot) \right\rangle_{\origspace} = 1 -  \sum_{j=0}^{r-1} \rgmtj(z)^2.
\end{equation*}
Consequently, the projection--dependent consistency condition can be
evaluated as
\begin{equation*}
    |\bD_r(z)|
    >
    \rkhsError{\xi}{r}(\bD_r)
    \sqrt{
        1-\sum_{j=0}^{r-1}\rgmtj(z)^2
    }.
\end{equation*}


Given the reduced centers $\tlambda_0,\ldots,\tlambda_{r-1}$, the $r$ kernel values required to evaluate the \ac{gmt} system can be computed once and reused. In the present implementation, evaluating the projection--dependent factor requires $r^2$ kernel evaluations. 
In contrast, the \Cref{thm:svmRomMarginAgreement} criterion requires no additional model dependent computation once the \ac{rom} decision value and the \ac{rkhs} error are available.
We note that in other \acp{rkhs}, such as boundary weighted Hardy spaces~\cite{qian_kernel_2026} or Bergman spaces~\cite{duren_bergman_2024}, we can compute $\rgmtj$ with a single kernel evaluation.

The two sufficient conditions naturally suggest a hierarchical certified inference strategy. For a new input $z\in \Omega$, the uniform certificate in   \Cref{cor:svmRomUniformMarginAgreement} can be verified after evaluating the \ac{rom}, since it requires essentially no additional online computation once the \ac{rom} decision value and the \ac{rkhs} error are available.
If the condition is satisfied, the \ac{rom} prediction is certified to agree with the \ac{fom} prediction.

The projection--based certificate of \Cref{thm:svmRomMarginAgreementV2} provides a tighter, but more computationally demanding alternative. When used together with the uniform
certificate, it can be applied to inputs that are not certified by the latter,
yielding a hierarchical two--stage certification procedure.

In the following, the applicability of the considered certificates is evaluated separately on the finite input sets
$\mathcal{Z}_{\mathrm{tr}}$ and $\mathcal{Z}_{\mathrm{te}}$.
We report the percentage of samples satisfying the uniform and projection--based certificate, and the exact pointwise classification agreement between the \ac{fom} and \ac{rom} as $\Cert_{\mathrm{unif}}$, $\Cert_{\mathrm{proj}}$, and
$\Cert_{\mathrm{true}}$, respectively.
The approximation errors, classification performance, and applicability
of the consistency conditions are summarized in
\Cref{tab:svm_certification_results}.

\begin{table*}[tbp]
\centering
\caption{Approximation errors, classification performance, and
consistency--certificate applicability for the \ac{fom}, direct \acp{fom}, and
projected \ac{gmt}--\acp{rom}.}
\label{tab:svm_certification_results}
\small
\setlength{\tabcolsep}{3pt}

\begin{tabular}{@{}llccccccc@{}}
\toprule
\textbf{Data}
& \textbf{Metric}
& $\bG$
& $\Gdirect{30}$
& $\Gdirect{100}$
& $\Gdirect{200}$
& $\Gweakproj{30}$
& $\Gweakproj{100}$
& $\Gweakproj{200}$
\\
\midrule

Global
& $\cE^r_{\xi,\mathrm{rel}}$
& $0$
& $0.4416$
& $0.1578$
& $0.1322$
& $0.4621$
& $0.0868$
& $0.0257$
\\

\midrule

\multirow{5}{*}{Training}
& $\cE_{\mathrm{pt}}$
& $0$
& $3.0816$
& $0.7373$
& $0.4306$
& $0.6019$
& $0.5163$
& $0.0571$
\\
& Acc. $(\%)$
& $80.50$
& $59.65$
& $80.48$
& $80.48$
& $81.56$
& $80.48$
& $80.50$
\\
& $\Cert_{\mathrm{unif}}$ $(\%)$
& --
& $0.00$
& $1.39$
& $3.94$
& $0.00$
& $14.19$
& $92.57$
\\
& $\Cert_{\mathrm{proj}}$ $(\%)$
& --
& $71.80$
& $99.57$
& $99.74$
& $86.53$
& $99.93$
& $100.00$
\\
& $\Cert_{\mathrm{true}}$ $(\%)$
& --
& $75.07$
& $99.83$
& $99.86$
& $92.79$
& $99.95$
& $100.00$
\\

\midrule

\multirow{5}{*}{Test}
& $\cE_{\mathrm{pt}}$
& $0$
& $2.8452$
& $0.2384$
& $0.2241$
& $0.5278$
& $0.1269$
& $0.0040$
\\
& Acc. $(\%)$
& $80.28$
& $57.94$
& $80.28$
& $80.28$
& $80.89$
& $80.28$
& $80.28$
\\
& $\Cert_{\mathrm{unif}}$ $(\%)$
& --
& $0.00$
& $1.08$
& $3.70$
& $0.00$
& $13.71$
& $94.76$
\\
& $\Cert_{\mathrm{proj}}$ $(\%)$
& --
& $71.49$
& $99.85$
& $100.00$
& $88.29$
& $100.00$
& $100.00$
\\
& $\Cert_{\mathrm{true}}$ $(\%)$
& --
& $73.81$
& $100.00$
& $100.00$
& $93.53$
& $100.00$
& $100.00$
\\

\bottomrule
\end{tabular}
\end{table*}

\begin{figure}[tbp]
    \centering
    \includegraphics{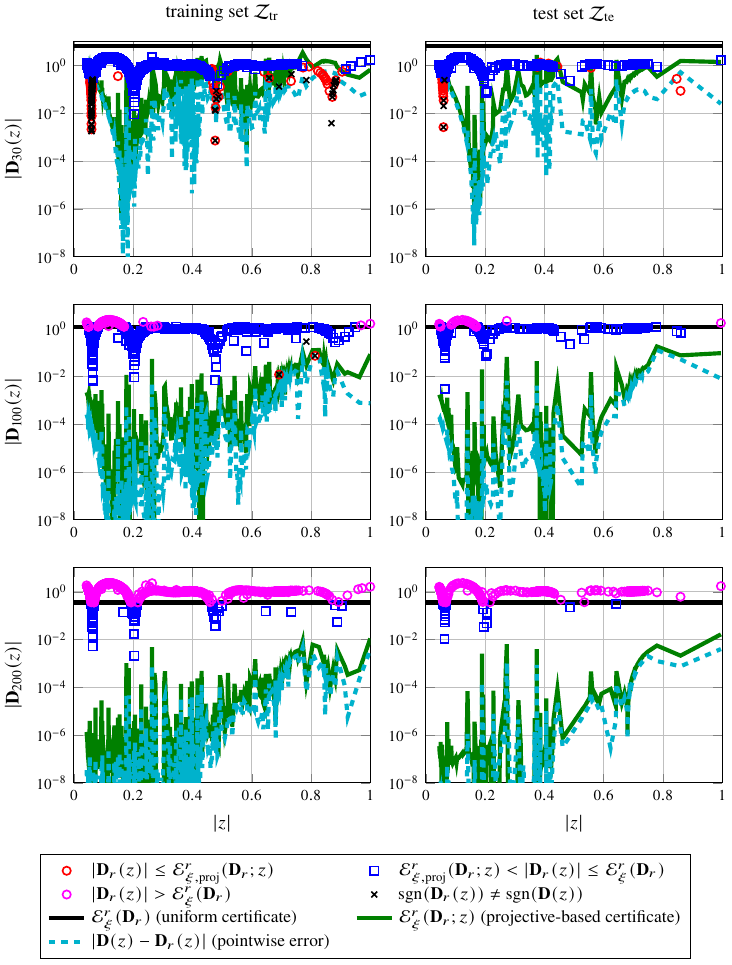}
    \caption{Consistency certificates for the projected \ac{gmt}--\acp{rom} with
$r=30,100,200$, evaluated on $\mathcal{Z}_{\mathrm{tr}}$ (left) and
$\mathcal{Z}_{\mathrm{te}}$ (right).}
\label{fig:SVM_Outputs_and_Bounds}
\end{figure}

\Cref{fig:SVM_Outputs_and_Bounds} illustrates how the applicability of the consistency certificates evolves with the \ac{rom} order  when evaluated on
$\mathcal Z_{\mathrm{tr}}$ (left) and $\mathcal Z_{\mathrm{te}}$ (right). The results highlight that approximation accuracy and
certificate applicability should be considered as distinct properties of the \ac{rom}. Increasing the reduced order generally improves the agreement between the \ac{rom} and the \ac{fom}, but the computational usefulness of the different sufficient conditions can vary significantly.

Recall that the \ac{fom} has order $n=2928$, while the reduced orders considered here are $r\in \{30, 100, 200\}$.
For $r=30$, the uniform certificate is too conservative to certify any of
the evaluated samples, whereas the projection--based certificate from \Cref{thm:svmRomMarginAgreementV2} remains
applicable to a large fraction of both the training and test inputs. Therefore, this regime provides a representative case in which the second, projection--based certification stage is practically useful.
In the present implementation, the corresponding $O(r^2)$ arithmetic cost amounts to $30^2=900$ operations at the level of the reduced representation and remains computationally competitive with the \ac{fom} evaluation. 

For $r=100$, the results highlight the distinction between \ac{rom} fidelity and certificate applicability. At this order, the reduced models already reproduce the \ac{fom} classification accuracy very closely, and the projected \ac{gmt}--\ac{rom} exhibits a small pointwise decision--function error. Nevertheless, the uniform certificate remains applicable only to a relatively limited fraction of the samples. The projection--dependent condition is much tighter and certifies almost all samples. However, its evaluation cost of order $O(r^2)$ is now considerably larger than evaluating the \ac{fom} directly. Consequently, in this regime the projection--based certification stage may no longer be computationally attractive relative to evaluating the \ac{fom} directly.

A different behavior is observed at $r=200$. For the projected \ac{gmt}--\ac{rom},
the uniform certificate becomes applicable to the large majority of the
samples, while the projection--based certificate is essentially fully applicable but too costly. This regime is therefore particularly
favorable for certified inference, because most predictions can already be
certified using the inexpensive uniform condition, without evaluating the more costly projection--based factor.

 \Cref{tab:svm_certification_results} also shows that the direct reduction approach becomes highly competitive as the reduced order
increases. In particular, for $r=100$ and $r=200$, the direct \ac{rom} attains
essentially the same classification accuracy as the projected \ac{gmt}--\ac{rom} and
closely reproduces the \ac{fom} predictions. Its relative \ac{rkhs} error, however, decreases
more slowly, and consequently the sufficient consistency certificates derived
above remain applicable to a smaller fraction of the samples. This difference
should not be interpreted as an intrinsic limitation of the direct approach.
Rather, in the present implementation, the reduction procedure  employed after the \ac{gmt} transformation in
$H_2(\D)$ is considerably more developed and numerically refined than the
optimization performed directly in the original \ac{rbf}--\ac{rkhs}. The comparison
therefore also reflects the different maturity of the two optimization
procedures, in addition to the properties of the corresponding reduced
representations.

The figure also provides a visual interpretation of the hierarchical inference strategy described above. After evaluating the \ac{rom}, samples that satisfy the uniform certificate can be used immediately. For the remaining samples, the tighter projection--based certificate can be checked before deciding whether to accept the \ac{rom} prediction without certification or whether the \ac{fom} should be evaluated as a fallback. Thus, improving the \ac{rom} affects not only approximation accuracy, but also the practical usefulness of the certification mechanism by increasing the fraction of inputs in $\mathcal Z_{\mathrm{tr}}$ and $\mathcal Z_{\mathrm{te}}$ for which \ac{fom}--\ac{rom} consistency can be established without a full--order evaluation.

For reference, if each kernel center and coefficient is counted as one stored quantity and the common bias is counted once, the \ac{fom} requires $p_{\mathrm{FOM}}=2n+1$ stored quantities, whereas the direct and projected \acp{rom} of order $r$ require $p_{\mathrm{dir}}=p_{\mathrm{proj}}=2r+1$. This parameter reduction is also directly related to inference cost, since evaluating an \ac{rbf}--\ac{svm} requires one kernel evaluation per stored center. Therefore, reducing the number of kernel centers from $n$ to $r$ lowers the cost of evaluating the decision function and, consequently, the expected runtime and energy consumption in embedded or real--time implementations. In particular, the projected \ac{rom} is smaller than the \ac{fom} whenever $r<n$.

Finally, the consistency results above should be interpreted as optional certification mechanisms rather than requirements for deploying the
reduced classifier. Once constructed, a strong \ac{rom} can be evaluated independently of the \ac{fom} and may therefore be used as a standalone surrogate whenever memory, computational, or hardware constraints make full--order inference impractical.

\subsection{MOR on splines and ReLU networks}
\label{subsec:relu}

In this example, we perform \ac{mor} on a deep feedforward ReLU network of type $\R \to \R$ using a \ac{gmt} transformation based approach. We follow~\cite[Equation (1.1)]{plonka_spline_2023} to obtain the following definition of ReLU networks mapping $\R \to \R$:
\begin{equation}
    \label{eq:reLUNet}
    \begin{split}
        \bF_1(t) &= t \bA^{(1)} + \bb^{(1)}, \\
        \bF_{\ell}(t) &= \bA^{(\ell)} \relu \left( \bF_{\ell-1}(t) + t \bc^{(\ell)} + \bb^{(\ell)} \right) \quad (\ell=2,\ldots,N-1), \\
       \bF(t) := \bF_N(t) &= \bA^{(N)} \relu \left(\bF_{N-1}(t) + t \bc^{(N)} + \bb^{(N)}\right),
    \end{split}
\end{equation}
where $\bA^{(\ell)} \in \R^{k_{\ell} \times k_{\ell-1}}, \bb^{(\ell)}, \bc^{(\ell)} \in \R^{k_{\ell}}$ for $\ell=1,\ldots,N$ with $k_N = 1$ and
\begin{equation*}
    \relu(\bx) \coloneqq \left[ \max\{x_1, 0\}, \ldots, \max\{ x_k, 0 \} \right]^\top,  \quad
    \bx = [x_1,\ldots,x_k]^\top \in\R^k.
\end{equation*}

ReLU networks have been instrumental to the development of modern deep learning. They form the building blocks of many state--of--the--art data--driven modeling schemes~\cite{vaswani_attention_2017, kovacs_vpnet_2022}. They are universal approximators~\cite{hornik_approximation_1991}, which makes them particularly well suited to learn nonlinear maps from data. On the other hand, to achieve good approximation, large networks with exceedingly high parameter counts are often required. For example, modern image and language processing deep learning models often contain billions of parameters, which makes training and model evaluation expensive. 

Compression of neural networks is often performed using pruning~\cite{blalock_what_2020}, which removes neurons (rows and columns of $\bA^{(\ell)}$) while trying to preserve the input--output relationship defined by $\bF$. This removal or pruning step is then followed by an adjustment of the remaining parameters $\bA_r^{(\ell)}, \bb_r^{(\ell)}$, and $\bc_r^{(\ell)}$ for $\ell=1,\ldots ,N$ using gradient--based backpropagation schemes~\cite{kingma_adam_2014}. Variants of this heuristic approach have been empirically verified to perform well for trained networks~\cite{blalock_what_2020}. However, no theoretical guarantees are given regarding the performance of the reduced (pruned) network. Furthermore, pruning has been observed to be unusable for very complex cases, for example, when the parameters of the neural network $\bF$ are chosen randomly. 

In the approach proposed here, we do not attempt to reduce the model as given in~\eqref{eq:reLUNet}. Instead, we exploit the fact that any scalar--input scalar--output ReLU network of this form can be represented as a \ac{cpl} function, or equivalently with linear splines. In~\cite{plonka_spline_2023}, an efficient algorithm is proposed to obtain the \ac{cpl} representation
\begin{equation}
    \label{eq:CPL}
    \bF(t) = q_0 + q_1 t + \sum_{k=0}^{n-1} \wres_k \relu(t - \lambda_k ) \quad (t \in \R)
\end{equation}
of a deep ReLU network $\bF$ of the form~\eqref{eq:reLUNet}. Although the \ac{cpl} representation can always be obtained, at first glance the relationship between the number of parameters needed to describe $\bF$ with~\eqref{eq:reLUNet} when compared to~\eqref{eq:CPL} looks discouraging. Indeed, in~\cite{montufar_number_2014} it is shown that $n$ in~\eqref{eq:CPL}
is bounded by $\cO(N W^2)$, where $W \coloneqq \max\left\{ k_0, \ldots, k_{N} \right\}$ denotes the width of the ReLU network. In contrast, in~\cite{hanin_complexity_2019} the authors give evidence that despite this theoretical bound, in practice $n$ grows linearly with the number of neurons. 

In the presented experiment, we consider a four--layer neural network ($N=4$) with $k_1=k_4=1, k_2=k_3=25$, and $\bc^{(\ell)} = 0$ for $\ell=2,3,4$, hence a total of $M=1375$ parameters are required to describe the model in the form~\eqref{eq:reLUNet}. In our experiment, the components of the parameters $\bA^{(\ell)}, \bb^{(\ell)}$ are chosen randomly. For this network, the \ac{cpl} form provided by the algorithm in~\cite{plonka_spline_2023} is described by only $n=187$ knots $\lambda_k$. Thus, in this example, the \ac{cpl} form provides a more compact representation of the same \ac{fom}.

In order to use the algorithms presented in \Cref{sec:irka} to perform \ac{mor} on~\eqref{eq:CPL}, we have to express $\bF$ as a linear combination of  kernel slices. To this end, consider the kernel
\begin{equation}
    \label{eq:sob1Ker}
    \xi(t, \lambda) \coloneqq 1 + \min\{t - \Delta, \lambda - \Delta\},
\end{equation}
where $\Delta$ is a parameter satisfying $\Delta < \min\left\{ \lambda_0, \ldots, \lambda_{n-1} \right\}$. With this choice, $\xi$ is the reproducing kernel of the Sobolev--like \ac{rkhs} 
\begin{equation*}
    \origspace \coloneqq \left\{ \bG \in L^2(\Omega) \ : \ \bG(\Delta)^2 + \int_{\Omega} \bG'(t)^2\, \du t < \infty \right\},
\end{equation*}
where $\Omega \coloneqq [\Delta, T]$ with $T \in \R$ or $\Omega = [\Delta, \infty)$~\cite[Section 1.2]{Wahba_1990}, and differentiation is understood in the weak sense. Note also, that $\origspace$ is equipped with the inner product
\begin{equation*}
    \langle \bG, \bH \rangle_{\origspace} \coloneqq \bG(\Delta)\bH(\Delta) + \int_{\Omega} \bG'(t) \bH'(t) \,\du t.
\end{equation*}
In practice, the choice of $\Delta$ in~\eqref{eq:sob1Ker} along with the choice of $T$ in $\Omega = [\Delta, T]$ depend on the considered \ac{fom} network $\bF$. In this experiment, they are determined as  $\Delta \coloneqq - \left( |\min\{ \lambda_0, \ldots, \lambda_{n-1} \} - 1| + |\max\{ \lambda_0,\ldots,\lambda_{n-1} \}| \right)$, resulting in the values  $T\coloneqq -\Delta = 70.19$. To express~\eqref{eq:CPL} in terms of $\xi$ from~\eqref{eq:sob1Ker}, observe that $\min\{a,b\} = a - \relu(a-b)$ for all $a,b \in \R$ and therefore,
\begin{equation*}
    \xi_{\lambda_k}(t) = 1 + \min\{ t - \Delta, \lambda_k - \Delta \} = 1 - \Delta + t - \relu(t - \lambda_k) \quad (t \in \Omega, \ k=0,\ldots,n-1).
\end{equation*}
Applying this to $\relu(t-\lambda_k)$ in~\eqref{eq:CPL} gives
\begin{align*}
    \bF(t) &= q_0 + q_1 t + \sum_{k=0}^{n-1} \wres_k \left[ (1-\Delta+t) - \xi_{\lambda_k}(t) \right]
    \\ &= q_0 + (1-\Delta) \sum_{k=0}^{n-1} \wres_k  + \left( q_1 + \sum_{k=0}^{n-1} \wres_k \right) t - \sum_{k=0}^{n-1} \wres_k \xi_{\lambda_k}(t).
\end{align*}
Finally, with $\wres_k = -\res_k$ for $k=0,\ldots,n-1$, we can write~\eqref{eq:CPL} as
\begin{equation}
    \label{eq:kerNet}
    \bF(t) = \beta(t) + \bG(t) = \beta_0 + \beta_1 t + \sum_{k=0}^{n-1} \res_k \xi_{\lambda_k}(t),
\end{equation}
where $\bG \in \fomspaceo \subseteq \origspace$, and
\begin{equation*}
    \beta_0 \coloneqq q_0 + (1-\Delta) \sum_{k=0}^{n-1} \wres_k, \quad \beta_1 \coloneqq q_1 + \sum_{k=0}^{n-1} \wres_k.
\end{equation*}
Similarly to the \ac{svm} experiment in \Cref{subsec:svm}, the affine component  $\beta(t)=\beta_0+\beta_1 t$ is retained exactly, while the \ac{mor} procedure is applied only to the non--affine component $\bG$ in~\eqref{eq:kerNet}.

Note that even though the \ac{fom} $\bG$ now belongs to the \ac{rkhs} $\origspace$, the derivatives $\funder{\xi_{\lambda}}{\lambda}$ and $\funder{\xi_{\lambda}}{\olambda}$ do not exist for any $\lambda \in \Omega$. Consequently, we cannot apply \Cref{thm:orthocond} to evaluate the optimality of a \ac{rom} $\bG_r$. Accordingly, Algorithm~\ref{alg:irkaRKHS} also cannot be applied directly.

To overcome this limitation, we express the \ac{fom} in the \ac{gmt} basis
\begin{equation*}
    \bG = \sum_{k=0}^{n-1} \left\langle \bG, \gmtk\right\rangle_{\origspace} \gmtk
\end{equation*}
and apply the \ac{gmt} transformation from \Cref{def:GMTtransform} to map the problem to an \ac{rkhs} with a sufficiently smooth kernel. Then, the \ac{mor} procedure can be carried out using Algorithm~\ref{alg:cdcRKHS}. For this example, the companion \ac{rkhs} $\ambspace$ is chosen as $\realHtwo$ as given in~\eqref{eq:realH2}. The corresponding domain maps are given as
\begin{equation*}
    \gamma(t) = t / T, \quad \gamma^{-1}(z) =  z T  \quad (t \in [-T, T], z \in (-1, 1)).
\end{equation*}
We observe that choosing these initial kernel centers has a notable effect on the \ac{rom}, although all investigated choices lead to highly accurate \acp{rom}.

We consider two baseline algorithms for comparison, and two different initialization schemes for Algorithm~\ref{alg:cdcRKHS}. First, we compare the efficiency of Algorithm~\ref{alg:cdcRKHS} to the structured pruning method~\cite{blalock_what_2020}. This method acts directly on the model representation~\eqref{eq:reLUNet}. First, it assigns an importance score to each neuron in the hidden network layers. This score is a heuristic to measure the effect of the neuron with respect to the incoming and outgoing information. The importance score is given by
\begin{equation*}
    \left\| \bA^{(\ell)}_j \right\|_{2} \left\|\left(\bA^{(\ell+1)}\right)^\top_j \right\|_{2} \quad (j=0,\ldots,k_{\ell}-1),
\end{equation*}
where $ \bA^{(\ell)}_j$ denotes the $j$th column of $ \bA^{(\ell)}$. In a single step of the pruning method, the neuron with the lowest importance score is pruned from the network, that is, the corresponding column of $\bA^{(\ell)}$ and row of $\bA^{(\ell+1)}$ are removed. This process continues iteratively until a defined compression ratio is achieved.

Since the proposed approach acts on~\eqref{eq:kerNet}, we also compare the effectiveness of Algorithm~\ref{alg:cdcRKHS} with a kernel slice fitting scheme. In particular, we consider a variant of the \ac{eim} adapted to fitting kernel slices constructed from $\xi$ in~\eqref{eq:sob1Ker}. \ac{eim} was originally introduced in~\cite{barrault_empirical_2004} to approximate the solution of parametric partial differential equations, but has since been widely used to approximate nonlinear operators in general. In the spline--kernel fitting context of this experiment, it greedily selects the knots $\tlambda_j$, $j=0,\ldots,r-1$ describing the \ac{rom}. The considered \ac{eim} variant is summarized in Algorithm~\ref{alg:eim}. We note that when applied to kernel slices, the algorithm is also known as greedy kernel interpolation, or kernel matching pursuit~\cite{vincent_kernel_2002}. 

\begin{algorithm}[tbp]
\caption{Kernel--\ac{eim} for piecewise--linear spline approximation}
\label{alg:eim}
\begin{algorithmic}[1]
\Require \ac{fom} $\bG$ queryable on $[\Delta, T]$,  kernel $\xi$ from~\eqref{eq:sob1Ker}, reduced order $r$.
\Ensure \ac{rom} $\bG_r$.
\State $\bE_0 \gets \bG$.
\State $\tlambda_0 \gets \argmax_{t \in [\Delta, T]} |\bE_0(t)|$.
\State $c_0 \gets \bG \big(\tlambda_0 \big) / \xi\big(\tlambda_0, \tlambda_0 \big)$.
\State $\bG_1 \gets c_0 \xi \big(\cdot, \tlambda_0 \big)$.
\For{$j = 1,\ldots, r-1$}
    \State $\bE_{j} \gets \bG - \bG_{j}$.
    \State $\tlambda_j \gets \argmax_{t \in [\Delta,T]} |\bE_{j}(t)|$.
    \State Assemble $\mathbf{K}^{(j)} = \left[\xi\big(\tlambda_i, \tlambda_k \big)\right]_{i,k=0}^{j}$ and $\mathbf{g}^{(j)} = \left[\bG\big(\tlambda_i\big)\right]_{i=0}^j$.
    \State Solve $\mathbf{K}^{(j)} \mathbf{c}^{(j)} = \mathbf{g}^{(j)}$ for $\bc^{(j)} \coloneqq \big[ c_0^{(j)},\ldots, c_j^{(j)}\big]^\top$.
    \State $\bG_{j+1} \gets \sum_{i=0}^j c_i^{(j)} \xi_{\tlambda_i}$.
\EndFor
\State \textbf{return} $\bG_r$.
\end{algorithmic}
\end{algorithm}
We use Algorithm~\ref{alg:eim} in two distinct ways. First, it is used as a baseline method for Algorithm~\ref{alg:cdcRKHS}, when the latter is used with a quasi--random heuristic initialization. Then, we use \ac{eim} to initialize a starting point $\tilde{\blambda}^{(0)} \in [\Delta, T]^r$ for Algorithm~\ref{alg:cdcRKHS}. With this ``fine tuning", we ensure that the final \ac{rom} satisfies our $\origspace$ optimality conditions.

In summary, we consider the following models:
\begin{itemize}
    \item $\bF$: the \ac{fom} neural network expressed according to~\eqref{eq:reLUNet},
    \item $\bF^{\mathrm{prun}}$: \ac{rom} obtained by applying pruning directly to $\bF$,
    \item  $\bG$: the \ac{fom} kernel representation of the non--affine part obtained according to~\eqref{eq:kerNet},
    \item $\Gweakproj{r}$: projected \ac{rom} obtained using \ac{gmt} transformation. In particular, $\Gweakproj{r} \coloneqq  \cP_{V_r^{\tilde{\blambda}}}\left[ \bG \right]$, where $V_r^{\tilde{\blambda}} = \myspan\left\{ \xi_{\tlambda_0}, \ldots, \xi_{\tlambda_{r-1}} \right\}$ and the kernel centers $\tlambda_j$, $j=0,\ldots,r-1$ are acquired using Algorithm~\ref{alg:cdcRKHS} and heuristic initialization,
    \item $\bG_r^{\mathrm{econ\text{-}EIM}}$: projected \ac{rom} obtained using Algorithm~\ref{alg:cdcRKHS} with \ac{eim} initialization, i.e., where the initial kernel centers $\tlambda_j$, $j=0,\ldots,r-1$ are found with the greedy Algorithm~\ref{alg:eim},
    \item $\bG_r^{\mathrm{EIM}}$: baseline \ac{rom} obtained with Algorithm~\ref{alg:eim}.
\end{itemize}

To evaluate the approximation quality, we first consider the mean squared error
\begin{equation}
    \label{eq:relLsq}
    \cE_{\mathrm{MSE}}(\bH_r; \bH) \coloneqq \frac{1}{K} \sum_{k=0}^{K-1} \left( \bH(t_k) - \bH_r(t_k)  \right)^2,
\end{equation}
where $\{t_k\}_{k=0}^{K-1}$ is a $K$--point equidistant grid of $\Omega$. Here, we choose $K=1000$. 
For the pruning baseline, we evaluate
$\cE_{\mathrm{MSE}}(\bF^{\mathrm{prun}};\bF)$, whereas for the
kernel--based \acp{rom} we evaluate
$\cE_{\mathrm{MSE}}(\Gweakproj{r};\bG)$.
Since the affine component $\beta$ is retained exactly, the latter is
equivalent to evaluating the \ac{mse} between the functions $\bF=\beta+\bG$ and  $\bF_r = \beta+\bG_r$, where $\bG_r$ can be any of the \acp{rom} discussed above.

In addition, for the kernel--based \acp{rom}  we report the relative \ac{rkhs} error as defined in \eqref{eq:relErr}, where again, $\bG_r$ denotes an arbitrary \ac{rom}.
This metric is not reported  for $\bF^{\mathrm{prun}}$, because no guarantees can be given that its \ac{cpl} form 
\begin{equation*}
    \bF^{\mathrm{prun}}(t) \eqqcolon \tq(t) + \bG^{\mathrm{prun}}(t) = \tq_0 + \tq_1 t + \sum_{j=0}^{\tn-1} \wtres_j \relu\big(t - \tilde{\mu}_j \big)
\end{equation*}
satisfies $\Delta < \min \big\{\tilde{\mu}_0,\ldots,\tilde{\mu}_{\tn-1}\big\}$ whose violation would result in $\bG^{\mathrm{prun}} \not \in \origspace$.

In our experiments, we consider the reduced orders
\begin{equation}
    \label{eq:nnorder}
    r \in \{ 50, 60, 70, 80, 90, 100 \}
\end{equation}
and recall that the \ac{fom} $\bG$ in~\eqref{eq:kerNet} can be described with $n=187$ knots. We already mentioned that the deep ReLU form of the \ac{fom} in~\eqref{eq:reLUNet} is defined by $M=1375$ parameters. To ensure a fair comparison for pruning, we compare the \ac{mse} of \acp{rom} that provide the same overall compression ratio
\begin{equation*}
    \CR = \frac{\textrm{Total number of ROM parameters}}{\textrm{Total number of FOM parameters}} \cdot 100\%,
\end{equation*}
where the deep network or the kernel--slice formulations are considered depending on the \ac{mor} method. Therefore, the different \ac{mor} schemes are compared along the following (approximate) compression ratios determined by~\eqref{eq:nnorder}:
\begin{equation*}
    \CR \in \{ 27.3, 32.6, 38.0, 43.3, 48.7, 54.0 \}\%.
\end{equation*}
\Cref{tab:nn} summarizes the results of the experiment. To better resolve the differences between the kernel--based \acp{rom},
we additionally report the improvement relative to the \ac{eim} baseline,
defined as
\begin{align*}
    \cI^{\mathrm{econ}}_r &\coloneqq \log_{10}\left(\cE\left(\Gweakproj{r}\right)\right) - \log_{10}\left(\cE\left(\bG_r^{\mathrm{EIM}}\right)\right), \\ \cI^{\mathrm{econ\text{-}EIM}}_r &\coloneqq \log_{10}\left(\cE\left(\bG_r^{\mathrm{econ\text{-}EIM}}\right)\right) - \log_{10}\left(\cE\left(\bG_r^{\mathrm{EIM}}\right)\right),
\end{align*}
where $\cE$ denotes either the \ac{mse} in~\eqref{eq:relLsq} or the relative \ac{rkhs} error in~\eqref{eq:relErr}. A negative value indicates an improvement over the \ac{eim} baseline, whereas a positive value indicates that \ac{eim} achieves the smaller error, and the absolute value indicates the differences in the order of magnitudes of the errors.
Since $\Gweakproj{r}, \bG_r^{\mathrm{EIM}}$, and $\bG_r^{\mathrm{econ\text{-}EIM}}$ all provide high quality approximations, these improvement measures help resolve the small differences between their performances.
This is further confirmed in \Cref{fig:NNapr}, which shows the reconstruction of $\bF$ over $[\Delta, T]$. In contrast, the considered pruning approach performs poorly for the randomly parameterized \ac{fom} used in this experiment.
Pruning's inability to reduce neural networks with random weights has already been observed in~\cite{blalock_what_2020}.

\begin{table}[tbp]
    \centering
    \caption{Proposed and baseline \acp{rom} for \ac{mor} on a ReLU network.}
    \label{tab:nn}

    \resizebox{\textwidth}{!}{%
    \begin{tabular}{c c c c c c c c}
        \toprule
        Metric
        & $\CR\,(\%)$
        & $\Gweakproj{r}$
        & $\bG_r^{\mathrm{econ\text{-}EIM}}$
        & $\bG_r^{\mathrm{EIM}}$
        & $\bF^{\mathrm{prun}}$
        & $\cI_r^{\mathrm{econ}}$
        & $\cI_r^{\mathrm{econ\text{-}EIM}}$ \\
        \midrule

        \multirow{6}{*}{$\cE^r_{\xi, \mathrm{rel}}$}
        & 27.27
        & 0.08018
        & 0.03824
        & 0.03824
        & --
        & $0.3216$
        & $-1.1856 \times 10^{-8}$ \\

        & 32.62
        & 0.07276
        & 0.03701
        & 0.03701
        & --
        & $0.2936$
        & $ -1.2626 \times 10^{-8}$ \\

        & 37.97
        & 0.07132
        & 0.03215
        & 0.03215
        & --
        & $0.3461$
        & $-1.6617 \times 10^{-8}$ \\

        & 43.32
        & 0.07111
        & 0.03156
        & 0.03156
        & --
        & $0.3528$
        & $-1.7235 \times 10^{-8}$ \\

        & 48.66
        & 0.07109
        & 0.02752
        & 0.02752
        & --
        & $0.4121$
        & $-2.2664 \times 10^{-8}$ \\

        & 54.01
        & 0.07053
        & 0.02659
        & 0.02659
        & --
        & $0.4236$
        & $-2.4156 \times 10^{-8}$ \\

        \midrule

        \multirow{6}{*}{$\cE_{\mathrm{MSE}}$}
        & 27.27
        & 0.03084
        & 0.01356
        & 0.01356
        & $1.949 \times 10^5$
        & $0.3569$
        & $-6.6911 \times 10^{-11}$ \\

        & 32.62
        & 0.01350
        & 0.01354
        & 0.01354
        & $1.753 \times 10^{6}$
        & $ -1.354 \times 10^{-3}$
        & $-5.9370 \times 10^{-11}$ \\

        & 37.97
        & 0.01112
        & 0.01306
        & 0.01306
        & $2.641 \times 10^{6}$
        & $-0.0701$
        & $-5.4816 \times 10^{-11}$ \\

        & 43.32
        & 0.01027
        & 0.01305
        & 0.01305
        & $1.192 \times 10^{6}$
        & $-0.1041$
        & $-5.1098 \times 10^{-11}$ \\

        & 48.66
        & 0.01014
        & 0.01297
        & 0.01297
        & $1.242 \times 10^{6}$
        & $-0.1070$
        & $-4.2457 \times 10^{-11}$ \\

        & 54.01
        & 0.009744
        & 0.01293
        & 0.01293
        & $6.573 \times 10^{5}$
        & $-0.1229$
        & $-3.6601 \times 10^{-11}$ \\

        \bottomrule
    \end{tabular}%
    }
\end{table}

\begin{figure}
    \centering
    \includegraphics{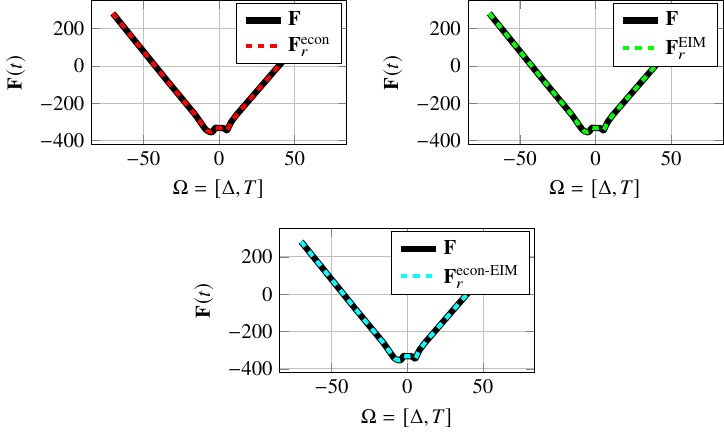}
    \caption{Approximation of the FOM neural network $\bF$ with the considered kernel expansion based ROMs. Here, $\bF_r^{\mathrm{econ}}(t) \coloneqq \Gweakproj{r}(t) + \beta(t)$, $\bF_r^{\mathrm{EIM}}(t) \coloneqq \bG_r^{\mathrm{EIM}}(t) + \beta(t)$, and $\bF_r^{\mathrm{econ\text{-}EIM}}(t) \coloneqq \bG_r^{\mathrm{econ\text{-}EIM}}(t) + \beta(t)$ are given in~\eqref{eq:kerNet}. In this figure, $r=70$ is considered.}
    \label{fig:NNapr}
\end{figure}

The spline--fitting--based \acp{rom} $\Gweakproj{r}, \bG_r^{\mathrm{econ\text{-}EIM}}$, and $\bG_r^{\mathrm{EIM}}$ provide very similar results. In order to explore subtle differences between them, we consult the last two columns of \Cref{tab:nn}. We observe that quasi--random initialization of Algorithm~\ref{alg:cdcRKHS} yields \acp{rom} with slightly worse $\origspace$ norm errors than $\bG_r^{\mathrm{EIM}}$ and $\bG_r^{\mathrm{econ\text{-}EIM}}$. 
On the other hand, when evaluated over a fine grid, the corresponding \ac{mse} scores are improved for $\Gweakproj{r}$ when compared to $\bG_r^{\mathrm{EIM}}$ and $\bG_r^{\mathrm{econ\text{-}EIM}}$ for most of the considered compression ratios. Recall that $\bG_r^{\mathrm{econ\text{-}EIM}}$ is obtained by initializing Algorithm~\ref{alg:cdcRKHS} from the kernel centers provided by \ac{eim}. When compared against $\bG_r^{\mathrm{EIM}}$, this model provides 
marginal improvements both in terms of
$\cE_{\xi,\mathrm{rel}}^{r}$ and $\cE_{\mathrm{MSE}}$.
 These improvements remain very small over the considered range of compression ratios. In fact, we find that $\bG_r^{\mathrm{econ\text{-}EIM}}$ and $\bG_r^{\mathrm{EIM}}$ match up to $9$--$12$ decimal digits across every considered compression ratio, indicating that the \ac{eim}--initialized knots already lie extremely close to the $\origspace$--optimal configuration found by Algorithm~\ref{alg:cdcRKHS}.
 We conclude that Algorithm~\ref{alg:cdcRKHS} can slightly refine
$\bG_r^{\mathrm{EIM}}$ when initialized from the \ac{eim} kernel centers.
However, both the proposed methods and the baseline
$\bG_r^{\mathrm{EIM}}$ provide high quality approximations to the \ac{fom}.
Hence, spline--fitting--based reduction of ReLU networks seems reasonable in cases when a spline representation of the network can be obtained.


We notice the following interesting behavior of the \ac{eim} baseline model $\bG_r^{\mathrm{EIM}}$. Although the greedy iteration in Algorithm~\ref{alg:eim} aims to reduce the $L^{\infty}$ error of the \ac{rom} with the addition of each new knot, we observe that $\bG_r^{\mathrm{EIM}}$ also approximately satisfies the equilibrium--based conditions of \Cref{thm:equibCond}.
This is reflected in \Cref{tab:equibNN}, which shows the
magnitude of the equilibrium function evaluated at the kernel centers
determining each \ac{rom}.
In this experiment, the \ac{eim}--\acp{rom} also approximately satisfy the
necessary conditions of $\origspace$ optimality in the \ac{gmt} transformation sense. This is notable because Algorithm~\ref{alg:eim} does not explicitly enforce these conditions.
On the other hand, the \acp{rom} obtained with Algorithm~\ref{alg:cdcRKHS} satisfies the equilibrium--based optimality conditions to even higher numerical accuracy. 

\begin{table}[tbp]
    \centering
    \caption{Maximum equilibrium function magnitude $\left\| \left[\hat{\cD}\left(\gamma\big(\tilde{\lambda}_0 \big),\ldots,\gamma\big(\tilde{\lambda}_{r-1} \big)\right)\right]^\top \right\|_{\infty}$ in the companion \ac{rkhs} $\realHtwo$ for different \acp{rom}. These results were computed using MATLAB's variable precision arithmetic using $100$ digits.}
    \label{tab:equibNN}
    \begin{tabular}{ccccc}
        \toprule
         $r$ & $\Gweakproj{r}$ & $\bG_r^{\mathrm{econ\text{-}EIM}}$ & $\bG_r^{\mathrm{EIM}}$ \\
        \midrule
         50 & $1.59 \times 10^{-29}$ & $1.43 \times 10^{-34}$ & $1.54 \times 10^{-18}$ \\
         60 & $2.72 \times 10^{-36}$ & $3.70 \times 10^{-41}$ & $1.24 \times 10^{-21}$ \\
         70 & $1.30 \times 10^{-43}$ & $4.23 \times 10^{-47}$ & $1.74 \times 10^{-24}$ \\
         80 & $4.81 \times 10^{-57}$ & $6.98 \times 10^{-53}$ & $2.82 \times 10^{-27}$ \\
         90 & $1.75 \times 10^{-62}$ & $2.22 \times 10^{-58}$ & $5.92 \times 10^{-30}$ \\
         100 & $9.49 \times 10^{-68}$ & $1.73 \times 10^{-64}$ & $7.35 \times 10^{-33}$ \\ \bottomrule
    \end{tabular}
\end{table}

In this experiment, we presented a viable application of Algorithm~\ref{alg:cdcRKHS} to reduce $\R \to \R$ type deep ReLU networks. We conclude that, for the randomly parameterized \ac{fom},
the proposed spline--based reduction methods substantially outperform
the considered pruning baseline.
The presented experiment also provides a strong case for further investigations into using spline--based representations of ReLU networks in order to manipulate already trained models. In addition, we observe that an effective greedy spline fitting technique, \ac{eim}, also achieves models which approximately satisfy the \ac{gmt}--transformation--based optimality conditions presented in this paper. Finally, we emphasize that this application is made possible through the use of \ac{gmt} transformations, because the native \ac{rkhs} containing the ReLU network does not have a differentiable kernel. Consequently, without the use of the \ac{gmt} transformation, the optimality conditions proposed in \Cref{sec:orthCond} cannot be directly evaluated.   

\subsection{PDE discretization}
\label{subsec:pdes}

We next demonstrate how Algorithm~\ref{alg:irkaRKHS} can be used to solve \acp{pde}. This example is once again meant to emphasize the scope of our results, because the problem considered is not a traditional \ac{mor} task. A systematic analysis of extending the framework for other \acp{pde} is left for future work.  
In this experiment, we consider the fourth--order Euler--Bernoulli beam
equation
\begin{equation}
    \label{eq:pde}
    \bG^{(4)}(x) = \bF(x),
    \quad x \in \Omega \coloneqq (0,1),
\end{equation}
with clamped--free boundary conditions
\begin{equation}
    \label{eq:bc}
    \bG(0)=\bG'(0)=0,
    \quad
    \bG''(1)=\bG^{(3)}(1)=0.
\end{equation}

For the numerical experiment, we construct the problem by prescribing the
analytic solution
\begin{equation}
    \label{eq:trueSol}
    \bG(x)
    =
     x^2(1-x)^2 \eu^{-\frac{(x-3/4)^2}{2\sigma^2}},
\end{equation}
with $\sigma=\frac{1}{20}$, and defining the corresponding right--hand side as
\begin{equation*}
    \bF(x)
    \coloneqq
    \bG^{(4)}(x)
    =
    \left[
        x^2(1-x)^2 \eu^{-\frac{(x-3/4)^2}{2\sigma^2}}
    \right]^{(4)}.
\end{equation*}
Hence, $\bG$ in~\eqref{eq:trueSol} is known by construction and can be
used to assess the accuracy of the numerical approximations.

Similarly to the \ac{fem}, our \ac{mor} based solution considers the weak form of the above equation:
\begin{equation*}
    \int_\Omega \bG^{(4)}(x) v(x)\, \du x = \int_\Omega \bF(x) v(x)\, \du x,
\end{equation*}
where $v$ belongs to the test space
\begin{equation}
    \label{eq:pdeRKHS}
    \origspace \coloneqq \left\{ v \in H^2(\Omega) \ : \ v(0) = v'(0) = 0 \right\}.
\end{equation}
Here, $H^2(\Omega)$ denotes the Sobolev--Hilbert space
\begin{equation*}
    H^2(\Omega) \coloneqq \left\{ v \in L^2(\Omega) \ : \ v', v'' \in L^2(\Omega),  \sum_{n=0}^{2} \big\| v^{(n)} \big\|_{L^2(\Omega)}^2 < \infty \right\},
\end{equation*}
where differentiation is understood in the weak sense. Following integration by parts and applying the boundary condition~\cite[Chapter 4]{Reddy_1993} gives the weak equation
\begin{equation}
    \label{eq:weakpde}
    \underbrace{\int_{\Omega} \bG''(x) v''(x)\, \du x}_{\eqqcolon a(\bG, v)} = \underbrace{\int_{\Omega} \bF(x) v(x)\, \du x}_{\eqqcolon f(v)} \quad (v \in \origspace).
\end{equation}
Note that $\origspace$ is a Hilbert space equipped with the inner product
\begin{equation*}
    \left\langle \bG, \bH \right\rangle_{\origspace} \coloneqq \int_{\Omega} \bG''(x) \bH''(x) \, \du x \quad (\bG, \bH \in \origspace). 
\end{equation*}
We observe, see e.g.,~\cite[Eq. (1.2.4)]{Wahba_1990} that $\origspace$ is an \ac{rkhs} with the kernel
\begin{equation}
    \label{eq:sobker}
    \xi(x,\lambda) = \frac{1}{2} \min\{x, \lambda\}^2 \max\{x, \lambda\} - \frac{1}{6} \min\{x, \lambda\}^3 \quad (x, \lambda \in \Omega).
\end{equation}
This coincides with Green's function corresponding to the 4th order differential operator~\cite{stakgold_greens_2011} and hence, the kernel slice generator $\lambda \mapsto \xi_{\lambda}$ is twice continuously differentiable on $\Omega$. Note that in our example, the test and ansatz space are the same since our analytic solution~\eqref{eq:trueSol} satisfies $\bG \in \origspace$. For the proposed method this is a necessity, since we are looking to obtain a weak solution $\bG_r \in \origspace$ which minimizes $\| \bG - \bG_r \|_{\origspace}$.

Let us assume for now, that we have given kernel centers
$\tilde{\blambda} \in \Omeganeq{r}$ for some $r \in \N$. We can approximate the unknown solution $\bG$ with the orthogonal (Galerkin) projection $\bG_r \coloneqq \cP_{V_r^{\tilde{\blambda}}}\left[ \bG \right]$, where $V_r^{\tilde{\blambda}} = \myspan \left\{ \xi_{\tlambda_0}, \ldots, \xi_{\tlambda_{r-1}} \right\} \subseteq \origspace$.
In fact, $\bG_r$ can be computed using~\eqref{eq:weakpde} as follows. Choosing the kernel slices $\xi_{\tlambda_j}$, $j=0,\ldots,r-1$ as test functions and since $\bG_r(x) = \sum_{k=0}^{r-1} \tres_k \xi_{\tlambda_k}(x)$ for some linear coefficients $\tres_k$, $k=0,\ldots,r-1$, we obtain 
\begin{align*}
    \int_{\Omega} \bG_r''(x) \xi_{\tlambda_j}''(x) \,\du x &=  \int_{\Omega} \left( \sum_{k=0}^{r-1} \tres_k \xi_{\tlambda_k}''(x) \right) \xi_{\tlambda_j}''(x) \,\du x \\ &= \sum_{k=0}^{r-1} \tres_k \int_{\Omega} \xi_{\tlambda_k}''(x) \xi_{\tlambda_j}''(x)\, \du x  = \int_{\Omega} \bF(x) \xi_{\tlambda_j}(x)\, \du x.
\end{align*}
Then, $\tres \coloneqq \left[\tres_0,\ldots,\tres_{r-1}\right]^\top$ is obtained as the solution to the linear system
\begin{equation*}
    \bA \boldsymbol{\tres} = \boldf,
\end{equation*}
where $\bA := \left[a\big(\xi_{\tlambda_k}, \xi_{\tlambda_j}\big)\right]_{k,j=0}^{r-1}$ and $\boldf \coloneqq \left[f\big(\xi_{\tlambda_j}\big)\right]_{j=0}^{r-1}$. We can compute the entries of $\bA$ exactly using the identity
\begin{equation*}
    \xi_{\lambda}''(x) = \begin{cases}
        \lambda - x, & \text{if } x \leq \lambda, \\
        0, & \text{if }  x > \lambda.
    \end{cases}
\end{equation*}
The components of $\boldf$ can be approximated using an appropriate numerical quadrature. In this experiment, we use the composite trapezoidal rule with $N=1000$ nodes.

Note that the bilinear form $a(\cdot, \cdot)$ on the left hand side of~\eqref{eq:weakpde} can be interpreted as an inner product $\left\langle \cdot, \cdot \right\rangle_{\origspace}$. In particular, $a(\cdot,\cdot)$ is coercive with coercivity constant 1, that is,
\begin{equation*}
    a(\bH, \bH) = \int_{\Omega} \bH''(x) \bH''(x) \, \du x \geq {\| \bH \|}_{\origspace}^2 \quad (\bH \in \origspace).
\end{equation*}
Moreover, by the Cauchy--Schwarz inequality, the
continuity constant of $a(\cdot,\cdot)$ is also equal to one.

The Galerkin solution satisfies
the Galerkin orthogonality
\begin{equation*}
    a\left(\bG-\bG_r,\bH\right)
    =
    \left\langle
        \bG-\bG_r,\bH
    \right\rangle_{\origspace}
    =
    0
    \quad
    \left(\bH\in V_r^{\tilde{\blambda}}\right).
\end{equation*}
Hence, the Galerkin solution coincides with the orthogonal projection $\bG_r =
    \cP_{V_r^{\tilde{\blambda}}}[\bG]$.
Furthermore, since both the continuity and coercivity constants are equal
to one, C\'ea's lemma gives
\begin{equation*}
    {\|\bG-\bG_r\|}_{\origspace}
    =
    \inf_{\bH\in V_r^{\tilde{\blambda}}}
    {\|\bG-\bH\|}_{\origspace}.
\end{equation*}
Thus, in this setting, the Galerkin solution is exactly the
$\origspace$--orthogonal best approximation of $\bG$ in
$V_r^{\tilde{\blambda}}$.

Up to this point, we established a \ac{fem}--like approach, where we approximate the true solution of~\eqref{eq:weakpde} using kernel slices instead of the usual piecewise polynomial basis. Such an approach is not without precedent, see, e.g.,~\cite{ xie_numerical_2008}. The novelty of the example comes from adaptively choosing the  "mesh" defined by the $r$ kernel centers $\tilde{\blambda} \in \Omeganeq{r}$ using Algorithm~\ref{alg:irkaRKHS}. That is, we interpret the solution $\bG$ and its approximation $\bG_r \in V_r^{\tilde{\blambda}}$ as a \ac{fom} and a \ac{rom}, respectively. 

In contrast to a usual \ac{mor} problem, we do not have access to the kernel centers defining $\bG$. However, to apply Algorithm~\ref{alg:irkaRKHS}, we only need to be able to evaluate the \emph{equilibrium function} in~\eqref{eq:EquibFun} (and its derivative due to our choice of Newton iteration), and no explicit knowledge of $\bG$ is required. However, evaluating the equilibrium function $\cD : \Omega^r \to \C^r$ and determining an optimal $\tilde{\blambda}$ using Algorithm~\ref{alg:irkaRKHS} requires access to $\bG'(x)$ and $\bG''(x)$ for all $x \in \Omega$, since by~\eqref{eq:EquibFun},
\begin{equation*}
    \cD\big(\tilde{\blambda}\big) = \left[ \bG'\big(\tlambda_0\big) - \cP_{V_r^{\tilde{\blambda}}}\left[ \bG \right]'\big(\tlambda_0\big), \ldots, \bG'\big(\tlambda_{r-1}\big) - \cP_{V_r^{\tilde{\blambda}}}\left[ \bG \right]'\big(\tlambda_{r-1}\big)\right]^\top.
\end{equation*}
While we cannot directly compute these values from~\eqref{eq:pde} or~\eqref{eq:weakpde}, we can leverage the method of \ac{dwr} introduced in~\cite{becker_optimal_2001}. We refer to~\cite{becker_optimal_2001} for a thorough discussion of \ac{dwr}, here we only state the following consequence. 
\begin{lemma}[Equilibrium function evaluation]
    \label{thm:dwr}
    Consider the weak problem~\eqref{eq:weakpde} and let $\origspace$ be the \ac{rkhs} defined in~\eqref{eq:pdeRKHS} with kernel $\xi$ from~\eqref{eq:sobker}. For a given $\tilde{\blambda} \in \Omeganeq{r}$, let 
    \begin{equation*}
    \bG_r = \cP_{V_r^{\tilde{\blambda}}}\left[ \bG \right] = \sum_{j=0}^{r-1} c_j \xi_{\tlambda_j}    
    \end{equation*}
    be the Galerkin projection of the true solution $\bG$ onto the  
        $V_r^{\tilde{\blambda}} = \myspan\left\{ \xi_{\tlambda_0}, \ldots, \xi_{\tlambda_{r-1}} \right\} \subseteq \origspace$.
    Let $\tilde{\blambda} \in \Omeganeq{r}$.  Then,
    \begin{align*}
        \left[\cD\big(\tilde{\blambda}\big)\right]_j &= \bG'\big(\tlambda_j\big) - \bG_r'\big(\tlambda_j\big) \\ &= \frac{1}{2} \int_0^{\tlambda_j} \bF(x) x^2\, \du x + \int_{\tlambda_j}^1 \bF(x) \left( \tlambda_j x - \frac{1}{2} \tlambda_j^2 \right)\, \du x - \bG_r'\big(\tlambda_j\big),
    \end{align*}
    and
    \begin{equation*}
        \left[\frac{\partial \cD}{\partial \tlambda_i}(\bx)\right]_j = \begin{cases}
            -\funder{\bG_r}{\tlambda_i}'\big(\tlambda_j\big), & \text{if } i \neq j, \\
            \bG''\big(\tlambda_j\big) - \bG_r''\big(\tlambda_j\big) = \int_{\tlambda_j}^1 \bF(x) \big(x - \tlambda_j\big) \,\du x - \bG_r''\big(\tlambda_j\big), & \text{if } i = j,
        \end{cases}    
    \end{equation*}
    where $i,j=0,\ldots,r-1$.
\end{lemma}
\begin{proof}
    The proof is given in~\Cref{app:dwr}.
\end{proof}
Using the formulas in \Cref{thm:dwr}, we can apply Algorithm~\ref{alg:irkaRKHS} to determine the $r$ kernel
centers $\tilde{\blambda}$ according to the optimality conditions of
\Cref{sec:orthCond}.
In this way, the points $\tilde{\blambda} \in \Omeganeq{r}$ that define $\bG_r$ can be interpreted as an ``optimal interior mesh". Finally, we note that the integrals in \Cref{thm:dwr} can be approximated using numerical quadratures. 

 Using the right--hand side in~\eqref{eq:pde} and the boundary conditions
in~\eqref{eq:bc}, we now proceed to examine the performance of the
proposed method. The proposed adaptive algorithm is given in Algorithm~\ref{alg:pdeSolver}. 
In the final step, we additionally include the kernel slice $\xi_1$ in the reduced space, i.e.,
\begin{equation}
\label{eq:Vr1}
        \Vn{r+1}{(\tilde{\blambda},1)}
    \coloneqq
    \myspan\left\{
        \xi_{\tlambda_0},\ldots,
        \xi_{\tlambda_{r-1}},\xi_1
    \right\}.
\end{equation}
By \Cref{thm:interpGMT}, the orthogonal projection onto
$ \Vn{r+1}{(\tilde{\blambda},1)}$ interpolates the target solution at the kernel centers.
Hence, including $\xi_1$ ensures interpolation of the exact solution
at $x=1$, which improves the approximation near the right boundary.
Throughout this experiment, however, $r$ refers to the number of
adaptive interior kernel centers, and we denote the final approximation
by $\bG_r$. No additional kernel slice at $x=0$ is required, since
$\xi_0\equiv0$ and the essential boundary conditions at $x=0$ are
already incorporated in the definition of $\origspace$. In this experiment we use $r=5$.
We note again that the corresponding orthogonal projection can be
computed without explicit access to $\bG$ by solving the Galerkin
system associated with~\eqref{eq:weakpde} on $ \Vn{r+1}{(\tilde{\blambda},1)}$.

\begin{algorithm}[tbp]
\caption{\ac{rkhs}--\ac{pde} solver for \eqref{eq:pde}, \eqref{eq:bc} with $\origspace$--optimal mesh}
\begin{algorithmic}[1]
\Require Right--hand side $\bF$ of \eqref{eq:pde},
kernel $\xi$ of the \ac{rkhs} $\cH$, kernel centers $\tilde{\blambda}^{(0)} \in \Omeganeq{r}$, damping factor of Newton's method $\beta \in (0,1]$.
\Ensure $\cH$--optimal approximate solution $\bG_r \in \romspaceo$.
\State $m \gets 0$.
\State $\varepsilon^{(0)}\gets  \infty$.
\While{$\varepsilon^{(m)}> \varepsilon$}
    \State $\bJ \gets \left[\frac{\partial \cD \left( \tilde{\blambda}^{(m)} \right)_i}{\partial \overline{\tlambda_j}} \right]_{i,j=0}^{r-1}$. \Comment{Use \Cref{thm:dwr} and $\bF$.}
    \State $\tilde{\blambda}^{(m+1)} \gets \tilde{\blambda}^{(m)} -\beta \bJ^{-1} \cD \left( \tilde{\blambda}^{(m)} \right)$.
    \State $\varepsilon^{(m)} \gets \left\| \tilde{\blambda}^{(m+1)} - \tilde{\blambda}^{(m)} \right\|_{\infty}. $
    \State $m \gets m+1$.
\EndWhile
\State $\bG_r \gets \cP_{\Vn{r+1}{\big(\tilde{\blambda}^{(m)},1\big)}}\left[ \bG \right]$.
\Comment{Projection to~\eqref{eq:Vr1}.}
\State \Return $\bG_r$.
\end{algorithmic}
\label{alg:pdeSolver}
\end{algorithm}


For comparison, the \ac{fem} meshes are constructed using the same number $r$ of interior mesh points together with the two boundary points $x=0$ and $x=1$. 
In addition, we solve~\eqref{eq:weakpde} using traditional \ac{fem} with cubic Hermite elements. In particular, we consider three baseline solutions:
\begin{itemize}
    \item $\bG_r^{\mathrm{FEM}}$: \ac{fem} solution generated using $r=5$ interior mesh points together with the two boundary points $x=0$ and $x=1$, resulting in an equidistant mesh with $r+2=7$ points.
     \item $\bG_r^{\mathrm{FEM\text{-}adapt}}$: \ac{fem} solution generated using $r=5$ interior mesh points concentrated in $[0.6,0.9]\subseteq\Omega$, which contains the main lobe of the true solution~\eqref{eq:trueSol} (see also \Cref{fig:PDEsols}). The two boundary points $x=0$ and $x=1$ are additionally included  to ensure compliance with~\eqref{eq:bc}, resulting in a mesh with $r+2=7$ points.
    \item $\bG_{2r}^{\mathrm{FEM}}$: \ac{fem} solution generated using $2r=10$ interior mesh points together with the two boundary points, resulting in an equidistant mesh with $2r+2=12$ points.
\end{itemize}
The results of this experiment are illustrated in \Cref{fig:PDEsols}. We evaluate each solution in terms of the relative $L^2(\Omega)$ error defined as
\begin{equation*}
    \cE_{L^2(\Omega)}(\bG_r) = \frac{\left\| \bG - \bG_r \right\|_{L^2(\Omega)}}{\left\| \bG \right\|_{L^2(\Omega)}} \approx \frac{\sqrt{\sum_{k=0}^{N-1} \big| \bG(x_k) - \bG_r(x_k) \big|^2}}{\sqrt{\sum_{k=0}^{N-1} \big| \bG(x_k)\big|^2}},
\end{equation*}
where for this experiment, we use an equidistant partitioning of $\Omega$ with $N=1000$ points. 

\begin{figure}[tb]
    \centering
    \includegraphics{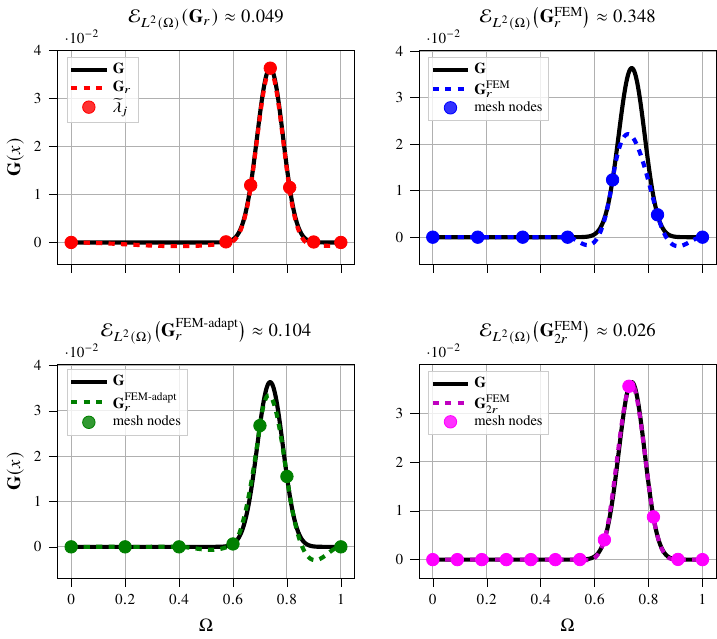}
    \caption{Solution to the 4th--order Euler beam equation using the proposed adaptive kernel mesh based approach and FEM for $r=5$. }
    \label{fig:PDEsols}
\end{figure}

\Cref{fig:PDEsols} demonstrates the utility of the proposed method for the considered problem~\eqref{eq:pde}. The solution obtained with Algorithm~\ref{alg:pdeSolver} significantly outperforms the \ac{fem}--based solutions for $r=5$ in terms of relative $L^2(\Omega)$ error. In fact, as observed in the figure, \ac{fem} achieves an error magnitude similar to that of the proposed approach only for $2r = 10$ elements. Notably, to find the solution $\bG_r$, we choose $\beta = 0.1$ in Algorithm~\ref{alg:pdeSolver} and initialize the kernel centers equidistantly $\tlambda_j^{(0)} \in [0.6, 0.9]$ for $j=0,\ldots,r-1$. Since the true solution $\bG$ vanishes quickly as we approach the boundary, failing to choose these hyperparameters leads to vanishing gradients: the iteration still converges, but considerably more slowly, since the equilibrium function and its Jacobian become extremely small away from the main lobe of $\bG$. We expect the severity of this sensitivity to depend strongly on the particular \ac{pde}/ODE under consideration; a systematic study of initialization strategies is beyond the scope of this proof--of--concept example and is left for future, application--specific work. The total runtime of the proposed Algorithm was $8.2$ seconds for a total of $200$ iterations. This gives a runtime of $0.04$ seconds per iteration. The number of iterations can be further decreased with better initialization.

This experiment showed the potential of using iterative \ac{pde} solvers that enforce \ac{rkhs}--based optimality conditions on the approximate solution. A separate study should be conducted to define a class of problems for which the proposed approach is applicable, and to develop numerically improved schemes. We note that \ac{rkhs}--based projections have already been considered for solving \acp{pde} in~\cite{stepaniants_learning_2023, xie_numerical_2008}, however, these works did not use the optimality conditions considered in this study.

\section{Conclusions}
\label{sec:conc}

In this work, we considered the problem of norm--optimal \ac{mor} in \acp{rkhs}. More precisely, we defined generalized model spaces and weak model spaces in \acp{rkhs} and introduced necessary conditions for optimal approximation between them. Our conditions generalize earlier results for \ac{mor} in $H_2$ Hardy--Hilbert spaces. The introduced conditions can only be applied if the kernel of the \ac{rkhs} is differentiable. In order to extend our methods to general \ac{rkhs}, we introduced the generalized Malmquist--Takenaka system. This orthogonal function system generalizes several earlier variants, such as the classical Malmquist--Takenaka system and Newton bases defined in \acp{rkhs}. Using the \ac{gmt} system, we introduced the \ac{gmt} transform, which provides a unitary map between any two separable \acp{rkhs}. Using these results, we developed several \ac{mor} methods and compared their performance on example applications.

We showed that in addition to solving \ac{mor} tasks in \acp{rkhs}, the proposed methods are also applicable to a variety of other problems. For example, in an academic example, we showed how fitting high--order polynomials to data can be interpreted as a \ac{mor} task. We showed how \ac{mor} can be applied to kernel methods such as support vector machines and provided several useful lemmas to clarify the relationship between \ac{fom} and \ac{rom}--\ac{svm} classifiers. We verified that we can apply the \ac{gmt}--based methodology to scalar--input scalar--output deep ReLU networks. Finally, we showed how a kernel basis discretization of certain \acp{pde} can be interpreted as norm--optimal \ac{mor} problem in Sobolev \acp{rkhs}. In this context, our algorithms can be interpreted as an adaptive mesh \ac{fem}--like approach that relies on a kernel slice basis instead of piecewise polynomials, with the additional benefit of the chosen "mesh" being determined according to necessary conditions of optimality.

Our results and experiments lay the foundation for further theoretical investigation as well as the thorough development of real--life applications. First, our optimality conditions and subsequent \ac{mor} algorithms should be extended to \acp{rkhs} defined over high--dimensional complex manifolds. This would allow the reduction of large machine learning and deep learning models. In addition, it would open the door toward the development of \ac{pde} applications. 

The asymptotic behavior and computational aspects of \ac{gmt} functions and the \ac{gmt} transformation should be further investigated. Generalized Blaschke--Sz\'asz conditions for a large class of \acp{rkhs} should be obtained to generate complete \ac{gmt} systems. A further open direction is to characterize the class of \acp{rkhs} for which weak and strong model spaces coincide under the \ac{gmt} transformation. Regarding computational aspects, kernel--sensitive variants of the Gram--Schmidt method should be developed. For the classical rational \ac{mt} functions, discrete orthogonal and fast \ac{fft}--based methods are already available for stable computation~\cite{dozsa_impulse_2026}. Possible generalizations of these methods to \acp{rkhs} should be investigated. 

Finally, the examples presented in this study were meant to showcase the versatility of the proposed \ac{mor} framework. The list of considered applications should be extended, for example, to obtain \acp{rom} of nonlinear dynamical systems. In addition, the presented examples should be extended into full--fledged applications in their own right.

\appendix
\section{Appendix}\label{sec:appendix}

\subsection{Equivalence of Theorems~\ref{thm:MLc} and~\ref{thm:orthoCondH2}}\label{app:equivMLc}

 \begin{theorem}[Equivalence of optimality conditions in $H_2(\C^+)$]
     Under the assumptions in \Cref{thm:MLc},~\eqref{eq:MLC} is equivalent to~\eqref{eq:orthCondH2}.
 \end{theorem} \begin{proof}
     From~\eqref{eq:transfun} and~\eqref{eq:kerh2c}, $\bG$ and $\bG_r$ consist of finite linear combinations of kernel slices. We use the reproducing property to obtain
    \begin{align*}
        \left\langle \bG - \bG_r, \bH_2 \right\rangle_{H_2(\C^{+})} &=  \left\langle \bG - \bG_r, \sum_{j=0}^{r-1} c_j \zeta_{\tlambda_j} \right\rangle_{H_2(\C^{+})} \\ &= 
        \sum_{j=0}^{r-1} \overline{c_j} \left\langle \bG, \zeta_{\tlambda_j} \right\rangle_{H_2(\C^{+})} - \sum_{j=0}^{r-1} \overline{c_j} \left\langle \bG_r, \zeta_{\tlambda_j} \right\rangle_{H_2(\C^{+})} \\ &= \sum_{j=0}^{r-1} \overline{c_j} \bG\big(\tlambda_j\big) - \sum_{j=0}^{r-1} \overline{c_j} \bG_r\big(\tlambda_j\big).
    \end{align*}
    Hence, we have
    \begin{equation}
    \label{eq:equivlagrange}
         \left\langle \bG - \bG_r, \bH_2 \right\rangle_{H_2(\C^{+})} = 0 \quad (\bH_2 \in \fG_\zeta^r)  \iff \bG\big(\tlambda_j\big) = \bG_r\big(\tlambda_j\big) \quad (j=0,\ldots,r-1).
    \end{equation}
    Notice that the derivative $\funder{\zeta_\lambda}{\olambda}$ satisfies
    \begin{equation*}
       \funder{\zeta_\lambda}{\olambda}(s) = - \frac{1}{\big(s + \overline{\lambda}\big)^2}.
    \end{equation*}
    Since $\bG_r$ is a \ac{rom} of order $r$ and by choosing $\bH_1 = \zeta_{\tlambda_j}$, $j \in \{0,\ldots,r-1\}$, we obtain 
    \begin{align*}
        \left\langle \bG - \bG_r, \bG_r \cdot \bH_1 \right\rangle_{H_2(\C^{+})} &=  \left\langle \bG - \bG_r, -\tres_j\funder{\zeta_{\tlambda_j}}{\overline{\tlambda}} + \bH_3 \right\rangle_{H_2(\C^{+})} \\ &=  \left\langle \bG - \bG_r, -\tres_j\funder{\zeta_{\tlambda_j}}{\overline{\tlambda}}\right\rangle_{H_2(\C^{+})} +  \left\langle \bG - \bG_r, \bH_3 \right\rangle_{H_2(\C^{+})}
    \end{align*}
    for some $\bH_3 \in \myspan\left\{\zeta_{\tilde{\lambda}_0},\ldots, \zeta_{\tilde{\lambda}_{r-1}}\right\}$ and $\tres_j \neq 0$. Here, we exploit the fact that, by using a partial fraction decomposition, the terms $\zeta_{\tlambda_i}\zeta_{\tlambda_j}\in \myspan\left\{\zeta_{\tilde{\lambda}_i}, \zeta_{\tilde{\lambda}_{j}}\right\}$ for $i,j = 0,\ldots,r-1$ with $i \neq j$.

    By the equivalence in \eqref{eq:equivlagrange}, it remains to show that for $j=0,\ldots,r-1$,
    \begin{equation*}
         \left\langle \bG - \bG_r, \funder{\zeta_{\tlambda_j}}{\overline{\tlambda}}\right\rangle_{H_2(\C^{+})} =0  \iff \bG'\big(\tlambda_j\big) - \bG_r'(\tlambda_j).
    \end{equation*}
    However, this condition is satisfied by \Cref{thm:derivKer}.

 \end{proof}

\subsection{Proof of \Cref{thm:specialCases}} \label{app:specialCases}

\begin{proof}
        \Cref{thm:orthocond} and \Cref{thm:MLcRKHS} necessarily apply to $\bG_r$, because it is a norm--optimal \ac{rom} in an \ac{rkhs}. First, we show that the optimality conditions in \Cref{thm:MLcRKHS} take the form~\eqref{eq:MLC} under the assumptions of the statement. Exploiting that $\tlambda_j$ are pairwise distinct points in $\C^{+}$ for $j=0,\ldots,r-1$, \eqref{eq:MLcRKHS} is given by 
    \begin{equation*}
        \bG\big(\tlambda_j\big) = \bG_r\big(\tlambda_j\big) \quad \text{and} \quad \bG'\big(\tlambda_j\big) = \bG_r'\big(\tlambda_j\big)
    \end{equation*}
    for $j=0,\ldots,r-1$. This condition matches the statement of \Cref{thm:MLc} exactly. In~\cite{gugercin_mathcalh_2_2008} (see also \Cref{app:equivMLc}), the equivalence between \Cref{thm:MLc} and \Cref{thm:orthoCondH2} is proven. On the other hand, by \Cref{thm:equivMlcOrtho}, \Cref{thm:orthocond} and \Cref{thm:MLcRKHS} are also equivalent. This implies that in the current setting \Cref{thm:orthoCondH2} is a special case of \Cref{thm:orthocond}.

\end{proof}

\subsection{Proof of \Cref{thm:errInterp}}\label{app:errInterp}

\begin{proof}
     The statement is a direct consequence of \Cref{thm:reprInnerProd} and the Cauchy--Schwarz inequality. Indeed, for any $x \in \Omega$,
    \begin{align*}
        \left|\bG(x) - \bG_r(x) \right|^2 &= \left| \left\langle \bG - \bG_r, \xi_x \right\rangle_{\mathbb{P}_n(\Omega)} \right|^2 \\ &\leq \left\| \bG - \bG_r \right\|_{\mathbb{P}_n(\Omega)}^2 \left\| \xi_{x} \right\|_{\mathbb{P}_n(\Omega)}^2 \\ &=  \left\| \bG - \bG_r \right\|_{\mathbb{P}_n(\Omega)}^2 \xi(x, x) \\ &\leq  \left\| \bG - \bG_r \right\|_{\mathbb{P}_n(\Omega)}^2 \sum_{k=0}^{n-1}  x^{2k}  \leq n \left\| \bG - \bG_r \right\|_{\mathbb{P}_n(\Omega)}^2.
    \end{align*}
\end{proof}

\subsection{Proof of \Cref{thm:svmRomMarginAgreement}}\label{app:svmMargin}

\begin{proof}
Let $\bE(z)\coloneqq\bD(z)-\bD_r(z)$.
Then, $\bD(z)=\bD_r(z)+\bE(z)$.

Assume first that $\bD_r(z)>0$. If $
    |\bD_r(z)|>|\bE(z)|$,
then $|\bD_r(z)|=\bD_r(z)$, and therefore, $
    \bE(z)>-|\bE(z)|>-\bD_r(z)$.
Hence, $
    \bD(z)=\bD_r(z)+\bE(z)>0$.
Thus, $\bD(z)$ and $\bD_r(z)$ have the same sign.

Assume now that $\bD_r(z)<0$. Then $|\bD_r(z)|=-\bD_r(z)$. Since $
    |\bE(z)|<|\bD_r(z)|=-\bD_r(z)$,
we have $\bE(z)<|\bE(z)|<-\bD_r(z)$.
Therefore, $\bD(z)=\bD_r(z)+\bE(z)<0$.
Again, $\bD(z)$ and $\bD_r(z)$ have the same sign.

Thus, in both cases,
\[
    |\bD_r(z)|>|\bD(z)-\bD_r(z)|
    \ \Longrightarrow \
   \sgn (\bD(z))=\sgn (\bD_r(z)).
\]
\end{proof}

\subsection{Proof of \Cref{thm:svmRomMarginAgreementV2}}\label{app:svmMarginV2}

\begin{proof}

Recall that $\bD=\bG+b$ and $\bD_r=\bG_r+b$ use the same bias. Therefore,\begin{equation*}
\bD-\bD_r=\bG-\bG_r.
\end{equation*} 
Since $\bG_r = \cP_{V_r^{\tilde{\blambda}}}[\bG]$, the Hilbert projection theorem implies
\begin{equation}
    \label{eq:svmOrth}
    \left\langle \bG - \bG_r, \bH \right\rangle_{\origspace} = 0 \quad \big(\bH \in \Vn{r}{\tilde{\blambda}}\big).
\end{equation}

Let $\bE \coloneqq \bD - \bD_r=\bG-\bG_r \in \origspace$, and notice that by the reproducing property $\bE(z) = \left\langle \bE, \xi_z \right\rangle_{\origspace}$ holds for all $z \in \Omega$.
The kernel slice $\xi_{z}$ can be decomposed as
\begin{equation*}
    \xi_z = \cP_{V_r^{\tilde{\blambda}}} \left[ \xi_z \right] + \cP_{\left(V_r^{\tilde{\blambda}}\right)^{\perp}} \left[ \xi_z \right].
\end{equation*}
Hence, we find that
\begin{equation*}
    \bE(z) = \bD(z) - \bD_r(z)=\bG(z) - \bG_r(z) = \left\langle \bE, \cP_{V_r^{\tilde{\blambda}}} \left[ \xi_z \right] \right\rangle_{\origspace} + \left\langle \bE, \cP_{\left(V_r^{\tilde{\blambda}}\right)^{\perp}} \left[ \xi_z \right] \right\rangle_{\origspace}.
\end{equation*}
By~\eqref{eq:svmOrth}, the first term vanishes and thus,
\begin{equation*}
    \bE(z) = \left\langle \bE, \cP_{\left(V_r^{\tilde{\blambda}}\right)^{\perp}} \left[ \xi_z \right] \right\rangle_{\origspace}.
\end{equation*}
Applying the Cauchy--Schwarz inequality to this gives
\begin{equation}
    \label{eq:csbound}
    \left| \bE(z) \right| \leq \rkhsError{\xi}{r}(\bD_r) \left\|\cP_{\left(V_r^{\tilde{\blambda}}\right)^\perp} \left[ \xi_{z} \right] \right\|_{\origspace}.
\end{equation}
Suppose $\bD_r(z) > \rkhsError{\xi}{r}(\bD_r) \left\|\cP_{\left(V_r^{\tilde{\blambda}}\right)^\perp} \left[ \xi_{z} \right] \right\|_{\origspace} > 0$. Then by~\eqref{eq:csbound},
\begin{equation*}
   \bD(z) = \bD_r(z) + \bE(z) >  \bD_r(z)  - \rkhsError{\xi}{r}(\bD_r) \left\|\cP_{\left(V_r^{\tilde{\blambda}}\right)^\perp} \left[ \xi_{z} \right] \right\|_{\origspace} > 0,
\end{equation*}
therefore $\bD(z) > 0$ and $\bD_r(z) > 0$. Similarly, assuming 
\begin{equation*}
\bD_r(z) < - \rkhsError{\xi}{r}(\bD_r) \left\|\cP_{\left(V_r^{\tilde{\blambda}}\right)^\perp} \left[ \xi_{z} \right] \right\|_{\origspace} < 0    
\end{equation*}
gives
\begin{equation*}
    \bD(z) = \bD_r(z) + \bE(z) < \bD_r(z) + \rkhsError{\xi}{r}(\bD_r) \left\|\cP_{\left(V_r^{\tilde{\blambda}}\right)^\perp} \left[ \xi_{z} \right] \right\|_{\origspace} < 0,
\end{equation*}
hence $\bD(z) < 0$ and $\bD_r(z) < 0$. 
\end{proof}

\subsection{Proof of \Cref{thm:dwr}}\label{app:dwr}

\begin{proof}
  For simplicity of notation, define the pointwise error
  \begin{equation*}
      \bE(x) \coloneqq \bG(x) - \bG_r(x) \quad (x \in \Omega)
  \end{equation*}
  and fix $\lambda \in \Omega$. Define the functionals $\origspace \to \origspace$ by
  \begin{equation*}
      \cK_i(\bH) \coloneqq \bH^{(i)}(\lambda) \quad (i=1,2).
  \end{equation*}
  Consider the problems of finding $\phi_i \in \origspace$ such that
  \begin{equation}
    \label{eq:dualWeakProb}
      a \left( \phi_i,\bH \right) = \cK_i \left( \bH \right) \quad \left( \bH \in \origspace, i=1,2  \right).
  \end{equation}
  Since $\cK_i$ are bounded linear functionals on $\origspace$, the Riesz representation theorem guarantees the existence of unique $\phi_1,\phi_2 \in \origspace$ that satisfies~\eqref{eq:dualWeakProb} for all $\bH \in \origspace$. Exploiting that $\origspace$ is an \ac{rkhs} with kernel $\xi$ and $a(\cdot, \cdot) = \langle \cdot, \cdot \rangle_{\origspace}$, we find
  \begin{equation*}
      a \left( \xi_{\lambda},\bH \right) = a \left( \bH, \xi_{\lambda} \right) = \bH(\lambda) \quad \left( \bH \in \origspace \right).
  \end{equation*}
  Differentiating both sides and exploiting \Cref{thm:derivKer} gives
  \begin{equation*}
      \frac{\partial}{\partial \lambda} a \left( \bH, \xi_{\lambda} \right) = \frac{\partial}{\partial \lambda} \left\langle \bH, \xi_{\lambda} \right\rangle_{\origspace} = \left\langle \bH, \funder{\xi_{\lambda}}{\olambda} \right\rangle_{\origspace} = a \left( \bH, \funder{\xi_{\lambda}}{\olambda} \right) = \bH'(\lambda).
  \end{equation*}
  Hence, we find that the solution $\phi_1$ in~\eqref{eq:dualWeakProb} is
  \begin{equation}
      \label{eq:soli1}
      \phi_1(x) = \funder{\xi_{\lambda}}{\olambda}(x) \stackrel{\eqref{eq:sobker}}{=} \begin{cases}
          \frac{1}{2} x^2, & \text{if } x \leq \lambda, \\
          \lambda x - \frac{1}{2} \lambda^2, & \text{if } x > \lambda.
      \end{cases}
  \end{equation}
  By repeating the above steps, we further get
  \begin{equation}
      \label{eq:soli2}
      \phi_2(x) = \funder[2]{\xi_{\lambda}}{\olambda}(x) = \begin{cases}
          0, & \text{if } x \leq \lambda, \\
          x - \lambda, & \text{if } x > \lambda.
      \end{cases}
  \end{equation}
  In particular,
  \begin{equation}
      \label{eq:phixi}
      a(\bG_r, \phi_i) = \bG_r^{(i)}(\lambda) \quad (i=1,2).
  \end{equation}
  Making the substitution $\bH = \bE$ in~\eqref{eq:dualWeakProb} gives
  \begin{equation}
    \label{eq:Kie}
      \cK_i(\bE) = \bG^{(i)}(\lambda) - \bG_r^{(i)}(\lambda) = a(\bE, \phi_i) = a(\bG - \bG_r, \phi_i) = a(\bG, \phi_i) - a(\bG_r, \phi_i). 
  \end{equation}
  Since $\bG$ is a solution of~\eqref{eq:weakpde}, we obtain
  \begin{equation}
    \label{eq:removeG}
      a(\bG, \phi_i) = \int_0^1 \bF(x) \phi_i(x) \,\du x,
  \end{equation}
  where $\bF$ is the right hand side of the strong problem~\eqref{eq:pde}.  Facilitating~\eqref{eq:phixi} and \eqref{eq:removeG} in~\eqref{eq:Kie} gives
  \begin{equation*}
      \bG^{(i)}(\lambda) - \bG_r^{(i)}(\lambda) = \int_0^1 \bF(x) \phi_i(x)\, \du x - \bG_r^{(i)}(\lambda) \quad (i=1,2).
  \end{equation*}
  Considering the definition of the equilibrium function~\eqref{eq:EquibFun}, choosing $\lambda \coloneqq \tlambda_j$, $j=0,\ldots,r-1$, and applying~\eqref{eq:soli1} and~\eqref{eq:soli2} gives the statement.
\end{proof}

\begin{ack}
Most of this work was carried out while TD was with UniDistance Suisse, Faculty of Mathematics and Computer Science.
\end{ack}

\begin{funding}
The research of TD has received funding from the Swiss Government Excellence Scholarship No. 2025.0057. The research of AA and MV was funded in part by the Swiss National Science Foundation (SNSF) grant No. 224943.
\end{funding}

\begin{ack}[CRediT author statement]
\textbf{Tam\'as D\'ozsa:} Conceptualization,
Methodology, Formal Analysis, Investigation, Software, Visualization, Writing --- Original
Draft; \textbf{Andrea Angino:} Methodology, Software,  Visualization,
Writing --- Original Draft; \textbf{Rolf Krause:} Writing --- Review \& Editing; \textbf{Matthias Voigt:} Investigation, Methodology, Visualization, Writing --- Review \&
Editing, Supervision, Project Administration, Funding Acquisition.
\end{ack}

\begin{ack}[Code availability statement]
The code and data for reproducing the numerical experiments presented in this paper are available at Zenodo under the DOI
\begin{center}
\href{https://doi.org/10.5281/zenodo.23494458}{10.5281/zenodo.23494458}.
\end{center}
The code is released under the MIT license.
\end{ack}

\section*{List of acronyms}
\printacronyms[heading=none]


\bibliographystyle{emss}
\bibliography{refs_focm}

@article{dozsa_brake_2024,
	title = {Brake disc deformation detection using intuitive feature extraction and machine learning},
	volume = {12},
	number = {4},
	journal = {Machines},
	author = {Dózsa, Tamás and Őri, Péter and Szabari, Mátyás and Simonyi, Ernő and Soumelidis, Alexandros and Lakatos, István},
	year = {2024},
	pages = {article no. 214},
}

@article{dozsa_detection_2026,
	title = {Detection of brake disc deformation with adaptive wavelet neural networks},
	volume = {75},
	journal = {IEEE Trans. Instrum. Meas.},
	author = {Dózsa, Tamás and Őri, Péter and Ungvári, Gergő and Szabó, Zoltán and Soumelidis, Alexandros and Lakatos, István and Kovács, Péter},
	year = {2026},
	pages = {article no. 2510713},
}

@article{pap_hyperbolic_2011,
	title = {Hyperbolic wavelets and multiresolution in {$H^2(\mathbb{T})$}},
	volume = {17},
	number = {5},
	journal = {J. Fourier Anal. Appl.},
	author = {Pap, Margit},
	year = {2011},
	pages = {755--776},
    mrnumber = {2838106},
    zblnumber = {1228.42042},
}

@article{schipp_rational_2003,
	title = {Rational bases generated by {B}laschke product systems},
	volume = {36},
	number = {16},
	journal = {IFAC Proc. Vol.},
	author = {Schipp, F. and Bokor, J.},
	year = {2003},
	pages = {1309--1314},
    note      = {13th IFAC Symposium on System Identification (SYSID 2003), Rotterdam, The Netherlands}
}

@article{berrut_barycentric_2004,
	title = {Barycentric {L}agrange interpolation},
	volume = {46},
	number = {3},
	journal = {SIAM Rev.},
	author = {Berrut, Jean-Paul and Trefethen, Lloyd N.},
	year = {2004},
	pages = {501--517},
    mrnumber = {2115059},
    zblnumber = {1061.65006},
}

@article{wei_random_2024,
  author    = {Wei, Xizixiang and Shen, Cong and Yang, Jing and Poor, H. Vincent},
  title     = {Random Orthogonalization for Federated Learning in Massive {MIMO} Systems},
  journal   = {IEEE Trans. Wireless Commun.},
  volume    = {23},
  number    = {3},
  pages     = {2469--2485},
  year      = {2024},
}

@article{jang_randomized_2025,
	title = {Randomized orthogonalization process with reorthogonalization},
	volume = {32},
	number = {4},
	journal = {Numer. Linear Algebra Appl.},
	author = {Jang, Yongseok and Grigori, Laura},
	year = {2025},
	pages = {article no. e70029},
    zblnumber = {8127837},
    mrnumber = {4939755},
}

@article{terao_lu-cholesky_2020,
	title = {{LU}-{C}holesky {QR} algorithms for thin {QR} decomposition},
	volume = {92},
	journal = {Parallel Comput.},
	author = {Terao, Takeshi and Ozaki, Katsuhisa and Ogita, Takeshi},
	year = {2020},
	pages = {article no. 102571},
    mrnumber = {4057043},
}

@article{de_figueiredo_generalized_1983,
	title = {A generalized {F}ock space framework for nonlinear system and signal analysis},
	volume = {30},
	number = {9},
	journal = {IEEE Trans. Circuits Systems},
	author = {De Figueiredo, R.},
	year = {1983},
	pages = {637--647},
    zblnumber = {0539.93041},
    mrnumber = {0718397},
}

@book{zhu_analysis_2012,
  author    = {Zhu, Kehe},
  title     = {{Analysis on Fock Spaces}},
  series    = {Grad. Texts in Math.},
  volume    = {263},
  publisher = {Springer},
  address   = {New York, NY},
  year      = {2012},
  zblnumber = {1262.30003},
  mrnumber  = {2934601},
}

@article{horowitz_zeros_1974,
	title = {Zeros of functions in the {B}ergman spaces},
	volume = {41},
	number = {4},
	journal = {Duke Math. J.},
	author = {Horowitz, Charles},
	year = {1974},
	pages = {693--710},
	mrnumber = {0357747},
	zblnumber = {0293.30035},
}

@book{petrushev_rational_1988,
	address = {Cambridge},
	series = {Encyclopedia Math. Appl.},
	title = {{Rational Approximation of Real Functions}},
	publisher = {Cambridge University Press},
    volume = {28},
	author = {Petrushev, P. P. and Popov, Vasil Atanasov},
	year = {1987},
    zblnumber = {0644.41010},
    mrnumber = {0940242},
}

@article{becker_optimal_2001,
	title = {An optimal control approach to a posteriori error estimation in finite element methods},
	volume = {10},
	journal = {Acta Numer.},
	author = {Becker, Roland and Rannacher, Rolf},
	year = {2001},
	pages = {1--102},
    mrnumber  = {2009692},
    zblnumber = {1105.65349},
}

@article{xie_numerical_2008,
  author    = {Xie, S.-S. and Heo, S. and Kim, S. and Woo, G. and Yi, S.},
  title     = {Numerical Solution of One--Dimensional {B}urgers' Equation Using Reproducing Kernel Function},
  journal   = {J. Comput. Appl. Math.},
  volume    = {214},
  number    = {2},
  pages     = {417--434},
  year      = {2008},
  zblnumber = {1140.65069},
  mrnumber  = {2398343},
}

@article{stepaniants_learning_2023,
	title = {Learning partial differential equations in reproducing kernel {H}ilbert spaces},
	volume = {24},
	journal = {J. Mach. Learn. Res.},
	author = {Stepaniants, George},
	year = {2023},
	pages = {article no. 86},
    zblnumber = {1576.68175},
    mrnumber = {4582508},
}

@book{stakgold_greens_2011,
  author    = {Stakgold, Ivar and Holst, Michael J.},
  title     = {{Green's Functions and Boundary Value Problems}},
  edition   = {3rd},
  series    = {Pure Appl. Math. (Hoboken)},
  volume    = {99},
  publisher = {John Wiley \& Sons},
  address   = {Hoboken, NJ},
  year      = {2011},
  zblnumber = {1221.35001},
  mrnumber  = {2789179}
}

@article{vincent_kernel_2002,
	title = {Kernel matching pursuit},
	volume = {48},
	number = {1-3},
	journal = {Mach. Learn.},
	author = {Vincent, Pascal and Bengio, Yoshua},
	year = {2002},
	pages = {165--187},
    zblnumber = {0998.68120},
}

@article{barrault_empirical_2004,
	title = {An `empirical interpolation' method: application to efficient reduced--basis discretization of partial differential equations},
	volume = {339},
	number = {9},
	journal = {C. R. Math. Acad. Sci. Paris},
	author = {Barrault, Maxime and Maday, Yvon and Nguyen, Ngoc Cuong and Patera, Anthony T.},
	year = {2004},
	pages = {667--672},
    mrnumber = {2103208},
    zblnumber = {1061.65118},
}

@inproceedings{hanin_complexity_2019,
	title = {Complexity of linear regions in deep networks},
	booktitle = {{Proceedings of the 36th International Conference on Machine Learning}},
	author = {Hanin, Boris and Rolnick, David},
	year = {2019},
	pages = {2596--2604},
    editor    = {Chaudhuri, Kamalika and Salakhutdinov, Ruslan},
    series    = {Proc. Mach. Learn. Res.},
    volume    = {97},
    pages     = {2596--2604},
    address   = {Long Beach, CA},
}

@inproceedings{montufar_number_2014,
	  author    = {Mont{\'u}far, Guido and Pascanu, Razvan and Cho, Kyunghyun and Bengio, Yoshua},
      title     = {On the Number of Linear Regions of Deep Neural Networks},
      booktitle = {{Advances in Neural Information Processing Systems 27}},
      address   = {Montreal, QC, Canada},
      editor    = {Ghahramani, Zoubin and Welling, Max and Cortes, Corinna and Lawrence, Neil D. and Weinberger, Kilian Q.},
      pages     = {2924--2932},
      year      = {2014},
      url       = {https://papers.nips.cc/paper/2014/hash/109d2dd3608f669ca17920c511c2a41e-Abstract.html},
}

@inproceedings{vaswani_attention_2017,
	 author    = {Vaswani, Ashish and Shazeer, Noam and Parmar, Niki and Uszkoreit, Jakob and Jones, Llion and Gomez, Aidan N. and Kaiser, {\L}ukasz and Polosukhin, Illia},
  title     = {Attention Is All You Need},
  booktitle = {{Advances in Neural Information Processing Systems~30}},
  editor    = {Guyon, Isabelle and von Luxburg, Ulrike and Bengio, Samy and Wallach, Hanna and Fergus, Rob and Vishwanathan, S. V. N. and Garnett, Roman},
  pages     = {5998--6008},
  address   = {Long Beach, CA, USA},
  year      = {2017},
}

@techreport{kingma_adam_2014,
	title = {Adam: a method for stochastic optimization},
	author = {Kingma, Diederik P. and Ba, Jimmy},
	year = {2014},
    type = {{arXiv preprint}},
	eprint = {1412.6980},
    note = {Presented at the 3rd International Conference on Learning Representations (ICLR), San Diego, CA, USA, 2015}
}

@article{blalock_what_2020,
	title = {What is the state of neural network pruning?},
	volume = {2},
	journal = {Proc. Mach. Learn. Syst.},
	author = {Blalock, Davis and Gonzalez Ortiz, Jose Javier and Frankle, Jonathan and Guttag, John},
    editor = {I. Dhillon and D. Papailiopoulos and V. Sze},
	year = {2020},
	pages = {129--146},
}

@article{hornik_approximation_1991,
	title = {Approximation capabilities of multilayer feedforward networks},
	volume = {4},
	number = {2},
	journal = {Neural Netw.},
	author = {Hornik, Kurt},
	year = {1991},
	pages = {251--257},
}

@article{kovacs_vpnet_2022,
	title = {{VPNet}: variable projection networks},
	volume = {32},
	number = {1},
	journal = {Int. J. Neural Syst.},
	author = {Kov\'acs, P\'eter and Bogn\'ar, Gerg\H{o} and Huber, Christian and Huemer, Mario},
	year = {2022},
	pages = {article no. 2150054},
}

@book{atkinson_introduction_2008,
	edition = {2nd},
	title = {{An Introduction to Numerical Analysis}},
	publisher = {John Wiley \& Sons},
    address = {New York, NY},
	author = {Atkinson, K.~E.},
	year = {1989},
    mrnumber = {1007135},
    zblnumber = {0402.65001},
}

@book{noauthor_theory_1970,
	title = {{Theory of {$H^p$} Spaces}},
	publisher = {Academic Press},
	author = {Duren, Peter L.},
    series = {Pure Appl. Math.},
    volume = {38},
    address = {New York, NY},
	year = {1970},
    mrnumber  = {0268655},
    zblnumber = {0215.20203},
}

@book{duren_bergman_2024,
	title = {{Bergman Spaces}},
	publisher = {American Mathematical Society},
	author = {Duren, Peter and Schuster, Alexander},
    series = {Math. Surveys Monogr.},
    volume = {100}, 
    address = {Providence, RI},
	year = {2004},
    mrnumber  = {2033762},
    zblnumber = {1059.30001},
}

@article{pap_multiresolution_2011,
	title = {Multiresolution in the {B}ergman space},
	author = {Pap, Margit},
    journal   = {Ann. Univ. Sci. Budapest. Sect. Comput.},
    volume    = {39},
    pages     = {333--353},
    year      = {2013},
    zblnumber = {1289.30234},
    mrumber = {3045618},
}

@article{wahlberg_system_1991,
	title = {System identification using {L}aguerre models},
	volume = {36},
	number = {5},
	journal = {IEEE Trans. Automat. Control},
	author = {Wahlberg, B.},
	year = {1991},
	pages = {551--562},
    zblnumber = {0738.93078},
    mrnumber  = {1101710}
}

@article{kovacs_generalized_2020,
	title = {Generalized rational variable projection with application in {ECG} compression},
	volume = {68},
	journal = {IEEE Trans. Signal Process.},
	author = {Kov\'acs, P\'eter and Fridli, S\'andor and Schipp, Ferenc},
	year = {2020},
	pages = {478--492},
    zblnumber = {1543.94230}
}

@book{garcia_finite_2018,
	address = {Cham},
	title = {{Finite {B}laschke Products and Their Connections}},
	publisher = {Springer},
	author = {Garcia, Stephan Ramon and Mashreghi, Javad and Ross, William T.},
	year = {2018},
    zblnumber = {1398.30002},
    mrnumber = {3793610},
}

@article{nagy-csiha_construction_2023,
	title = {Construction of rational interpolations using {M}almquist-{T}akenaka systems},
	volume = {6},
	number = {1},
	journal = {Constr. Math. Anal.},
	author = {Nagy-Csiha, Zsuzsanna and Pap, Margit and Weisz, Ferenc},
	year = {2023},
	pages = {55--76},
    zblnumber = {1524.42052},
    mrnumber = {4566456},
}

@article{szabo_identification_1999,
	title = {Identification of rational approximate models in {$H^\infty$} using generalized orthonormal basis},
	volume = {44},
	number = {1},
	journal = {IEEE Trans. Automat. Control},
	author = {Szab\'o, Z. and Bokor, J. and Schipp, F.},
	year = {1999},
	pages = {153--158},
    zblnumber = {1056.93528},
    mrnumber = {1666925}, 
}

@inproceedings{gozse_realizing_2015,
	address = {Torremolinos, Spain},
	title = {Realizing system poles identification on the unit disc based on the {F}ourier transform of {L}aguerre coefficients},
	booktitle = {{2015 23rd Mediterranean Conference on Control and Automation ({MED})}},
	author = {G\H{o}zse, Istvan and Soumelidis, Alexandros},
	year = {2015},
	pages = {821--826},
}

@article{dozsa_pole_2023,
	title = {Pole identification using discrete {L}aguerre expansion and variable projection},
	volume = {56},
	number = {2},
	journal = {IFAC-PapersOnLine},
	author = {Dózsa, Tamás and Szabari, Mátyás and Soumelidis, Alexandros and Kovács, Péter},
	year = {2023},
	pages = {3994--3999},
}

@book{heuberger_modelling_2005,
	address = {London},
	title = {{Modelling and Identification with Rational Orthogonal Basis Functions}},
	publisher = {Springer},
	editor = {Heuberger, Peter S. C. and Van Den Hof, Paul M. J. and Wahlberg, Bo},
	year = {2005},
}

@article{dozsa_bernoullis_2024,
	title = {On {B}ernoulli's method},
	volume = {62},
	number = {3},
	journal = {SIAM J. Numer. Anal.},
	author = {Dózsa, Tamás and Schipp, Ferenc and Soumelidis, Alexandros},
	year = {2024},
	pages = {1259--1277},
    zblnumber = {1545.65077},
    mrnumber = {4750273},
}

@techreport{dozsa_impulse_2026,
	title = {Impulse response estimation via {L}aguerre-{F}ourier expansion},
	author = {Dózsa, Tamás and Pelling, Art J. R. and Voigt, Matthias},
	year = {2026},
    type = {{arXiv preprint}},
	eprint = {2608.14769},
}

@techreport{cdcAngino,
	title = {{$H_2$}-optimal model order reduction using hyperbolic geometry},
	author = {Angino, Andrea and Dózsa, Tamás and Voigt, Matthias},
	year = {2026},
    type = {{arXiv preprint}},
    eprint = {2610.16249},
}

@article{flagg_convergence_2012,
	title = {Convergence of the iterative rational {K}rylov algorithm},
	volume = {61},
	number = {6},
	journal = {Systems Control Lett.},
	author = {Flagg, Garret and Beattie, Christopher and Gugercin, Serkan},
	year = {2012},
	pages = {688--691},
    zblnumber = {1401.93051},
    mrnumber = {2924212},
}

@book{conway_course_2007,
	address = {New York, NY},
	series = {Grad. Texts in Math.},
	title = {{A Course in Functional Analysis}},
	volume = {96},
	publisher = {Springer},
    edition = {2nd},
	author = {Conway, John B.},
	year = {2007},
    mrnumber  = {1070713},
    zblnumber = {0706.46003},
}

@article{zhou_derivative_2008,
	title = {Derivative reproducing properties for kernel methods in learning theory},
	volume = {220},
	number = {1--2},
    journal = {J. Comput. Appl. Math.},
	author = {Zhou, Ding-Xuan},
	year = {2008},
	pages = {456--463},
    zblnumber = {1152.68049},
    mrnumber = {2444183}, 
}

@article{mercer_xvi_1909,
	title = {Functions of positive and negative type, and their connection with the theory of integral equations},
	volume = {209},
	journal = {Philos. Trans. Roy. Soc. London Ser. A},
	author = {Mercer, James},
	year = {1909},
	pages = {415--446},
    zblnumber = {40.0408.02},
}

@techreport{dozsa_adaptive_2026,
	title = {Adaptive kernel methods},
	author = {Dózsa, Tamás and Angino, Andrea and Szabó, Zoltán and Bokor, József and Voigt, Matthias},
	year = {2026},
    type = {{arXiv preprint}},
	eprint = {2601.21707},
}

@article{jorgensen_new_2025,
	title = {New duality in choices of feature spaces via kernel analysis},
	volume = {23},
	number = {1},
	journal = {Sampl. Theory Signal Process. Data Anal.},
	author = {Jorgensen, Palle E. T. and Tian, James},
	year = {2025},
	pages = {article no. 5},
    mrnumber = {4867336},
    zblnumber = {1575.46021},
}

@article{mlinaric_irka_2025,
	title = {{IRKA} is a {R}iemannian gradient descent method},
	volume = {70},
	number = {5},
	journal = {IEEE Trans. Automat. Control},
	author = {Mlinari\'c, Petar and Beattie, Christopher A. and Drma\v{c}, Zlatko and Gugercin, Serkan},
	year = {2025},
	pages = {2979--2991},
    zblnumber = {8062336},
    mrnumber = {4898622},
}

@book{paulsen_introduction_2016,
	title = {{An Introduction to the Theory of Reproducing Kernel Hilbert Spaces}},
	author = {Paulsen, Vern I. and Raghupathi, Mrinal},
	year = {2016},
    series    = {Cambridge Stud. Adv. Math.},
    volume    = {152},
    publisher = {Cambridge University Press},
    address   = {Cambridge},
    zblnumber = {1364.46004},
    mrnumber = {3526117},
}

@article{dozsa_variable_2024,
	title = {Variable projection support vector machines and some applications using adaptive {H}ermite expansions},
	volume = {34},
	number = {1},
	journal = {Int. J. Neural Syst.},
	author = {Dózsa, Tamás and Deuschle, Federico and Cornelis, Bram and Kovács, Péter},
	year = {2024},
	pages = {article no. 2450004},
}

@inproceedings{vapnik_complete_2020,
  author    = {Vapnik, Vladimir and Izmailov, Rauf},
  title     = {Complete Statistical Theory of Learning: Learning Using Statistical Invariants},
  booktitle = {{Proceedings of the Ninth Symposium on Conformal and Probabilistic Prediction and Applications}},
  editor    = {Gammerman, Alexander and Vovk, Vladimir and Luo, Zhiyuan and Smirnov, Evgueni and Cherubin, Giovanni},
  series    = {Proc. Mach. Learn. Res.},
  volume    = {128},
  pages     = {4--40},
  year      = {2020},
}

@article{hearst_support_1998,
	title = {Support vector machines},
	volume = {13},
	number = {4},
	journal = {IEEE Intell. Syst. Appl.},
	author = {Hearst, M. A. and Dumais, S. T. and Osuna, E. and Platt, J. and Sch\"olkopf, B.},
	year = {1998},
	pages = {18--28},
}

@article{aronszajn_theory_1950,
	title = {Theory of reproducing kernels},
	volume = {68},
	number = {3},
	journal = {Trans. Amer. Math. Soc.},
	author = {Aronszajn, N.},
	year = {1950},
	pages = {337--404},
	mrnumber = {0051437},
    zblnumber = {0061.26204},
    url = {https://pubs.ams.org/TRAN/1950-068-03/S0002-9947-1950-0051437-7},
}

@inproceedings{shilton_reproducing_2025,
	address = {San Diego, CA, USA and Mexico City, Mexico},
	title = {Reproducing kernel {B}anach space models for neural networks with application to {R}ademacher complexity analysis},
	booktitle = {{Advances in Neural Information Processing Systems 38}},
	publisher = {Neural Information Processing Systems Foundation},
	author = {Shilton, Alistair and Gupta, Sunil and Rana, Santu and Venkatesh, Svetha},
	year = {2025},
	pages = {77765--77809},
}

@article{daubechies_nonlinear_2019,
	title = {Nonlinear approximation and (deep) {ReLU} networks},
	author = {Daubechies, I. and DeVore, R. and Foucart, S. and Hanin, B. and Petrova, G.},
    journal = {Constr. Approx.},
    volume = {55},
    number = {1},
    pages = {127--172}, 
	year = {2022},
    zblnumber = {1501.41003},
    mrnumber = {4376561},
}

@article{plonka_spline_2023,
	title = {Spline representation and redundancies of one--dimensional {ReLU} neural network models},
	volume = {21},
	number = {1},
	journal = {Anal. Appl. (Singap.)},
	author = {Plonka, Gerlind and Riebe, Yannick and Kolomoitsev, Yurii},
	year = {2023},
	pages = {127--163},
    zblnumber = {1507.41005},
    mrnumber = {4539213},
}

@article{muller_newton_2009,
	title = {A {N}ewton basis for kernel spaces},
	volume = {161},
	number = {2},
	journal = {J. Approx. Theory},
	author = {M\"uller, Stefan and Schaback, Robert},
	year = {2009},
	pages = {645--655},
    mrnumber = {2563073},
    zblnumber = {1185.41004},
}

@techreport{qian_kernel_2026,
	title = {Kernel {T}akenaka-{M}almquist systems and adaptive approximation in boundary weighted {H}ardy spaces},
	author = {Qian, Tao},
	year = {2026},
    type = {{arXiv preprint}},
	eprint = {2607.24473},
}

@article{felder_analogues_2022,
	title = {Analogues of finite {B}laschke products as inner functions},
	volume = {54},
	number = {4},
	journal = {Bull. Lond. Math. Soc.},
	author = {Felder, Christopher and Le, Trieu},
	year = {2022},
	pages = {1197--1219},
    zblnumber = {1529.30058},
    mrnumber = {4468315},
}

@article{mi_fast_2012,
	title = {A fast adaptive model reduction method based on {T}akenaka-{M}almquist systems},
	volume = {61},
	number = {1},
	journal = {Systems Control Lett.},
	author = {Mi, Wen and Qian, Tao and Wan, Feng},
	year = {2012},
	pages = {223--230},
    mrnumber = {2878709},
    zblnumber = {1256.93029},
}

@article{meier_approximation_1967,
	title = {Approximation of linear constant systems},
	volume = {12},
	number = {5},
	journal = {IEEE Trans. Automat. Control},
	author = {Meier, L. and Luenberger, D.},
	year = {1967},
	pages = {585--588},
    
}

@article{benner_interpolation-based_2012,
	title = {Interpolation--based {$\mathcal{H}_2$}-model reduction of bilinear control systems},
	volume = {33},
	number = {3},
	journal = {SIAM J. Matrix Anal. Appl.},
	author = {Benner, Peter and Breiten, Tobias},
	year = {2012},
	pages = {859--885},
    mrnumber = {3023455},
    zblnumber = {1256.93027},
}

@article{borghi_data-driven_2025,
	title = {Data--driven optimal approximation on {H}ardy spaces in simply connected domains},
	volume = {51},
	number = {6},
	journal = {Adv. Comput. Math.},
	author = {Borghi, Alessandro and Breiten, Tobias},
	year = {2025},
	pages = {article no. 59},
    zblnumber = {8146323},
    mrnumber = {5000002},
}

@article{soumelidis_hyperbolic_2017,
	title = {Hyperbolic geometrical approach to model reduction},
	volume = {50},
	number = {1},
	journal = {IFAC-PapersOnLine},
	author = {Soumelidis, Alexandros and Bokor, József and Schipp, Ferenc and Szabó, Zoltán},
	year = {2017},
	pages = {12905--12910},
}

@article{schroder_balanced_2023,
	title = {Balanced truncation model reduction with a priori error bounds for {LTI} systems with nonzero initial value},
	volume = {420},
	journal = {J. Comput. Appl. Math.},
	author = {Schröder, Christian and Voigt, Matthias},
	year = {2023},
	pages = {article no. 114708},
    zblnumber = {1504.93049},
    mrnumber = {4497919},
}

@book{obinata_model_2012,
	title = {{Model Reduction for Control System Design}},
	publisher = {Springer},
	author = {Obinata, Goro and Anderson, Brian D. O.},
    series    = {Comm. Control Engrg. Ser.},
    address   = {London},
	year = {2001},
}

@article{baur_model_2014,
	title = {Model order reduction for linear and nonlinear systems: a system--theoretic perspective},
	volume = {21},
	number = {4},
	journal = {Arch. Comput. Methods Eng.},
	author = {Baur, Ulrike and Benner, Peter and Feng, Lihong},
	year = {2014},
	pages = {331--358},
    mrnumber  = {3274477},
    zblnumber = {1348.93075},
}

@article{borghi_mathcal_2024,
	title = {{$\mathcal{H}_2$}-optimal rational approximation on general domains},
	volume = {50},
	number = {3},
	journal = {Adv. Comput. Math.},
	author = {Borghi, Alessandro and Breiten, Tobias},
	year = {2024},
	pages = {article no. 28},
    mrnumber = {4734558},
    zblnumber = {1541.41011},
}

@article{gugercin_mathcalh_2_2008,
	title = {{$\mathcal{H}_2$} model reduction for large--scale linear dynamical systems},
	volume = {30},
	number = {2},
	journal = {SIAM J. Matrix Anal. Appl.},
	author = {Gugercin, S. and Antoulas, A. C. and Beattie, C.},
	year = {2008},
	pages = {609--638},
    mrnumber  = {2421462},
    zblnumber = {1159.93318},
}

@article{gugercin_survey_2004,
	title = {A survey of model reduction by balanced truncation and some new results},
	volume = {77},
	number = {8},
	journal = {Internat. J. Control},
	author = {Gugercin, Serkan and Antoulas, Athanasios C.},
	year = {2004},
	pages = {748--766},
    zblnumber = {1061.93022},
    mrnumber = {2072207},
}

@incollection{beattie_ch7_2017,
author = {Christopher Beattie and Serkan Gugercin},
editor = {Peter Benner and Mario Ohlberger and Albert Cohen and Karen Willcox},
title = {Model Reduction by Rational Interpolation},
booktitle = {{Model Reduction and Approximation: Theory and Algorithms}},
series = {Comput. Sci. Eng.},
volume = {15},
chapter = {7},
pages = {297--334},
publisher = {SIAM},
address = {Philadelphia, PA},
doi = {10.1137/1.9781611974829.ch7},
year = {2017},
mrnumber = {3672151},
}

@book{Wahba_1990,
  author    = {Wahba, Grace},
  title     = {{Spline Models for Observational Data}},
  series    = {CBMS-NSF Regional Conf. Ser. in Appl. Math.},
  volume    = {59},
  publisher = {SIAM},
  address   = {Philadelphia, PA},
  year      = {1990},
  zblnumber = {0813.62001},
  mrnumber = {1045442},
}

@article{deBoor_divided_2005,
  title={Divided differences},
  author={de Boor, Carl},
  year={2005},
  journal = {Surv. Approx. Theory},
  volume = {1},
  pages = {46--69},
  zblnumber = {1071.65027},
  mrnumber = {2221566},
}

@book{Reddy_1993,
  author    = {Reddy, J. N.},
  title     = {{An Introduction to the Finite Element Method}},
  edition   = {2nd},
  series    = {McGraw-Hill Ser. Mech. Engrg.},
  publisher = {McGraw-Hill},
  address   = {New York},
  year      = {1993},
}

\end{document}